\documentclass{amsart}
\usepackage[
top=30truemm,
bottom=30truemm,
left=30truemm,
right=30truemm]{geometry}

\usepackage{amsfonts, amssymb, amsmath, verbatim, amsthm, mathrsfs, amscd, enumerate, ascmac, fancyhdr, caption, hyperref, framed,subfigure}
\usepackage{bbm}
\usepackage{eucal}
\usepackage{tikz}
\usepackage{tikz-cd}
\usetikzlibrary{shapes.geometric}
\usetikzlibrary{arrows.meta,arrows}
\usepackage{graphicx} 
\usepackage{pgf, tikz}
\usepackage{pgfplots}
\pgfplotsset{compat=1.18}
\usepackage{float}
\usepackage{pdfpages}
\usepackage{tcolorbox}
\usepackage{mathtools}
\usepackage{enumerate}
\mathtoolsset{showonlyrefs}

\usepackage[framemethod=tikz]{mdframed}
\mdfsetup{linecolor=blue,backgroundcolor=gray!10,roundcorner=10pt,leftmargin=2pt,innerleftmargin=2pt,innerrightmargin=2pt}
\usepackage{url}

\usepackage{xcolor}
\usepackage{tikz}
\usetikzlibrary{positioning,intersections, calc, arrows,decorations.markings}

\newtheorem{theorem}{Theorem}[section]
\newtheorem{lemma}[theorem]{Lemma}
\newtheorem{proposition}[theorem]{Proposition}
\newtheorem*{theorem*}{Theorem}

\newtheorem*{conjecture*}{Conjecture}

\newtheorem{claim}{Claim}
\newtheorem*{claim*}{Claim}
\theoremstyle{definition}
\newtheorem{definition}[theorem]{Definition}
\newtheorem{example}[theorem]{Example}

\newtheorem*{goal*}{Goal}

\numberwithin{equation}{section}
\numberwithin{theorem}{section}

\theoremstyle{remark}
\newtheorem{remark}[theorem]{Remark}

\usepackage{stackengine,scalerel}
\stackMath
\newcommand\reallywidehat[1]{%
\savestack{\tmpbox}{\stretchto{%
  \scaleto{%
    \scalerel*[\widthof{\ensuremath{#1}}]{\kern.1pt\mathchar"0362\kern.1pt}%
    {\rule{0ex}{\textheight}}
  }{\textheight}%
}{2.4ex}}%
\stackon[-6.9pt]{#1}{\tmpbox}%
}
\makeatletter
\DeclareRobustCommand\widecheck[1]{{\mathpalette\@widecheck{#1}}}
\def\@widecheck#1#2{%
    \setbox\z@\hbox{\m@th$#1#2$}%
    \setbox\tw@\hbox{\m@th$#1%
       \widehat{%
          \vrule\@width\z@\@height\ht\z@
          \vrule\@height\z@\@width\wd\z@}$}%
    \dp\tw@-\ht\z@
    \@tempdima\ht\z@ \advance\@tempdima2\ht\tw@ \divide\@tempdima\thr@@
    \setbox\tw@\hbox{%
       \raise\@tempdima\hbox{\scalebox{1}[-1]{\lower\@tempdima\box
\tw@}}}%
    {\ooalign{\box\tw@ \cr \box\z@}}}
\makeatother

\def\eqalign#1{\null\,\vcenter{\openup\jot\mathsurround\dimen12
  \ialign{\strut\hfil$\textstyle{##}$&$\textstyle{{}##}$\hfil
      \crcr#1\crcr}}\,}

\def\cS{\mathcal{S}}

\def\rd{\mathbb{R}^n}
\def\rdd{\mathbb{R}^{2n}}
\def\bR{\mathbb{R}}

\DeclareMathOperator*{\Sp}{Sp}
\DeclareMathOperator*{\Mp}{Mp}

\title[Pitt Inequalities and Log-type Uncertainty for metaplectic operators]{Pitt Inequalities and Logarithmic-type Uncertainty Principles for Metaplectic Operators}

\author[Giacchi]{Gianluca Giacchi}
\address[Gianluca Giacchi]{Università della Svizzera italiana, Faculty of Informatics, via la Santa 1, 6962 Lugano (Switzerland).}
\email{gianluca.giacchi@usi.it}
\author[Oliveira]{Itamar Oliveira}
\address[Itamar Oliveira]{School of Mathematics, The Watson Building, University of Birmingham, Edgbaston, Birmingham, B15 2TT, England.}
\email{i.oliveira@bham.ac.uk, oliveira.itamar.w@gmail.com}
\author[Tierney]{Amy Tierney}
\address[Amy Tierney]{School of Mathematics, The Watson Building, University of Birmingham, Edgbaston, Birmingham, B15 2TT, England.}
\email{axt944@student.bham.ac.uk}

\newcommand{\rmd}{\mathrm{d}}

\usepackage[dvipsnames]{xcolor}
\definecolor{scisRed}{RGB}{190, 90, 90}

\definecolor{aubergine}{HTML}{795367}

\definecolor{rose}{HTML}{AD718B}

\definecolor{lightgreen}{HTML}{A9C6B3}

\definecolor{green}{HTML}{587c64}

\definecolor{teal}{HTML}{075f6f}

\definecolor{blue}{HTML}{4670a8}

\begin{document}


\begin{abstract}
In this work, we provide a complete characterisation of Pitt's inequality for metaplectic operators. The symplectic geometry underlying a metaplectic operator distinguishes effective directions, along which its action is Fourier-type, from singular directions, along which concentration is preserved. For this reason, we adopt two complementary perspectives, thereby obtaining both an isotropic Pitt's inequality, emulating the classical theorem of Pitt, and an anisotropic inequality that adapts to the geometric features of the metaplectic group.
As a by-product, we obtain Pitt's inequality for Fourier transforms along subspaces of $\rd$ and new quantitative time-dependent boundedness results for metaplectic operators on homogeneous Sobolev spaces, with applications to Schr\"odinger evolutions generated by quadratic Hamiltonians. 

Additionally, we derive logarithmic and entropic uncertainty principles for the metaplectic group.
We show that the classical entropic uncertainty principle fails when both the effective and singular directions are present. In this case, a generally unbounded corrective entropy term accounts for the concentration-preserving directions and restores a meaningful lower bound. 
For quadratic Schr\"odinger evolutions, this correction follows the time-dependent geometry of the nondispersive directions and adjusts the uncertainty estimate accordingly.
\end{abstract}

\maketitle

\tableofcontents

\section{Introduction} 
\subsection{Two Fourier inequalities and two uncertainty principles} 
\label{subsec:two_ineq_UPs}
    Entropic and logarithmic uncertainty principles for the Fourier transform $\mathcal{F}:f\in \mathcal{S}(\rd)\mapsto \widehat f\in \mathcal S(\rd)$, given by
\begin{equation}\label{intro.defFT}
    \widehat f(\xi)=\int_{\rd}f(x)e^{-2\pi i\xi\cdot x}\rmd x, \qquad \xi\in\rd,
\end{equation}
were studied by W. Beckner in two influential papers \cite{Beck75,Beckner_logUP}. 
A common feature of the entropic and logarithmic uncertainty principles is the simple yet elegant differentiation argument that
reduces each of these to a consequence of a sharp form of some other significant result: Hausdorff-Young and Pitt's inequality, respectively. 

Motivated by applications to partial differential equations, in this work we study extensions of classical Fourier-analytic inequalities to the metaplectic setting, where metaplectic operators arise as propagators for Schrödinger equations with real quadratic Hamiltonians. 
Building on the complete characterisation of the Hausdorff–Young inequality for metaplectic operators established in \cite{Giacchi}, we complete Beckner's story in this framework by proving Pitt's inequality, as well as entropic and logarithmic uncertainty principles, for the metaplectic group.

\vspace{3mm}
\noindent
\textbf{From Hausdorff-Young to entropy uncertainty.}
In his influential 1975 paper \cite{Beck75}, W. Beckner established the sharp form of the Hausdorff-Young inequality for the Fourier transform.

\begin{theorem*}[Sharp Hausdorff-Young inequality \cite{Beck75}] For $1\leq p\leq 2$, $p'=p/(p-1)$ and $f\in L^{p}(\mathbb{R}^n)$,
\begin{equation}\label{Becknerthm75}
    \|\widehat{f}\|_{p'}\leq \left(\frac{p^{1/p}}{(p')^{1/p'}}\right)^{n/2}\|f\|_{p}.
\end{equation}
The constant in \eqref{Becknerthm75} is sharp.
\end{theorem*}
We refer the reader to \cite{Christ} for a more detailed account of previous results and refinements of this theorem. A well known corollary of Beckner's theorem is the \textit{entropic uncertainty principle}:

\begin{theorem*}[Entropic uncertainty principle] For $f\in\mathcal S(\bR^n)$ with $\|f\|_2=1$,
\begin{equation}\label{entropicUP-210526}
    -\int_{\mathbb{R}^{n}}\ln{|f(x)|}|f(x)|^{2}\mathrm{d}x -\int_{\mathbb{R}^{n}}\ln{|\widehat{f}(\xi)|}|\widehat{f}(\xi)|^{2}\mathrm{d}\xi\geq\frac{n}{2}(1-\ln{2}).
\end{equation}
\end{theorem*}
Throughout this work, we shall denote the entropy of $|f|^2$ by
\begin{equation}
    H[|f|^2]:=-\int_{\mathbb{R}^{n}}\ln{|f(x)|}|f(x)|^{2}\mathrm{d}x.
\end{equation}
{
For comparison with the usual probabilistic convention, let
\begin{equation}\label{intro.defh}
     h(\rho)=-\int\rho(x)\ln\rho(x)\rmd x
\end{equation}
denote the {\em differential entropy} of a probability density $\rho$. With our convention, $H[|f|^2]=\frac12h(|f|^2)$.
}

To see that the entropic uncertainty principle indeed follows from the Hausdorff-Young inequality, fix $1\leq p\leq 2$ and $f\in\mathcal{S}(\mathbb{R}^n)$ and define 
$$F(p):=\left(\int|\widehat{f}(\xi)|^{p'}\mathrm{d}\xi\right)^{\frac{1}{p'}}-\left(\frac{p^{1/p}}{(p')^{1/p'}}\right)^{n/2}\left(\int|f(x)|^{p}\mathrm{d}x\right)^{\frac{1}{p}}.$$

The function $F$ is differentiable and satisfies $F(p)\leq 0$ (due to Hausdorff-Young) and $F(2)=0$ (due to Plancherel). These two properties imply $F'(2)\geq 0$, which yields \eqref{entropicUP-210526} after a careful computation. 

\vspace{3mm}
\noindent
\textbf{From Pitt's inequality to logarithmic uncertainty.}
Pitt's inequality, named after Pitt for his work in the setting of Fourier series \cite{Pitt}, also has a more general form (see for example \cite{Pitt_general_p_FT} and the references therein for its origins) given as follows
\begin{theorem}[Pitt's inequality]\label{thm:PittpqFT}
For $f\in \mathcal{S}(\mathbb{R}^{n})$, the inequality
\begin{equation}
\label{eqn:pitt2}
\left(\int_{\mathbb{R}^{n}}|\xi|^{- \beta q}|\widehat{f}(\xi)|^{q}\rmd \xi \right)^{1/q}\leq K(p,q,\alpha,\beta,n) \left(\int_{\mathbb{R}^{n}}|x|^{\alpha p}|f(x)|^{p}\rmd x\right)^{1/p},
\end{equation}
holds for some $K(p,q,\alpha,\beta,n)>0$ if and only if one of the following cases holds:
\begin{equation}\label{assumptionPitt}
\begin{split}
(a) \qquad
    &1 < p \leq q < \infty, \qquad 0 \leq \beta < \frac{n}{q}, \qquad 0\leq\alpha < \frac{n}{p'}, \qquad \text{and} \quad \beta = \alpha + n \left(\frac{1}{q}- \frac{1}{p'}\right);\\
   (b) \qquad &p=1, \qquad q=\infty, \qquad\text{and}\qquad \alpha=\beta=0.
    \end{split}
\end{equation}
\end{theorem}
The full sharp range of exponents, including the endpoint case $(b)$ is discussed in \cite{SaucedoTikhonov2025}.
The sharp constant in \eqref{eqn:pitt2} is not known in general. When $\alpha=\beta=0$, the problem reduces to the sharp Hausdorff-Young inequality. In the case $p=q=2$, $0\leq\alpha<n/2$ and $\alpha=\beta$, the sharp constant was obtained by Beckner in \cite{Beckner_logUP}.
\begin{theorem}[Sharp Pitt's Inequality \cite{Beckner_logUP}]
\label{sharp_Pitt}
For $f\in \mathcal{S}(\mathbb{R}^{n})$ and $0 \leq \alpha < n/2$,
\begin{equation}
\label{eqn:pitt}
\int_{\mathbb{R}^{n}}|\xi|^{- 2\alpha}|\widehat{f}(\xi)|^{2}\rmd \xi \leq K_{n,\alpha}\int_{\mathbb{R}^{n}}|x|^{2\alpha}|f(x)|^{2}\rmd x,
\end{equation}
where
\begin{equation}
\label{constant:sharp_Pitt}
    K_{n,\alpha} = \pi^{2\alpha}\left[\frac{\Gamma\left(\frac{n- 2\alpha}{4} \right)}{\Gamma\left(\frac{n+ 2\alpha}{4} \right)}
    \right]^2 \text{is sharp}.
\end{equation}
\end{theorem}

Where the entropic uncertainty principle was derived from sharp Hausdorff-Young, Beckner used the same argument to prove a logarithmic estimate of uncertainty from the sharp Pitt's inequality.
\begin{theorem*}[Logarithmic uncertainty principle \cite{Beckner_logUP}] 
\label{log_UP}
For $f\in \mathcal{S}(\mathbb{R}^{n})$,
\begin{equation}\label{eqn:logUP}
\int_{\mathbb{R}^{n}}\ln{|x|}|f(x)|^{2}\rmd x + \int_{\mathbb{R}^{n}}\ln{|\xi|}|\widehat{f}(\xi)|^{2}\rmd \xi\geq \tilde{K}_n \int_{\mathbb{R}^n} |f(x)|^2 \rmd x,
\end{equation}
where 
\begin{equation}
\label{constant:logUP}
   \tilde{K}_n = \psi(n/4) - \ln \pi, \qquad \psi(t) = \frac{\rmd}{\rmd t} (\ln \Gamma(t)),
\end{equation}
and $\Gamma(t)=\int_{0}^{\infty}e^{-x}x^{t-1}\mathrm{d}x$ is the Gamma function.
\end{theorem*}
This follows again by Beckner's differentiation argument, now differentiating in $\alpha$ instead of $p$ and evaluating at $\alpha=0$ since when $\alpha =0$, sharp Pitt's inequality becomes Plancherel's identity. 

    \subsection{Metaplectic operators}
    Consider the Cauchy problem
    \begin{equation}\label{Schro}
\begin{cases}
i\frac{1}{2\pi}\partial_t u(t,x)=a^{\mathrm{w}}(x,\mathrm{D})u(t,x),
& t\in\bR,\; x\in\rd,\\
u(0,\cdot)=u_0\in\mathcal{S}(\rd).
\end{cases}
\end{equation}
Assume that the Hamiltonian operator $a^{\mathrm{w}}(x,\mathrm{D})$ is the Weyl quantization of a real-valued quadratic form on phase space. This means that
\begin{equation}\label{Weyl}
    a^{\mathrm{w}}(x,\mathrm{D})f(x)=\int_{\bR^{2n}}f(y)a\left(\frac{x+y}{2},\xi\right)e^{2\pi i\xi\cdot(x-y)}\rmd y\rmd\xi, \qquad f\in\mathcal{S}(\rd),
\end{equation}
where $a(z)=\frac 1 2Az\cdot z$ for some $A\in\mathrm{Sym}(2n,\bR)$.
The corresponding Hamiltonian flow defines a one-parameter subgroup $\{S_t\}_{t\in\bR}$ of the group of $2n\times 2n$ symplectic matrices $\Sp(n,\mathbb{R})$, whereas the propagator of \eqref{Schro}
\begin{equation}\label{propagatorSchro}
u(t,x)=e^{-2\pi i ta^{\mathrm{w}}(x,\mathrm{D})}u_0(x), \qquad t\in\bR, \; x\in\rd,
\end{equation}
is a one-parameter subgroup of unitary operators on $L^2(\rd)$ contained in the {\em metaplectic group} $\Mp(n,\bR)$, the two-fold cover of $\Sp(n,\bR)$. We refer to Section \ref{sec:prelim} below for the precise definitions. To every
$\widehat S\in\Mp(n,\bR)$ there corresponds a unique symplectic matrix, referred to as its {\em projection},
\begin{equation}\label{intro.blockS}
    \begin{array}{ccc}
   \widehat S\in\Mp(n,\bR) & \longrightarrow & S=\begin{pmatrix}
       A & B\\
       C & D
       \end{pmatrix}\in\Sp(n,\bR), \; A,B,C,D\in\bR^{n\times n}.
     \end{array}
\end{equation}
This correspondence makes it possible to adopt a linear-algebraic perspective on operator calculus, relating analytic properties of metaplectic operators to the geometric structure of their symplectic projections. 
Under this perspective, the classical Hamiltonian flow $S_t$ is the projection of the propagator $e^{-2\pi i ta^{\mathrm{w}}(x,\mathrm{D})}$, namely $\widehat S_t=e^{-2\pi i ta^{\mathrm{w}}(x,\mathrm{D})}$. 
Metaplectic techniques can therefore be used to derive results for Schr\"odinger evolutions generated by quadratic Hamiltonians. Moreover, since the Fourier transform is itself a metaplectic operator, classical results such as the Hausdorff-Young inequality, Pitt's inequality, and several uncertainty principles can be interpreted as model cases of estimates for the metaplectic representation. This observation provides the guiding principle of the present work:
\[
    \begin{array}{c}
       \text{Results for the Fourier transform}\\
       \Downarrow\\
       \text{Results for operators in $\Mp(n,\mathbb{R})$}\\
       \Downarrow\\
       \text{Results for propagators of Schr\"odinger problems \eqref{Schro}.}
    \end{array}
\]
This perspective is supported by the structure of metaplectic operators. Heuristically speaking, after suitable linear changes of variables and multiplication by quadratic phase factors, a metaplectic operator acts through a Fourier-type transformation along the so-called {\em effective directions}, i.e., along $\ker(B)^\perp$, while exhibiting an identity-type, nondispersive behaviour along the so-called {\em singular directions}, i.e., along $D^\top A(\ker(B))$. We remark that $\rd=\ker(B)^\perp\oplus D^\top A(\ker(B))$, so effective and singular directions span all $\rd$, see also Figure~\ref{Figure1}.

\begin{figure}
\begin{center}
\begin{tikzpicture}[>=Stealth, thick]

  \begin{scope}[shift={(-3.7,0)}]

    \draw[->] (-2.2,0) -- (2.2,0) node[right] {$x_1$};
    \draw[->] (0,-2.2) -- (0,2.2) node[above] {$x_2$};

    \pgfmathsetmacro{\th}{36}
    \draw[line width=1pt]
        ({-2*cos(\th)},{-2*sin(\th)}) -- ({2*cos(\th)},{2*sin(\th)});

    \pgfmathsetmacro{\tho}{144}
    \draw[line width=1pt, dashed]
        ({-2*cos(\tho)},{-2*sin(\tho)}) -- ({2*cos(\tho)},{2*sin(\tho)});

    \node[above right, font=\small]
        at ({2*cos(\th)},{2*sin(\th)})
        {$\ker(B)^{\perp}$};

    \node[above left, font=\small]
        at ({2*cos(\tho)-0.05},{2*sin(\tho)+0.10})
        {$D^\top A(\ker(B))$};

    \node[below, font=\small] at (0,-2.7)
        {Time-domain (domain of $f$)};

  \end{scope}

  \draw[-{Stealth[length=8pt]}, line width=1pt, bend left=20]
      (-0.9,0.6) to node[above] {$\widehat S$} (0.9,0.6);

  \begin{scope}[shift={(3.7,0)}]

    \draw[->] (-2.2,0) -- (2.2,0) node[right] {$\xi_1$};
    \draw[->] (0,-2.2) -- (0,2.2) node[above] {$\xi_2$};

    \pgfmathsetmacro{\ra}{50}
    \draw[line width=1pt]
        ({-2*cos(\ra)},{-2*sin(\ra)}) -- ({2*cos(\ra)},{2*sin(\ra)});

    \pgfmathsetmacro{\rb}{110}
    \draw[line width=1pt, dashed]
        ({-2*cos(\rb)},{-2*sin(\rb)}) -- ({2*cos(\rb)},{2*sin(\rb)});

    \node[above right, font=\small]
        at ({2*cos(\ra)+0.05},{2*sin(\ra)+0.05})
        {$R(B)$};

    \node[above left, font=\small]
        at ({2*cos(\rb)},{2*sin(\rb)})
        {$A(\ker(B))$};

    \node[below, font=\small] at (0,-2.7)
        {Frequency-domain (domain of $\widehat Sf$)};

  \end{scope}

\end{tikzpicture}
\end{center}
\caption{
    The \textit{effective directions} span $\ker(B)^\perp$. 
    Along these directions, $\widehat S$ behaves like a
    modified Fourier transform.
    Similarly, the \textit{singular directions} are associated with
    $D^\top A(\ker(B))$ in the time domain and with $A(\ker(B))$ in the
    frequency domain. Along these directions, $\widehat S$ behaves as a
    modified identity.
    In synthesis, effective and singular directions are naturally paired
    with corresponding directions in the frequency domain:
    $v\in\ker(B)^\perp\mapsto Bv\in R(B)$ and
    $D^\top Av\in D^\top A(\ker(B))\mapsto Av\in A(\ker(B))$.}
    \label{Figure1}
\end{figure}
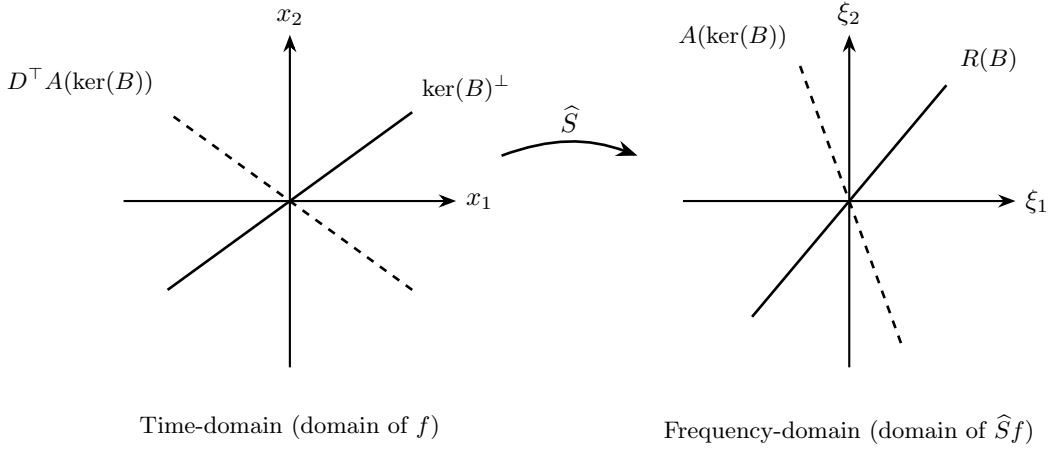

\medskip
\noindent
The general strategy is therefore to separate these two regimes, analyse their contributions independently, and finally reassemble them in order to obtain explicit results expressed in terms of the blocks of the corresponding symplectic projection. 

This technique has been used in \cite{CGM2025,cordero2026anisotropic,Giacchi} to investigate Hausdorff-Young's inequality and anisotropic uncertainty principles for metaplectic operators. An alternative approach was developed in \cite{grochenig2025more} using metaplectic Wigner distributions \cite{CorderoRodino2022}, whereas Heisenberg's uncertainty principle was proven by Dias, de Gosson and Prata \cite{CDP2024}.
Relevant to the entropic uncertainty principles of this paper is \cite{Giacchi}, wherein the distinction between the behaviour along singular and effective directions means that metaplectic operators satisfy Hausdorff-Young's inequality if and only if $B\in\mathrm{GL}(n,\bR)$.

\subsection{Present contributions}
As we have begun to allude to, the main results of this paper are:
\begin{itemize}
    \item Entropic uncertainty principles for the metaplectic group,
    \item A complete characterisation of the {radial} Pitt's inequality,
    \item {A directional form of Pitt's inequality, with sharp constants when $p=q=2$,}
    \item Logarithmic uncertainty principles for the metaplectic group,
    \item Metaplectic estimates on homogeneous Sobolev spaces.
\end{itemize}
Proving the first two consumes the majority of the work of this paper, with the final two being their useful consequences that we wish to highlight. We now present and interpret each in turn. 
Throughout the rest of the introduction, we assume $\widehat S\in\Mp(n,\bR)$ and that its projection $S$ has blocks \eqref{intro.blockS}. Moreover, we set $r = \mathrm{rank}(B)$.

\medskip
\noindent
   \textbf{The entropic uncertainty principle.}
     As aforementioned, it is proved in \cite{Giacchi} that if the block $B\neq O$ is singular, metaplectic operators do not satisfy Hausdorff-Young's inequality. Due to Beckner's connection between sharp Hausdorff-Young and entropic uncertainty principles, one might wonder whether similar results hold for the entropic uncertainty principles for metaplectic operators. Interestingly, the inequality in \eqref{entropicUP-210526} also fails if $B\notin\mathrm{GL}(n,\bR)$.
 \begin{proposition}[Entropic uncertainty principle for metaplectic operators]\label{thm:main1}
    The following statements hold for every $f\in \mathcal{S}(\rd)$ with $\Vert f\Vert_2=1$.
    \begin{enumerate}[(i)]
        \item If $B=O$, then the action of $\widehat S$ preserves the entropy of $f$ up to the natural Jacobian correction. Explicitly,
        \begin{equation}\label{entrop1-270526}
H[|\widehat Sf|^2]=\frac 1 2 \ln|\det(A)|+H[|f|^2].
\end{equation}
        \item If $B$ is invertible, then $\widehat S$ satisfies an entropic uncertainty principle in its standard form. Indeed,
       \begin{align}\label{entrop2-270526}
H[|f|^2] + H[|\widehat Sf|^2] &\geq \frac{n}{2}(1-\ln(2|\det(B)|^{-1/n})).
\end{align}
        \item In the remaining case, the operator $\widehat S$ does not satisfy an entropic uncertainty principle in its standard form. In particular,
        there exists no constant $K\in\bR$ such that
    \begin{equation}\label{entrop3-270526}
        H[|f|^2]+H[|\widehat Sf|^2]\geq K.
    \end{equation}
    \end{enumerate}
\end{proposition}
When the block $B=O$, the entropy is preserved along \textit{every direction}. When $B$ is invertible, there is \textit{no direction} along which the entropy is preserved.
The entropic uncertainty principle \textit{fails} in the remaining case due to the presence of directions where the entropy \textit{is preserved}. Indeed, because of the logarithm factor defining $H$, the singular directions contribute to the entropic inequality, and their contribution opposes the entropy variation along the effective directions.

By making use of the different behaviour of the effective and singular directions, we may prove a variant of the entropic uncertainty principle. 
To do so, we first define a corrective term. 
In the following, we denote by
$B^+$ the Moore-Penrose inverse of $B$ and by $P=I-B^+B$ the orthogonal projection onto $\ker(B)$, see Section \ref{sec:prelim} for the precise notation.
\begin{definition}\label{intro.def.corrective.term}
    Let $f\in\mathcal S(\rd)$ have $\Vert f\Vert_2=1$.
    Consider the marginal
    \begin{equation}\label{intro.defmKB}
        m_{\ker(B)}(f)(y):=\int_{\ker(B)^\perp}|f(x_1+y)|^2\rmd x_1, \qquad y\in\ker(B),
    \end{equation} 
    We define
    \begin{equation}\label{intro.AddTerm}
        H_S[f]:=h\big(m_{\ker(B)}(f)\big)=-\int_{\ker(B)}m_{\ker(B)}(f)(y)\ln\big(m_{\ker(B)}(f)(y)\big)\rmd y.
    \end{equation}
\end{definition}
\begin{remark}\label{intro.RemmK}
    {Since $f\in\cS(\rd)$ has $\|f\|_2=1$, the marginal $m_{\ker(B)}(f)$ is a probability density on $\ker(B)$ with respect to the Euclidean Lebesgue measure, and $H_S[f]$ is its differential entropy.
    In particular, $H_S[f]$ depends on $S$ only through the space $\ker(B)$. We remark that if $\ker(B)=\{0\}$, then $m_{\ker(B)}$ is the unit mass on a zero-dimensional space and $H_S[f]=0$.
    }
\end{remark}

    The entropic uncertainty principle for metaplectic operators reads as follows.
    \begin{theorem}\label{thm:main2}
    Assume that $1\leq r=\mathrm{rank}(B)\leq n$. 
    Then, for every $f\in \mathcal S(\rd)$, with $\|f\|_2=1$,
       \begin{equation}\label{entropyUPS}\begin{split}
    	H[|f|^2]+H[|\widehat Sf|^2]-H_S[f]\geq\frac{r}{2}\big(1-\ln(2 |\det(B+AP)|^{-1/r} )\big).
    \end{split}\end{equation}
\end{theorem}

\begin{example}\label{exampleConcrete}
    {For the sake of concreteness, we discuss the metaplectic operator in dimension $n=2$, given by
    \begin{equation}
        \widehat{\mathcal{A}_w}F(x,\xi)=\int_{\bR}F\left(x+\frac y2,x-\frac y2\right)e^{-2\pi i\xi y}\rmd y, \qquad F\in\cS(\bR^2),
    \end{equation}
    whose projection is the $4\times 4$ symplectic matrix
    \begin{equation}\label{intro.defAw}
        \mathcal A_w=\left(\begin{array}{cc|cc}
            1/2 & 1/2 & 0 & 0\\
            0 & 0 & 1/2 & -1/2\\
            \hline
            0 & 0 & 1 & 1\\
            -1 & 1 & 0 & 0
        \end{array}\right).
    \end{equation}
    This operator is pivotal in time-frequency and quantum harmonic analysis, for it defines the {\em Wigner distribution}
    \begin{equation}
        Wf(x,\xi)=\widehat{\mathcal A_w}(f\otimes \bar f), \qquad f\in L^2(\bR).
    \end{equation}
    Let
    \[
        e_+=\frac{1}{\sqrt2}(1,1),\qquad e_-=\frac{1}{\sqrt2}(1,-1).
    \]
    For the upper-right block $B$ of $\mathcal A_w$,
    \[
    \ker(B)=\bR e_+,
    \qquad
    \ker(B)^\perp=\bR e_-.
    \]
    Thus Definition~\ref{intro.def.corrective.term},
    written in these orthonormal coordinates, gives
    \[
    \rho(s):=m_{\ker(B)}(F)(se_+)
    =\int_{\bR}
    \left|
    F\left(\frac{s+t}{\sqrt2},\frac{s-t}{\sqrt2}\right)
    \right|^2\rmd t,
    \]
    where we are now assuming $\|F\|_2=1$, and
    \begin{equation}\label{HSexample}
    H_{\mathcal A_w}[F]=h(\rho)=-\int_{\bR}\rho(s)\ln\rho(s)\rmd s.
    \end{equation}
    Since $\|F\|_2=1$,
    \begin{equation}
        \int_{\bR}\rho(s)\rmd s=1.
    \end{equation}
    Thus, $\rho$ is a marginal probability density describing how the $L^2$-mass of $F$ is distributed across the family of parallel lines $x+\xi=\sqrt{2}s$, regardless of how the mass is distributed along each individual line.
    Its differential entropy $H_{\mathcal A_w}[F]$ quantifies the effective spread of this mass distribution across the different values of $s$. 
    To understand the role of this term in the entropic uncertainty principle, Plancherel's identity gives
    \[
        \int_{\bR}|\widehat{\mathcal A_w}F(x,\xi)|^2\rmd\xi=\sqrt{2}\rho(\sqrt{2}x).
    \]
    Thus, $H_{\mathcal A_w}[F]$ is the differential entropy of the $x$-marginal of $|\widehat{\mathcal A_w}F|^2$, up to the additive constant $\frac12\ln 2$ arising from the change of coordinates, and its subtraction in \eqref{entropyUPS} compensates for the contribution of this non-dispersive direction.}

    {
    A particularly interesting consequence arises when $F=f\otimes\bar f$, with $f\in\mathcal S(\bR)$ and $\|f\|_2=1$. In this case, the entropic uncertainty principle provides quantitative information on the time-frequency concentration of $f$ through its Wigner distribution.
    Indeed, setting $p_f:=|f|^2$, we have
    \begin{equation}\label{intro.examplerho}
        \rho(s)=\sqrt{2}(p_f * p_f)(\sqrt{2}s),
    \end{equation}
    and consequently
    \begin{equation}\label{WigCorrEntropy}
    \begin{split}
        H_{\mathcal A_w}[f\otimes\bar f]=-\int_{\bR}(p_f * p_f)(x)\ln (p_f * p_f)(x)\rmd x-\frac12\ln 2.
    \end{split}
    \end{equation}
    From a probabilistic perspective, let $X$ and $Y$ be independent random variables with common density $p_f$. Because of their independence, their joint density will be
    \[ 
        \nu(x,y)=p_f(x)p_f(y)=|f(x)\overline{f(y)}|^2,
    \]
    while $\rho$, defined as in \eqref{intro.examplerho}, is the probability density of $(X+Y)/\sqrt{2}$. Thus, the corrective term quantifies the entropic dispersion of the normalised sum of two independent positions distributed according to $|f|^2$.
    Moreover, Plancherel's identity gives
    \begin{equation}
        \int_{\bR}|Wf(x,\xi)|^2\rmd\xi=2(p_f * p_f)(2x),
    \end{equation}
    showing that the same convolution determines the spatial marginal of the nonnegative phase-space density $|Wf|^2$.
    Since
    \begin{equation}
        H[|f\otimes\bar f|^2]=2H[|f|^2]=-\int_{\bR}p_f(x)\ln p_f(x)\rmd x,
    \end{equation}
    Theorem \ref{thm:main2}, together with $|\det(B+AP)|=1/2$ (where $A$ and $B$ are the top left and right blocks in \eqref{intro.defAw}, respectively) yields
    \begin{equation}\label{WigEntEst}
    \begin{split}
        H[|Wf|^2]\geq\frac12-\frac32\ln 2-\int_{\bR}(p_f * p_f)(x)\ln(p_f * p_f)(x)\rmd x+\int_{\bR}p_f(x)\ln p_f(x)\rmd x.
    \end{split}
    \end{equation}
    Furthermore, the Shannon-Stam inequality \cite{Stam1959} implies
    \begin{equation}
        \begin{split}
            -\int_{\bR}(p_f * p_f)(x)\ln(p_f * p_f)(x)\rmd x+\int_{\bR}p_f(x)\ln p_f(x)\rmd x\geq\frac12\ln 2.
        \end{split}
    \end{equation}
    Hence, we retrieve the classical inequality due to Lieb \cite{Lieb},
    \begin{equation}\label{WigEntSharp}
        H[|Wf|^2]\geq\frac12-\ln 2,
    \end{equation}
    see also \cite{DDP2019}.
    We remark that this bound is sharp, with equality attained by the normalised Gaussian $f(x)=2^{1/4}e^{-\pi x^2}$, whose Wigner distribution is $Wf(x,\xi)=2e^{-2\pi(x^2+\xi^2)}$. 
    Moreover, we stress that \eqref{WigEntSharp} is invariant under metaplectic transformations, because of \textit{metaplectic covariance}, that is
    \begin{equation}
        W(\widehat Sf)(z)=Wf(S^{-1}z), \qquad \widehat S\in \Mp(1,\bR),\; f\in L^2(\bR),\; z\in\bR^2.
    \end{equation}
    This highlights an interesting distinction between the entropic uncertainty principle obtained in this work and \eqref{WigEntSharp}. 
    Even though the time-frequency setting is particularly well suited to the study of metaplectic operators, the resulting entropic inequality does not depend on the underlying symplectic matrix, due to metaplectic covariance. 
    Hence, \eqref{WigEntSharp} does not capture the dependence on the geometry of the symplectic blocks, which is instead explicitly reflected in our general result \eqref{entropyUPS} through the corrective entropic term.}
    
\end{example}

\medskip
\noindent
{\bf A sharp directional Pitt's inequality and a logarithmic uncertainty principle.}
While uncertainty principles in the literature do typically exhibit a directional behaviour, the presence of an additional term mitigating the effect of the singular directions is what makes the entropic uncertainty principle unique. In contrast with the entropic uncertainty principle, Pitt's inequality does not fail for metaplectic operators with $1\leq \mathrm{rank}(B)<n$.

Motivated by Theorem~\ref{sharp_Pitt}, we prove a sharp form of Pitt's inequality for metaplectic operators, relying on the same directional structure we witnessed above. We wish to highlight that the inequality \eqref{PittwithmatricesBinv} below is indeed adapted to the partial dispersivity of $\widehat S$: it exhibits the weighted behaviour along the effective directions, while neglecting the complementary singular directions. 

    \begin{proposition}[Directional Pitt's inequality for $p=q=2$]\label{intro.thm:Pittpq2}
     Assume $r>0$, then for every $f\in\mathcal{S}(\mathbb{R}^n)$ and $0\leq\alpha< r/2$,
    \begin{equation}\label{PittwithmatricesBinv}
            \int_{\rd}|B^+B(B+AP)^{-1}\xi|^{-2\alpha}|\widehat Sf(\xi)|^2\rmd \xi\leq K_{r,\alpha}\int_{\rd}|(I-D^\top AP)x|^{2\alpha} |f(x)|^2\rmd x.
        \end{equation}
        The constant $K_{r,\alpha}$ defined as in \eqref{constant:sharp_Pitt} is sharp.
    \end{proposition}
The result above is an instance of a more general directional Pitt's inequality for metaplectic operators, for which we identify the full admissible range of parameters, see Proposition \ref{Cor-Pitt-03082026} below.

From the statement presented above, a logarithmic uncertainty principle naturally occurs. Indeed, in the case of $\alpha = \beta=0$, \eqref{PittwithmatricesBinv} holds with equality since metaplectic operators are unitary. Therefore, applying Beckner's differentiation argument discussed in Section \ref{subsec:two_ineq_UPs}, the following consequence is derived
    \begin{proposition}[Logarithmic uncertainty principle for metaplectic operators]\label{intro.logUP}
    Assume $r>0$. Then, for every $f\in\mathcal{S}(\mathbb{R}^n)$,
    \begin{equation}\label{LogUPforS}
        \int_{\mathbb{R}^n}\log|B^+B(B+AP)^{-1}\xi||\widehat Sf(\xi)|^2\mathrm{d}\xi + \int_{\rd}\log|(I-D^\top AP)x||f(x)|^2\mathrm{d}x \geq \tilde{K}_{r}\|f\|_2^2,
    \end{equation}
    where $\tilde{K}_{r}$ is defined by \eqref{constant:logUP}.
    \end{proposition}
    {In this work, we consider the subcase $r=n$ of $0<r\leq n$ separately, for the sake of clarity. 
    The corresponding results are present in the literature, see for example \cite[Theorems 6, 7 and 8]{Jing2020NDimensional}, and are recovered here as consequences of Propositions~\ref{intro.thm:Pittpq2}, \ref{intro.logUP} and~\ref{thm:main1}(ii). We remark that the case $0<r<n$ requires a significantly subtler analysis and it is one of the main contributions of this manuscript.}
    
\medskip
\noindent
{\bf A full classification of radial Pitt's inequality.}
It is possible to `decouple' the effective from the singular directions, yielding a form more reminiscent of the structure of Pitt's inequality for the Fourier transform, but further from the geometry of metaplectic operators. 

It is in this radial form, devoid of the directional dependence, that we present our main result: a full classification of Pitt's inequality in the form
\begin{equation}\label{Pittgeneral}
            \left(\int_{\rd}|\xi|^{-\beta q}|\widehat Sf(\xi)|^q\rmd\xi\right)^{1/q}\leq K_S(p,q,\alpha,\beta,r,n)
        \left(\int_{\rd}|x|^{\alpha p}|f(x)|^p\rmd x\right)^{1/p},
        \end{equation}
    for $1\leq p,q\leq\infty$, $\alpha,\beta\geq0$.
Anticipating Theorem \ref{intro.thm:mainmetap} below, the class of parameters $(p,q,\alpha,\beta,r,n)$ for which \eqref{Pittgeneral} holds is the following.

\begin{definition}[Pitt-admissible datum]\label{intro.defPittAdmissible}
    Let $0\leq r\leq n$, $1\leq p,q\leq\infty$ and $\alpha,\beta\geq0$. We say that the datum $(p,q,\alpha,\beta,r,n)$ is {\em Pitt-admissible} if it satisfies one of (i)-(iv) below:
    \begin{enumerate}[(i)]
        \item \label{Pitt_admiss_i} $r=0$, $\alpha=\beta=0$ and $1\leq p=q\leq\infty$.
        \item \label{Pitt_admiss_ii} $r=n$, 
        \begin{equation}
        \begin{split}
    &\textnormal{(a)}\qquad 1 < p \leq q < \infty, \qquad 0 \leq \beta < \frac{n}{q}, \qquad 0\leq\alpha < \frac{n}{p'}, \qquad \text{and} \quad \beta = \alpha + n \left(\frac{1}{q}- \frac{1}{p'}\right),\\
    &\textnormal{(b)}\qquad p=1,\qquad q=\infty,\qquad \text{and} \qquad \alpha=\beta=0.
    \end{split}
\end{equation}
\item \label{Pitt_admiss_iii} $1\leq r<n$, and
\begin{equation}\label{NecCondSfglob}
                 1<p=q<\infty, \qquad  \alpha-\beta=r\left(1-\frac 2 p\right),\qquad 0\leq\alpha+\beta<r.
            \end{equation}
            \item \label{Pitt_admiss_iv} $1\leq r<n$, and $1\leq q< p\leq\infty$, and
    \begin{align}
        \left|\alpha-\beta-r\left(1-\frac 1p-\frac 1q\right)\right|&<(n-r)\left(\frac 1q-\frac1p\right),\\
        \alpha+\beta&>(n-r)\left(\frac 1q-\frac 1p\right)+r\max\left\{\frac 1q-\frac 1p,\left|1-\frac 1p-\frac1q\right|\right\},\\
        \alpha+\beta&<r+n\left(\frac 1q-\frac 1p\right).
    \end{align}
\end{enumerate}
\end{definition}
We defer to Section \ref{sec:PittForPartial}, in particular to Figures \ref{fig:PFT-necessary-1} and \ref{fig:PFT-necessary-2}, for a geometric discussion of what the system of inequalities of item \eqref{Pitt_admiss_iv} encapsulates. We also refer to Remarks \ref{rem:constants}, \ref{rem:Step344} and \ref{remarkmatrices} for the exhaustive discussions about the constant $K_S$ in \eqref{Pittgeneral}.

\begin{theorem}[Radial Pitt's inequality for metaplectic operators]\label{intro.thm:mainmetap}
Pitt's inequality \eqref{Pittgeneral} holds if and only if $(p,q,\alpha,\beta,r,n)$ is Pitt-admissible. Moreover, if $B=O$, then
    \begin{align*}
                    \left(\int_{\mathbb{R}^{n}}|x|^{\gamma p}|\widehat{S}f(x)|^{p} \rmd x\right)^{1/p}
                        &\leq |\det(A)|^{\frac 1p-\frac 12} \max\{\sigma_{\max}(A)^\gamma ,\sigma_{\min}(A)^\gamma \}\left(\int_{\mathbb{R}^{n}}|x|^{\gamma p}|f(x)|^{p} \rmd x\right)^{1/p}
                \end{align*}
                for every $\gamma\in\bR$ and $0<p<\infty$. 
\end{theorem}

{
\begin{remark}
\label{intro.rem:scope-sharpness}
The sharpness of Theorem~\ref{intro.thm:mainmetap} refers to its necessary and sufficient conditions on the parameters. We do not claim that the constants supplied for the general radial inequality \eqref{Pittgeneral} are optimal. 
In contrast, Proposition~\ref{intro.thm:Pittpq2} identifies the optimal constant in the directional integral inequality for $p=q=2$  \eqref{PittwithmatricesBinv} as Beckner's $r$-dimensional
constant $K_{r,\alpha}$, see Remarks~\ref{rem:constants}, \ref{rem:Step344} and~\ref{remarkmatrices} for the geometric bounds and the sharpness argument.
\end{remark}
}

The importance of Theorem \ref{intro.thm:mainmetap} is not only that the admissible range is sharp, but also that it makes explicit how the geometry of a metaplectic operator determines the possible weighted estimates. At the two extreme values of the rank, one recovers two familiar behaviours: when $r=0$ the metaplectic operator acts without Fourier-type dispersion, whereas for $r=n$ the admissible range is the classical one for Pitt's inequality. 
The directional phenomenon occurs when $1\leq r<n$, where the distinction between effective and singular directions is directly reflected in the admissible exponents: the case $p=q$ is governed by the $r$-dimensional Pitt relation \eqref{assumptionPitt}, while the isotropic weights may also act on the singular directions and allow the additional regime $q<p$ described in \eqref{Pitt_admiss_iv}. Thus, the rank of $B$ does not merely enter the constants in the estimates, but determines the structure of the admissible region itself.
This mechanism is already present in the model case of the Fourier transform along a subspace, studied in Theorem \ref{prop:PFT-necessary}, and this model plays a central role in the proof of the full classification. 
    
\begin{example}\label{examplePittWigner}
    {Let us consider again the metaplectic operator $\widehat{\mathcal{A}_w}$ defining the Wigner distribution, with projection $\mathcal{A}_w$ as in \eqref{intro.defAw}.
    Since the top-right block of $\mathcal{A}_w$ has rank $r=1$, Theorem \ref{intro.thm:mainmetap} yields
    \begin{equation}\label{PittWigner}
        \begin{split}
            \left(\int_{\bR^2}(x^2+\xi^2)^{-\frac{\beta q}{2}}|\widehat{\mathcal A_w}F(x,\xi)|^q\rmd x\rmd\xi\right)^{1/q}
            \lesssim\left(\int_{\bR^2}(u^2+v^2)^{\frac{\alpha p}{2}}|F(u,v)|^p\rmd u\rmd v\right)^{1/p},
        \end{split}
    \end{equation}
    for every $F\in\mathcal S(\bR^2)$ if and only if $(p,q,\alpha,\beta,1,2)$ is Pitt-admissible.
    In particular, choosing $p=4$, $q=2$, $\alpha=1/2$ and $\beta=1/4$, which are Pitt-admissible, we obtain
    \begin{equation}\label{PittWigner42}
        \begin{split}
            \left(\int_{\bR^2}\frac{|\widehat{\mathcal A_w}F(x,\xi)|^2}
            {(x^2+\xi^2)^{1/4}}
            \rmd x\rmd\xi\right)^{1/2}\lesssim\left(\int_{\bR^2}(u^2+v^2)|F(u,v)|^4\rmd u\rmd v\right)^{1/4}.
        \end{split}
    \end{equation}
    Again, we discuss the case $F=f\otimes\bar f$, with $f\in\mathcal S(\bR)$. Since
    \begin{equation}
        \begin{split}
            \int_{\bR^2}(u^2+v^2)|f(u)|^4|f(v)|^4\rmd u\rmd v
            =2\|f\|_{4}^4\big\||\cdot|^{1/2}f\big\|_4^4,
        \end{split}
    \end{equation}
    the inequality \eqref{PittWigner42} becomes
    \begin{equation}\label{PittWignerTensor}
        \begin{split}
            \left(\int_{\bR^2}\frac{|Wf(x,\xi)|^2}{(x^2+\xi^2)^{1/4}}
            \rmd x\rmd\xi\right)^{1/2}
            \lesssim
            \|f\|_4\big\||\cdot|^{1/2}f\big\|_4.
        \end{split}
    \end{equation}
    Thus, the weighted $L^2$-norm of the Wigner distribution is controlled by the product of an unweighted and a weighted $L^4$-norm of the signal.
    In comparison, Theorem~\ref{intro.thm:mainmetap} with $p=q=2$ and $\alpha=\beta=1/4$ in \eqref{PittWigner}, combined with the trivial inequality
    \begin{equation}
        (u^2+v^2)^{1/4}\leq |u|^{1/2}+|v|^{1/2}
    \end{equation}
    yields
    \begin{equation}\label{PittWignerClassical}
        \begin{split}
            \left(\int_{\bR^2}\frac{|Wf(x,\xi)|^2}{(x^2+\xi^2)^{1/4}}\rmd x\rmd\xi\right)^{1/2}\lesssim\|f\|_2\big\||\cdot|^{1/4}f\big\|_2.
        \end{split}
    \end{equation}
    Hence, \eqref{PittWignerTensor} provides a different estimate for the same weighted $L^2$-norm of the Wigner distribution, involving weighted $L^4$-norms of the signal.
    Unlike \eqref{PittWignerClassical}, this estimate arises from the regime $q<p$, which is not allowed in the classical radial Pitt's inequality for the Fourier transform. We stress that the right hand-sides of \eqref{PittWignerTensor} and \eqref{PittWignerClassical} are not comparable.}
\end{example}

\medskip
\noindent
{\bf Boundedness on homogeneous Sobolev spaces and applications to Schr\"odinger evolutions.}
For $-\infty<\alpha<n/2$, the {\em homogeneous Sobolev space} $\dot H^\alpha(\rd)$ contains the tempered distributions $f\in\cS'(\rd)$ so that $\hat f\in L^1_{loc}(\rd)$ and
\begin{equation}
    \Vert f\Vert_{\dot H^\alpha}^2=\int_{\rd}|\xi|^{2\alpha}|\widehat f(\xi)|^2\rmd \xi<\infty.
\end{equation}

\begin{theorem}[Boundedness on $\dot H^\alpha$ spaces]\label{intro.thmHalpha}
    The following statements hold true.
    \begin{enumerate}[(i)]
        \item {\bf Case $C=O$.} It holds that $\widehat S:\dot H^\alpha(\rd)\to\dot H^\alpha(\rd)$ continuously for every $\alpha<n/2$.
        \item {\bf Case $C\neq O$.} It holds that $\widehat S:\dot H^\alpha(\rd)\to \dot H^{-\alpha}(\rd)$ continuously for every $0\leq\alpha<\mathrm{rank}(C)/2$.
    \end{enumerate}
\end{theorem}
{For certain simple examples of metaplectic operators, such as unitary linear changes of variables and the free Schr\"odinger propagator, these norm estimates are trivial or well-known. Theorem~\ref{intro.thmHalpha} for the Fourier transform is instead classical, and it follows by applying Pitt's inequality to $\widehat f$. More broadly, local and global Sobolev regularity for Fourier integral operators was studied by Seeger, Sogge and Stein in \cite{SSS1991}, and Ruzhansky and Sugimoto in \cite{RS2011}.
In this work, we obtain a general boundedness result for metaplectic operators on homogeneous Sobolev spaces, using our Pitt's inequality to infer quantitative estimates of their operator norms.
We address Theorem \ref{thm:metaplecticSobolev} below for the precise statement, limiting ourselves here to reporting the qualitative result.
}

Boundedness results of this type are usually obtained within the so-called {\em modulation} or {\em Wiener amalgam spaces}, which provide the natural framework for metaplectic operators. Results can be obtained also for {\em Shubin-Sobolev spaces}, which are particular instances of modulation spaces. See for example, \cite{CGPT2026,corderobook,DeGosson,FS2024,giacchi2026metaplectic}, and the references therein.
{For metaplectic operators, Theorem~\ref{thm:metaplecticSobolev} provides a quantitative boundedness result within homogeneous Sobolev spaces, depending explicitly on the geometry of the blocks of the symplectic projections. For $\alpha>0$, these power weights vanish or diverge at the origin of the frequency domain. Their effect is absent from the elementary mapping between {\em inhomogeneous} Sobolev spaces $H^\alpha\to H^{-\alpha}$, which follows from $L^2$ unitarity of metaplectic operators and the inclusion relations within $H^\alpha$ spaces, as explained in Remark~\ref{remInhomogeneous}.}

Theorem \ref{intro.thmHalpha}, together with the operator norm estimates in Theorem \ref{thm:metaplecticSobolev} below, applies to Schr\"odinger evolutions \eqref{Schro} subject to real quadratic Hamiltonians. Together with applications of the entropic uncertainty principle, these estimates are derived in Section \ref{sec:SchrodingerApplications} for the free particle equation, the quantum harmonic oscillator, the uniform magnetic potential equation and the anisotropic harmonic oscillator. In particular, the last equation treats the different directions separately, and it is precisely in this setting that the metaplectic framework becomes essential, providing a unified theory in which anisotropic equations can be treated without introducing direction-dependent arguments.

\medskip
\noindent
{\bf Outline.}
The paper is organised as follows. Section \ref{sec:prelim} collects the notation and the basic facts on symplectic and metaplectic operators that will be used throughout. Section \ref{sec:EUP} is devoted to the entropic uncertainty principle for metaplectic operators, while the complete characterisation of the radial Pitt's inequality for Fourier transforms along subspaces is obtained in Section \ref{sec:PittForPartial}. In Section \ref{sec:Pitt}, we use this characterisation to establish the radial Pitt's inequality for general metaplectic operators and then derive its directional counterpart, finally obtaining boundedness estimates for metaplectic operators on homogeneous Sobolev spaces. Section \ref{sec:SchrodingerApplications} applies the entropic and Sobolev estimates to Schr\"odinger evolutions generated by quadratic Hamiltonians. Throughout this work we shall use change of variables along subspaces and direct sums, and we have collected the necessary formulae in Appendix \ref{appendixA}.

\section{Preliminaries}\label{sec:prelim}
\subsection{Notation}
We shall denote by $x\cdot y$, $x,y\in\rd$, the standard inner product on $\rd$. We denote by $\bR^{m\times n}$ the space of $m\times n$ matrices with real entries, with $\mathrm{GL}(n,\bR)$ the space of $n\times n$ invertible matrices, and with $\mathrm{Sym}(n,\bR)$ the class of $n\times n$ symmetric matrices. 
If $M\in\bR^{m\times n}$, we shall denote by $R(M)$ its range. We denote by $I_n$ and $O_n$ the $n\times n$ identity and zero matrix, respectively, and we shall omit $n$ when it does not cause confusion.

We denote by $M^+$ the Moore-Penrose (pseudo-)inverse of $M\in\mathbb{R}^{n\times n}$, see \cite{MatrixAnalysis}. 
Recall that if $M=U\Sigma V^\top$ is a singular value decomposition of $M$, with $\Sigma=\mathrm{diag}(\sigma_1,\ldots,\sigma_n)$ the diagonal matrix with the singular values of $M$, then $M^+=V\Sigma^+U^\top$, where $\Sigma^+=\mathrm{diag}(\sigma_1^{-1},\ldots,\sigma_n^{-1})$ (upon interpreting $1/0:=0$). Concerning singular values, we denote by $\sigma_{\min}(M)$ and $\sigma_{\max}(M)$ the smallest and the largest singular values of $M$. 
The product of the non-zero singular values of $M$ is denoted by $\sigma(M)$.

The direct sum of the subspaces $\mathcal{L}_{1,2}\subseteq\rd$ is denoted by $\mathcal{L}_1\oplus\mathcal{L}_2$.
If $\mathcal{L}_1,\mathcal{L}_2\subseteq\rd$ are subspaces with $\rd=\mathcal{L}_1\oplus\mathcal{L}_2$, we denote by $\mathcal{P}_{\mathcal{L}_1}^{\mathcal{L}_2}x$ the projection of $x\in\rd$ onto $\mathcal{L}_1$ along $\mathcal{L}_{2}$. Explicitly, if $x=x_{1}+x_{2}\in\mathcal{L}_{1}\oplus\mathcal{L}_{2}$, then
\begin{equation*}
    \mathcal{P}_{\mathcal{L}_1}^{\mathcal{L}_2}x=x_{1}\qquad\textnormal{and}\qquad\mathcal{P}_{\mathcal{L}_2}^{\mathcal{L}_1}x=x_{2}.
\end{equation*}
We shall always consider these projections as operators from $\rd$ to itself. We shall therefore denote by $\|\mathcal{P}_{\mathcal{L}_1}^{\mathcal{L}_2}\|=\|\mathcal{P}_{\mathcal{L}_1}^{\mathcal{L}_2}\|_{\ell^2\to\ell^2}$ the spectral norm of $\mathcal{P}_{\mathcal{L}_1}^{\mathcal{L}_2}$. 
If $\mathcal{L}_2=\mathcal{L}_1^\perp$, we simplify the notation to $\mathcal{P}_{\mathcal{L}_1}x$ (orthogonal projection). With the pseudo-inverse notation, $M^+Mx=\mathcal{P}_{\ker(M)^\perp}x$, whereas $MM^+x=\mathcal{P}_{R(M)}x$. In particular, if $V:\mathbb{R}^r\to\mathcal{L}$ is a parametrisation of the linear subspace $\mathcal{L}\subseteq\rd$ which maps an orthonormal basis of $\mathbb{R}^r$ into an orthonormal basis of $\mathcal{L}$, i.e., $V^\top V=I_r$, then $V^+=V^\top$.

For an $n \times n$ matrix $M$ and a linear subspace $\mathcal{L}$ of $\mathbb{R}^n$ with $\dim \mathcal{L} = r$, $q_\mathcal{L}(M)$ denotes the $r$-dimensional volume of 
\[
X = \left\{ x \in \mathbb{R}^n \;\middle|\; x = \alpha_1 M e_1 + \cdots + \alpha_r M e_r,\ 0 \leq \alpha_i \leq 1,\ i = 1, \dots,r \right\}
\]
spanned by the vectors $Me_1, \dots, Me_r$, with $e_1, \dots, e_r$  any orthonormal basis of $\mathcal{L}$.
If $\dim M(\mathcal{L})= r$, then the $r$-dimensional volume of $X$ is positive, otherwise, this volume is zero. The number $q_{\mathcal{L}}(M)$ can be associated with a matrix determinant as follows: we collect the vectors $e_1, \dots, e_r$ as columns into the $n \times r$ matrix $E = (e_1| \dots| e_r)$. Assuming that $\dim M(\mathcal{L}) = r$, the matrix $ME$ has full column rank and
\begin{equation}\label{qL}
q_{\mathcal{L}}(M) = \mathrm{vol}(ME) = \sqrt{\det(E^\top M^\top M E)}.
\end{equation}
If $r=0$ we set $q_\mathcal{L}(M):=1$.
Observe that, if $\mathcal{L} = \mathbb{R}^n$ and $M$ is nonsingular, then  $q_\mathcal{L}(M) = |\det M|$.

We denote by $\|f\|_{L^p(\rd)}$ the $L^p(\rd)$-norm of $f$, and we write $\|f\|_p$ whenever the dimension can be omitted without causing confusion. If necessary, we shall use notations such as $\|\cdot\|_{L^p_x}$ or $\|f(x)\|_{L^p_x}$, to clarify the variables of integration.
The tensor product of two functions $f,g$ is $f\otimes g(x,y)=f(x)g(y)$. 

\subsection{The symplectic group}
Standard references for the following theory are \cite{corderobook,DeGosson,grochenigbook}.
Let 
    \begin{equation}
        J=\begin{pmatrix}
            O & I\\
            -I & O
        \end{pmatrix}
    \end{equation}
    be the matrix of the canonical symplectic form of $\rdd$. A matrix
    \begin{equation}\label{blockS}
        S=\begin{pmatrix}
            A & B\\
            C & D
        \end{pmatrix}, \qquad A,B,C,D\in\bR^{n\times n}
    \end{equation}
    is {\it symplectic} if $S^\top J S=J$, i.e.,
    \begin{align}
        \label{symplRel1}
        & A^\top C = C^\top A,\\
        \label{symplRel2}
        & B^\top D = D^\top B,\\
        \label{symplRel3}
        & A^\top D-C^\top B=I.
    \end{align}
    We shall use the following result from \cite{CGM2025,TMO}.
    \begin{proposition}\label{PropIsomorphisms}
        Let $S\in\Sp(n,\bR)$ have blocks \eqref{blockS}. Then,
        \begin{enumerate}[(i)]
            \item $A:\ker(B)\to A(\ker(B))$ is an isomorphism.
            \item $A^\top:R(B)^\perp\to\ker(B)$ is an isomorphism.
            \item $D^\top A:\ker(B)\to D^\top A(\ker(B))$ is an isomorphism.
            \item $AA^\top:R(B)^\perp\to A(\ker(B))$ is an isomorphism.
        \end{enumerate}
    \end{proposition}
    Consequently, we have the following decompositions into direct sums:
    \begin{align}
    \label{symplDecom1}
    &\rd=\ker(B)^\perp\oplus D^\top A(\ker(B)),\\
    \label{symplDecom2}
	&\rd=R(B)\oplus A(\ker(B))\\
    \label{symplDecom3}
    &\rd=R(B)\oplus AA^\top(R(B)^\perp).
    \end{align} 
    Throughout this work, we shall denote $P=I-B^+B$, the orthogonal projection onto $\ker(B)$.
    The projections corresponding to \eqref{symplDecom1}, \eqref{symplDecom2} and \eqref{symplDecom3}, can be written in terms of the blocks of $S$ as follows. Recall that $\sigma(B)$ is the product of the non-zero singular values of $B$ and $q_{R(B)^\perp}(A^\top)$ is defined as in \eqref{qL}.
    \begin{proposition}\label{propProjections}
        Let $S\in\Sp(n,\bR)$ have blocks \eqref{blockS}, and $P=I-B^+B$ be the orthogonal projection onto $\ker(B)$. Then, $B+AP$ is invertible and
        \begin{equation}
            |\det(B+AP)|=\sigma(B)q_{R(B)^\perp}(A^\top).
        \end{equation}
        Moreover,
        \begin{enumerate}[(i)]
            \item $\mathcal P_{D^\top A(\ker(B))}^{\ker(B)^\perp}=D^\top AP$.
            \item $\mathcal P_{R(B)}^{A(\ker(B))}=I-APD^\top=B(B+AP)^{-1}$.
            \item $\mathcal P_{\ker(B)^\perp}^{D^\top A(\ker(B))}+P=I-B^\top CP=(I+B^\top CP)^{-1}$.
            \item $B^+\mathcal P^{A(\ker(B))}_{R(B)}+PD^\top=(B+A P)^{-1}$.
        \end{enumerate}
    \end{proposition}
    \begin{proof}
        See Appendix~\ref{AppendixProj}.
    \end{proof}

\subsection{The metaplectic group}
For $x,\xi\in\rd$ and $\tau\in\bR$, let us consider the unitary operator on $L^2(\rd)$
    \begin{equation}
        \rho(x,\xi;\tau)f(t)=e^{2\pi i\tau}e^{-i\pi x\cdot\xi}e^{2\pi i\xi \cdot t}f(t-x), \qquad f\in L^2(\rd),
    \end{equation}
    the {\it Schrödinger representation of the Heisenberg group}. For every $S\in\Sp(n,\bR)$ there exists an $\widehat S$ unitary on $L^2(\rd)$ such that
    \begin{equation}\label{intertS}
       \widehat S \rho(x,\xi;\tau)\widehat S^{-1}=\rho(S(x,\xi);\tau), \qquad x,\xi\in\rd,\,\tau\in\bR.
    \end{equation}
    Such operators are called {\it metaplectic operators}.
    
    If $\widehat S$ satisfies \eqref{intertS}, then the unitary operators satisfying \eqref{intertS} are exactly those in the form $c\widehat S$ for some $c\in\mathbb{C}$, $|c|=1$ ({\it phase factor}). The group $\{\widehat S:S\in\Sp(n,\bR)\}$ has a subgroup that contains precisely two operators for each $S\in\Sp(n,\bR)$, differing by a sign. This subgroup is denoted by $\Mp(n,\bR)$ and it is called the {\it metaplectic group}. The projection $\pi^{Mp}:\widehat S\in\Mp(n,\bR)\mapsto S\in\Sp(n,\bR)$ is a double covering of $\Sp(n,\bR)$ and it is a homomorphism. Consequently, every $S\in\Sp(n,\bR)$ admits exactly two metaplectic lifts $\pm\widehat S$. Since this phase ambiguity does not affect our results, we shall disregard it throughout this work.
    Moreover, if $\widehat S\in\Mp(n,\mathbb{R})$, it is implied that $S$ is its projection and $A,B,C$ and $D$ are its blocks in \eqref{blockS}, i.e., $S=\pi^{Mp}(\widehat S)$. If the notation causes ambiguities, we shall state this relation explicitly.
    
    The following continuity properties are standard.
    \begin{proposition}
        Let $\widehat S\in\Mp(n,\bR)$. Then,
        \begin{enumerate}[(i)]
            \item $\widehat S:L^2(\rd)\to L^2(\rd)$ is unitary.
            \item $\widehat S$ restricts to a homeomorphism of $\cS(\rd)$.
            \item $\widehat S$ extends to a homeomorphism of $\cS'(\rd)$ by duality.
        \end{enumerate}
    \end{proposition} 

We recall the following integral representation formula, due to ter Morsche and Oonincx, \cite{TMO}, see also \cite{CGM2025} for the formula reinterpreted with the present notation. Recall that the decomposition \eqref{symplDecom2} holds.

\begin{theorem}
Let $\widehat S\in\Mp(n,\bR)$. Let
\begin{equation}\label{defmus}
    \mu_S=\sqrt{\frac{1}{q_{R(B)^\perp}(A^\top)\sigma(B)}}=|\det(B+AP)|^{-1/2}.
\end{equation} 
The following expressions hold in $\mathcal{S}(\rd)$.
\begin{enumerate}[(i)]
    \item If $B=O$, then
    \begin{equation}\label{repformulaBO}
        \widehat Sf(\xi)=|\det(A)|^{-1/2}e^{i\pi CA^{-1}\xi\cdot\xi}f(A^{-1}\xi),\qquad \xi\in\rd.
    \end{equation}
    \item If $B\neq O$, for every $f\in\mathcal{S}(\rd)$ the identity
\begin{equation}\label{integralSf}
    \widehat Sf(\xi_1+\xi_2)=\mu_Se^{i\pi(DB^+\xi_1\cdot\xi_1+DC^\top\xi_2\cdot\xi_2)}\int_{\ker(B)^\perp}f(y+D^\top\xi_2)e^{i\pi B^+Ay\cdot y}e^{-2\pi i(B^+\xi_1-C^\top\xi_2)\cdot y}\mathrm{d}y,
\end{equation}
holds for $\xi=\xi_1+\xi_2\in\rd$, where $\xi_1\in R(B)$, $\xi_2\in A(\ker(B))$.
\item In particular, when $S$ has block $B$ invertible, the integral representation \eqref{integralSf} simplifies:
\begin{equation}\label{integralSffree}
    \widehat Sf(\xi)=|\det(B)|^{-1/2}e^{i\pi DB^{-1}\xi\cdot\xi}\int_{\rd}f(y)e^{i\pi B^{-1}Ay\cdot y}e^{-2\pi iB^{-1}\xi \cdot y}\mathrm{d}y, \qquad f\in\mathcal{S}(\rd).
\end{equation}
\end{enumerate}
\end{theorem}
Formula \eqref{integralSf} is used in \cite{CGM2025,cordero2026anisotropic} to prove uncertainty principles for metaplectic operators. 
The main tool for this technique is the following directional representation of metaplectic operators, see \cite[Corollary 3.7]{CGM2025}.

\begin{lemma}\label{lemmaDecomp}
Assume $r=\mathrm{rank}(B)>0$ and let $V:\bR^{r}\to\ker(B)^\perp$ be a linear parametrisation of $\ker(B)^\perp$ with the property that $V^\top V=I_r$. 
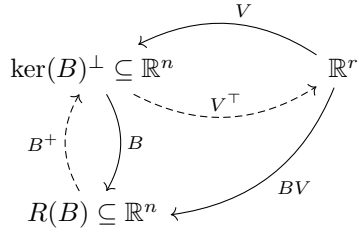
\begin{center}
\begin{tikzcd}[row sep=huge, column sep=huge]
    \ker(B)^\perp\subseteq\rd\arrow[r,"V^\top", bend right=30, dashrightarrow] \arrow[r,"V", bend left=30,leftarrow]  \arrow[d, bend left = 30,"B"] & \bR^r  \arrow[dl, "BV",bend left=30]  \\
     R(B)\subseteq\rd \arrow[u,bend left=30,"B^+",dashrightarrow]
\end{tikzcd}
\end{center}
For $\xi=\xi_1+\xi_2\in R(B)\oplus A(\ker(B))$, define 
    \begin{equation}\label{defgxi2}
        g_{\xi_2}(u)=f(Vu+D^\top\xi_2)e^{i\pi(V^\top B^+AVu\cdot u-2V^\top C^\top \xi_2\cdot u)}, \qquad u\in\bR^r.
    \end{equation}
    Then, 
    \begin{equation}\label{Sfghat}
        |\widehat Sf(\xi_1+\xi_2)|=\mu_S|\widehat g_{\xi_{2}}(V^\top B^+\xi_1)|.
    \end{equation}
\end{lemma}
\begin{remark}\label{remarkgBinvertible}
    If $B\in\mathrm{GL}(n,\bR)$, then $V=I_n$ and $R(B)=\rd$, while $A(\ker(B))=\{0\}$ and $\mu_S=|\det(B)|^{-1/2}$. In this case, the function $g_{\xi_2}$ of Lemma \ref{lemmaDecomp} simplifies as follows:
    \begin{equation}
        g(x)=f(x)e^{i\pi B^{-1}Ax\cdot x}, \qquad x\in\rd,
    \end{equation}
    and, indeed, 
    \begin{equation}\label{compactSffree}
        \widehat Sf(\xi)=|\det(B)|^{-1/2}e^{i\pi DB^{-1}\xi\cdot\xi}\int_{\rd}g(y) e^{-2\pi iB^{-1}\xi\cdot y}\rmd y=|\det(B)|^{-1/2}e^{i\pi DB^{-1}\xi\cdot\xi}\widehat g(B^{-1}\xi).
    \end{equation}
\end{remark}

\section{The entropic uncertainty principle}\label{sec:EUP}
In this section, we prove Theorem \ref{thm:main2} and Proposition \ref{thm:main1}. 
\subsection{Proof of Theorem \ref{thm:main2}}
Let $f\in \mathcal{S}(\rd)$ have $\|f\|_2=1$. 
Let $V:\bR^r\to\ker(B)^\perp$ be an orthogonal parametrisation of $\ker(B)^\perp$. Let $\xi_2\in A(\ker(B))$ and $g_{\xi_2}$ be defined as in \eqref{defgxi2}. Assume $\|g_{\xi_2}\|_{2}\neq0$. We may apply the standard entropic uncertainty principle to $g_{\xi_2}/\|g_{\xi_2}\|_{L^2(\mathbb{R}^r)}$, following the technique in \cite{CGM2025,cordero2026anisotropic}, mentioned in Section \ref{sec:prelim}.  
We have:
    \begin{align}
         \underbrace{-\int_{\bR^r}\ln\left|\frac{g_{\xi_2}(u)}{\|g_{\xi_2}\|_{2}}\right|\cdot \left|\frac{g_{\xi_2}(u)}{\|g_{\xi_2}\|_{2}}\right|^2\mathrm{d}u}_{(I)} \underbrace{-\int_{\bR^r}\ln\left|\frac{\widehat{g_{\xi_2}}(\eta)}{\|g_{\xi_2}\|_{2}}\right|\cdot \left|\frac{\widehat{g_{\xi_2}}(\eta)}{\|g_{\xi_2}\|_{2}}\right|^2\mathrm{d}\eta}_{(II)}\geq \frac{r}{2}(1-\ln(2)).
    \end{align}
    We now study the two integrals separately. Observe that
    \begin{equation}\label{normgxi2}
        \Vert g_{\xi_2}\Vert_{2}^2=\int_{\bR^r}|f(Vu+D^\top\xi_2)|^2\rmd u=\int_{\ker(B)^\perp}|f(x_1+D^\top\xi_2)|^2\rmd x_1.
    \end{equation}
    For $v\in\ker(B)$, we have $D^\top Av-v=B^\top Cv\in\ker(B)^\perp$, as in the proof of Proposition~\ref{propProjections} $(i)$. 
    Translation by $D^\top Av-v$ therefore gives \begin{equation}\label{eq:fibre-marginal}
        \begin{aligned}
            \|g_{Av}\|_2^2=\int_{\ker(B)^\perp}|f(x_1+D^\top Av-v+v)|^2\rmd x_1
            =\int_{\ker(B)^\perp}|f(x_1+v)|^2\rmd x_1=m_{\ker(B)}(f)(v).
    \end{aligned}
    \end{equation}
    Firstly,
    \begin{align}
        (I)&=-\frac{1}{\|g_{\xi_2}\|_{2}^2}\int_{\bR^r}\big( \ln|g_{\xi_2}(u)|-\ln\|g_{\xi_2}\|_{2} \big)|g_{\xi_2}(u)|^2\mathrm{d}u\\
        &=-\frac{1}{\|g_{\xi_2}\|_{2}^2}\int_{\bR^r}\ln|g_{\xi_2}(u)||g_{\xi_2}(u)|^2\mathrm{d}u+\ln\|g_{\xi_2}\|_{2}\\
        &=\ln\|g_{\xi_2}\|_{2}-\frac{1}{\|g_{\xi_2}\|_{2}^2}\int_{\bR^r}\ln|f(Vu+D^\top\xi_2)||f(Vu+D^\top\xi_2)|^2\mathrm{d}u\\
        &=\ln\|g_{\xi_2}\|_{2}-\frac{1}{\|g_{\xi_2}\|_{2}^2}\int_{\ker(B)^\perp}\ln|f(x_1+D^\top\xi_2)||f(x_1+D^\top\xi_2)|^2\mathrm{d}x_1.
    \end{align}
   Again making use of Lemma \ref{lemmaDecomp}, \eqref{CV} and
   \begin{align}
        \|g_{\xi_2}\|_{2}^2&=\|\widehat{g_{\xi_2}}\|_{2}^2=\int_{\bR^r}|\widehat{g_{\xi_2}}(\eta)|^2\mathrm{d}\eta=\mu_S^{-2}\int_{\bR^r}|\widehat Sf(BV\eta+\xi_2)|^2\mathrm{d}\eta.
    \end{align}
   the second integral becomes
    \begin{align}
        (II)&=-\frac{1}{\|g_{\xi_2}\|_{2}^2}\int_{\bR^r}\big(\ln|\widehat{g_{\xi_2}}(\eta)|-\ln\|\widehat{g_{\xi_2}}\|_{2}\big)|\widehat{g_{\xi_2}}(\eta)|^2\mathrm{d}\eta\\
        &=\ln\|g_{\xi_2}\|_{2}-\frac{1}{\|g_{\xi_2}\|_{2}^2}\int_{\bR^r}\ln|\widehat{g_{\xi_2}}(\eta)||\widehat{g_{\xi_2}}(\eta)|^2\mathrm{d}\eta\\
        &=\ln\|g_{\xi_2}\|_{2}-\frac{1}{\mu_S^2\|g_{\xi_2}\|_{2}^2}\int_{\bR^r}\ln|\mu_S^{-1}\widehat Sf(BV\eta+\xi_2)||\widehat Sf(BV\eta+\xi_2)|^2\mathrm{d}\eta\\
        &=\ln\|g_{\xi_2}\|_{2}-\frac{\ln \mu_S ^{-1}}{\mu_S^2\|g_{\xi_2}\|_{2}^2}\int_{\bR^r}|\widehat Sf(BV\eta+\xi_2)|^2\mathrm{d}\eta\\
        &\qquad \qquad\qquad -\frac{1}{\mu_S^2\|g_{\xi_2}\|_{2}^2}\int_{\bR^r}\ln|\widehat Sf(BV\eta+\xi_2)||\widehat Sf(BV\eta+\xi_2)|^2\mathrm{d}\eta\\
        &=\ln\|g_{\xi_2}\|_{2}+\ln \mu_S \\
        &\qquad \qquad\qquad-\frac{1}{\mu_S^2\sigma(B)\|g_{\xi_2}\|_{2}^2}\int_{R(B)}\ln|\widehat Sf(\xi_1+\xi_2)||\widehat Sf(\xi_1+\xi_2)|^2\mathrm{d}\xi_1.
    \end{align}  
    This way,
    \begin{equation}\label{Interm1}\begin{split}
        &\ln\|g_{\xi_2}\|_{2}^2 
        -\frac{1}{\|g_{\xi_2}\|_{2}^2}\int_{\ker(B)^\perp}\ln|f(x_1+D^\top\xi_2)||f(x_1+D^\top\xi_2)|^2\mathrm{d}x_1\\
        &\qquad\qquad\qquad-\frac{q_{R(B)^\perp}(A^\top)}{\|g_{\xi_2}\|_{2}^2}\int_{R(B)}\ln|\widehat Sf(\xi_1+\xi_2)||\widehat Sf(\xi_1+\xi_2)|^2\mathrm{d}\xi_1\geq\frac{r}{2}-\ln(2^{r/2} \mu_S ).
   \end{split} \end{equation}
    Multiplying both sides of \eqref{Interm1} by $\|g_{\xi_2}\|^2_{L^2(\bR^{{r}})}$, and integrating over $A(\ker(B))$, yields
    \begin{align}
    	&\underbrace{\int_{A(\ker(B))}\|g_{\xi_2}\|_{2}^2 \ln\|g_{\xi_2}\|_{2}^2 \mathrm{d}\xi_2}_{(III)}-\int_{A(\ker(B))}\int_{\ker(B)^\perp}\ln|f(x_1+D^\top\xi_2)||f(x_1+D^\top\xi_2)|^2\mathrm{d}x_1\mathrm{d}\xi_2\\
	&-q_{R(B)^\perp}(A^\top)\int_{A(\ker(B))}\int_{R(B)}\ln|\widehat Sf(\xi_1+\xi_2)||\widehat Sf(\xi_1+\xi_2)|^2\mathrm{d}\xi_1\mathrm{d}\xi_2\\
	&\geq \left(\frac{r}{2}-\ln(2^{r/2}\mu_S)\right)\int_{A(\ker(B))}\|g_{\xi_2}\|_{2}^2\mathrm{d}\xi_2.
    \end{align}
    Observe that on the fibers where $g_{\xi_2}=0$ all the terms multiplied by $\Vert g_{\xi_2}\Vert_{2}^{ 2}$ vanish, with the convention that $0\ln0=0$. 
    By the substitution $\xi_2=Av$, \eqref{eq:fibre-marginal} and \eqref{CV}, we have
        \begin{equation}\label{RHSentropy}
            \begin{aligned}
                \int_{A(\ker(B))}\|g_{\xi_2}\|_2^2\rmd\xi_2=q_{\ker(B)}(A)\int_{\ker(B)}m_{\ker(B)}(f)(v)\rmd v=q_{\ker(B)}(A),
            \end{aligned}
            \end{equation}
    since the marginal in Definition~\ref{intro.def.corrective.term} is a probability density.
    Again making use of \eqref{Edo1}, we obtain the expression for $H[|f|^2]$:
    \begin{align}
    	\label{entropy1}
	&\int_{A(\ker(B))}\int_{\ker(B)^\perp}\ln|f(x_1+D^\top\xi_2)||f(x_1+D^\top\xi_2)|^2\mathrm{d}x_1\mathrm{d}\xi_2={q_{\ker(B)}(A)}\int_{\rd}\ln|f(x)||f(x)|^2\mathrm{d}x.
    \end{align}
    For the entropy of $|\widehat Sf|^2$, the substitution $\xi_2=Av$ and \eqref{eq:output-fibres} give
    \[
    \begin{aligned}
        q_{R(B)^\perp}(A^\top)\int_{A(\ker(B))}\int_{R(B)}\ln|\widehat Sf(\xi_1+\xi_2)||\widehat Sf(\xi_1+\xi_2)|^2\rmd\xi_1\rmd\xi_2
        =q_{\ker(B)}(A)\int_{\rd}\ln|\widehat Sf(\xi)||\widehat Sf(\xi)|^2\rmd\xi.
    \end{aligned}
    \]
    Finally, the change of variables $\xi_2=Av$ and \eqref{eq:fibre-marginal}, yield \begin{equation}\label{HSstems}
    \begin{aligned}
        (III)=\int_{A(\ker(B))}\|g_{\xi_2}\|_2^2\ln\|g_{\xi_2}\|_2^2\rmd\xi_2
        &=q_{\ker(B)}(A)\int_{\ker(B)}m_{\ker(B)}(f)(v)\ln m_{\ker(B)}(f)(v)\rmd v\\
        &=-q_{\ker(B)}(A)H_S[f],
    \end{aligned}
    \end{equation}
    by Definition~\ref{intro.def.corrective.term}.
    Collecting the identities yields \eqref{entropyUPS}, thereby concluding the proof.

\begin{remark}
\label{remark:case2_immediate}
    	Definition~\ref{intro.def.corrective.term} for $r=n$ gives $H_S[f]=0$, while $\mu_S=|\det B|^{-1/2}$. Thus, we retrieve Proposition~\ref{thm:main1} $(ii)$ from Theorem~\ref{thm:main2}.
    \end{remark}

    \begin{remark}\label{rem:conditional-entropy}
    {The corrective term $H_S[f]$ admits a natural interpretation in terms of conditional differential entropy. Indeed, with the notation of Lemma~\ref{lemmaDecomp}, let
    \[
        p(u,v)=|g_{Av}(u)|^2,\qquad q(\eta,v)=|\widehat{g_{Av}}(\eta)|^2,
    \]
    where $u,\eta\in\bR^r$ and $v\in\ker(B)$.
    By \eqref{eq:fibre-marginal} and Plancherel's identity, $p$ and $q$ are probability densities sharing the same marginal $m=m_{\ker(B)}(f)$ in the variable $v$. 
    Let $(U,Z)$ and $(\Xi,Z)$ be random vectors with joint densities $p$ and $q$, respectively. 
    That is, for all Borel sets $E\subseteq\bR^r$ and $F\subseteq\ker(B)$,
    \[
    \begin{aligned}
        \mathrm{Prob}(U\in E,\ Z\in F)
        &=\int_F\int_E p(u,v)\rmd u\rmd v,\\
        \mathrm{Prob}(\Xi\in E,\ Z\in F)
        &=\int_F\int_E q(\eta,v)\,\rmd\eta\rmd v.
    \end{aligned}
    \]    
    In the entropy integrals below, the integrands are understood to vanish wherever $m(v)=0$.
    For every Borel set $E\subseteq\bR^r$ and $m(v)\rmd v$-almost every $v$ with $m(v)>0$,
    \[
        \mathrm{Prob}(U\in E\mid Z=v)=\int_E p(u\mid v)\rmd u=\frac{1}{m(v)}\int_E |g_{Av}(u)|^2\rmd u
    \]
    where the conditional density of $U$ given $Z=v$ is
    \begin{align}
        p(u\mid v)=\frac{p(u,v)}{m(v)}, \qquad m(v)>0
    \end{align}
    and similarly for $q$, see \cite[Section 8.4]{TT2006}. 
    The {\em conditional differential entropies} are
    \begin{equation}
        h_p(U\mid Z):=-\int_{\bR^r\times\ker(B)}p(u,v)\ln\frac{p(u,v)}{m(v)}\rmd u\rmd v
    \end{equation}
    and
    \begin{equation}
        h_q(\Xi\mid Z):=-\int_{\bR^r\times\ker(B)}q(\eta,v)\ln\frac{q(\eta,v)}{m(v)}\rmd \eta\rmd v,
    \end{equation}
    and, by \cite[Equation (8.33)]{TT2006}, they satisfy
    \[
        h_p(U\mid Z)=h(p)-h(m),\qquad h_q(\Xi\mid Z)=h(q)-h(m).
    \]
    By the changes of coordinates in \eqref{Edo1} and \eqref{eq:output-coordinates}, together with Lemma~\ref{lemmaDecomp}, we obtain
    \[
        \begin{aligned}
            &2H[|f|^2]=h_p(U\mid Z)+h(m),\\
            &2H[|\widehat Sf|^2]=h_q(\Xi\mid Z)+h(m)-2\ln\mu_S.
        \end{aligned}
    \]
    Since $H_S[f]=h(m)$ by Definition~\ref{intro.def.corrective.term}, it follows that
    \begin{equation}\label{eq:conditional-entropy-identity}
        \begin{split}
            H[|f|^2]+H[|\widehat Sf|^2]-H_S[f]+\ln\mu_S=\frac{h_p(U\mid Z)+h_q(\Xi\mid Z)}2.
        \end{split}
    \end{equation}
    We remark that, regarding the right hand-side, the contributions of $f$ and $\widehat Sf$ are hidden within $p$ and $q$.
    Thus, subtracting $H_S[f]$ removes the full contribution of the common marginal entropy, and the entropic uncertainty principle reduces to the uncertainty principle for the Fourier transform, applied conditionally on $Z$.}
\end{remark}

    \begin{remark}\label{rem:entropy-optimality}
    {The lower bound in \eqref{entropyUPS} is optimal for each fixed $S$ with $r>0$. Choose $\phi(u)=2^{r/4}e^{-\pi|u|^2}$ and $\psi\in\mathcal S(\ker(B))$ nowhere vanishing with $\|\psi\|_2=1$. Define the normalised Schwartz function $f$ by
    \begin{equation}\label{foptimal}
        f(Vu+D^\top Av)=\phi(u)\psi(v)e^{-i\pi(V^\top B^+AVu\cdot u-2V^\top C^\top Av\cdot u)}.
    \end{equation}
    The unit Jacobian in \eqref{Edo1} gives $\|f\|_2=\|\phi\|_2\|\psi\|_2=1$. Moreover, \eqref{eq:fibre-marginal} yields $m_{\ker(B)}(f)=|\psi|^2$. 
    For every $v$, $g_{Av}(u)=\phi(u)\psi(v)$, so
    \[
        \frac{g_{Av}}{\|g_{Av}\|_2}=\frac{\psi(v)}{|\psi(v)|}\phi,
    \]
    which attains equality in the Fourier entropic uncertainty principle. All subsequent steps in the proof of Theorem~\ref{thm:main2} are identities, so equality holds in \eqref{entropyUPS}.
    When $r=n$, the same argument uses the zero-dimensional factor $\psi=1$.
    }
    \end{remark}

    \subsection{Proof of Proposition \ref{thm:main1}}
The two endpoint cases follow by Remark \ref{remark:case2_immediate} for the $B$ invertible case and the straightforward computation below in the $B=O$ case. In the latter situation, the operator $\widehat S\in\Mp(n,\bR)$ takes the form
\begin{equation}
    \widehat Sf(x)=|\det(A)|^{-1/2}e^{i\pi CA^{-1}x\cdot x}f(A^{-1}x), \qquad f\in L^2(\mathbb{R}^n)
\end{equation}
and the entropy of $\widehat S f$ can therefore be computed as follows:
\begin{align}
    \int_{\mathbb{R}^n}\ln|\widehat Sf(x)||\widehat Sf(x)|^2\mathrm{d}x&=\int_{\mathbb{R}^n}\ln(|\det(A)^{-1/2} f(A^{-1}x)|)|f(A^{-1}x)|^2|\det(A)|^{-1}\mathrm{d}x\\
    &=\int_{\mathbb{R}^n}(\ln|\det(A)^{-1/2}|+\ln|f(y)|)|f(y)|^2\mathrm{d}y\\
    &=-\frac 1 2 \ln|\det(A)|+\int_{\mathbb{R}^n}\ln|f(y)||f(y)|^2\mathrm{d}y,
\end{align}
which is exactly \eqref{entrop1-270526}. 

\medskip
To prove $(iii)$, choose $\psi\in\mathcal S(\ker(B))$ with $\|\psi\|_2=1$. For $\lambda>0$, set $\psi_\lambda(v)=\lambda^{(n-r)/2}\psi(\lambda v)$. Let $\phi$ be the $L^2(\bR^r)$-normalised Gaussian defined in Remark~\ref{rem:entropy-optimality}, and
\begin{equation}
    f_\lambda(Vu+D^\top Av)=\phi(u)\psi_\lambda(v)e^{-i\pi(V^\top B^+AVu\cdot u-2V^\top C^\top Av\cdot u)}.
\end{equation}
Then, by \eqref{eq:fibre-marginal},
\[
    m_{\ker(B)}(f_\lambda)=|\psi_\lambda|^2,
\]
and hence
\[
    H_S[f_\lambda]=h(|\psi|^2)-(n-r)\ln\lambda.
\]
By Remark~\ref{rem:entropy-optimality}, each $f_\lambda$ attains equality in \eqref{entropyUPS}, whence
\[
    H[|f_\lambda|^2]+H[|\widehat Sf_\lambda|^2]=\frac r2-\ln(2^{r/2}\mu_S)+h(|\psi|^2)-(n-r)\ln\lambda,
\]
which goes to $-\infty$ as $\lambda\to+\infty$, thereby concluding the proof.

\section{Radial Pitt's inequality for partial Fourier transforms}\label{sec:PittForPartial}
We begin by considering a guiding example of a metaplectic operator, which our more general Theorem \ref{intro.thm:mainmetap} will rely upon.
With this in mind, the goal of the current section is to provide a complete characterisation of Pitt's inequality for Fourier transforms along subspaces.
Let $\mathcal{L}\subset\mathbb{R}^{n}$ be a linear subspace such that $0<\dim(\mathcal{L}^{\perp})=r<n$.
For $f\in\mathcal{S}(\mathcal{L}^{\perp}\times \mathcal{L})$, let
\[
    \mathcal{F}_{x}f(\xi,u):=\int_{\mathcal{L}^{\perp}}e^{-2\pi i x\cdot\xi}f(x,u)\mathrm{d}x
\]
denote the Fourier transform in the $\mathcal{L}^{\perp}$ variable.

To state our theorem, we recall the definition of Pitt-admissibility from Definition \ref{intro.defPittAdmissible}, noting that \eqref{Pitt_admiss_i} and \eqref{Pitt_admiss_ii} are not relevant for this example since $0 < r< n$ is assumed. 
\begin{theorem}
\label{prop:PFT-necessary}
There exists $K^{\ast}=K^{\ast}(p,q,\alpha,\beta,r,n)>0$, so that the inequality
\begin{equation}
\label{eq:PFT-Pitt-necessary-assumption}
\left\|
|u+\xi|^{-\beta}\mathcal{F}_{x}f(\xi,u)
\right\|_{L^{q}(\mathcal{L}^{\perp}_{\xi}\times \mathcal{L}_{u})}
\leq K^{\ast}
\left\|
|u+x|^{\alpha}f(x,u)
\right\|_{L^{p}(\mathcal{L}^{\perp}_{x}\times \mathcal{L}_{u})}
\end{equation}
holds for every $f\in\mathcal{S}(\mathcal{L}^{\perp}\times \mathcal{L})$ if and only if the datum $(p,q,\alpha,\beta,r,n)$ is Pitt-admissible, that is if and only if the datum satisfies \eqref{Pitt_admiss_iii} or \eqref{Pitt_admiss_iv} of Definition \ref{intro.defPittAdmissible}.
\end{theorem}

When proving Theorem \ref{prop:PFT-necessary}, it will be important to make reference to the expanded three inequalities in \eqref{Pitt_admiss_iv} of Definition \ref{intro.defPittAdmissible}, which we write explicitly as:
\begin{align}
\alpha-\beta
-r\left(1-\frac1p-\frac1q\right)
&>
-(n-r)\left(\frac1q-\frac1p\right),
\tag{B}
\label{eq:PFT-necessary-interface-lower}
\\[4pt]
\alpha-\beta
-r\left(1-\frac1p-\frac1q\right)
&<
(n-r)\left(\frac1q-\frac1p\right),
\tag{D}
\label{eq:PFT-necessary-interface-upper}
\\[4pt]
\alpha+\beta
&>
(n-r)\left(\frac1q-\frac1p\right)
+r\left(1-\frac1p-\frac1q\right),
\tag{$\textnormal{G}_{+}$}
\label{eq:PFT-necessary-lower-sum-plus}
\\[4pt]
\alpha+\beta
&>
(n-r)\left(\frac1q-\frac1p\right)
-r\left(1-\frac1p-\frac1q\right),
\tag{$\textnormal{G}_{-}$}
\label{eq:PFT-necessary-lower-sum-minus}
\\[4pt]
\alpha+\beta
&>
n\left(\frac1q-\frac1p\right),
\tag{F}
\label{eq:PFT-necessary-packet}
\\[4pt]
\alpha+\beta
&<
r+n\left(\frac1q-\frac1p\right).
\tag{C}
\label{eq:PFT-necessary-upper-sum}
\end{align}
To digest what these conditions stipulate, we refer to Figures~\ref{fig:PFT-necessary-1} and \ref{fig:PFT-necessary-2}.

\definecolor{necfill}{RGB}{236,236,236}
\definecolor{diaggreen}{RGB}{25,135,85}

\tikzset{
  axis/.style={line width=.55pt,-{Stealth[length=4pt]}},
  box/.style={line width=.55pt},
  necbd/.style={line width=1.0pt,color=black},
  strictbd/.style={
    line width=1.2pt,
    dash pattern=on 4pt off 3pt,
    color=red!75!black
  },
  tinylabel/.style={font=\scriptsize},
  smalllabel/.style={font=\small}
}

\newcommand{\PFTopenpt}[2]{%
  \draw[line width=.85pt,fill=white] (#1,#2) circle[radius=.0105];
}
\newcommand{\PFTclosedpt}[2]{%
  \fill[diaggreen] (#1,#2) circle[radius=.0115];
}
\newcommand{\PFTaxes}[2]{%
  \draw[box] (0,0) rectangle (1,1);
  \draw[axis] (0,0)--(1.075,0);
  \draw[axis] (0,0)--(0,1.075);
  \node[smalllabel,anchor=west] at (1.085,0) {$\frac1p$};
  \node[smalllabel,anchor=south] at (0,1.085) {$\frac1q$};
  \draw[line width=.42pt] (0,0)--(1,1);
  \draw[line width=.42pt] (.5,-.012)--(.5,.012);
  \draw[line width=.42pt] (-.012,.5)--(.012,.5);
  \node[tinylabel,anchor=north] at (.5,-.030) {$\frac12$};
  \node[tinylabel,anchor=east] at (-.032,.5) {$\frac12$};
  \node[tinylabel,anchor=north east] at (0,0) {$0$};
  \node[tinylabel,anchor=north] at (1,0) {$1$};
  \node[tinylabel,anchor=east] at (0,1) {$1$};
  \node[smalllabel,align=center,anchor=south] at (.5,1.13)
       {#1\\[-1pt]{\scriptsize #2}};
}


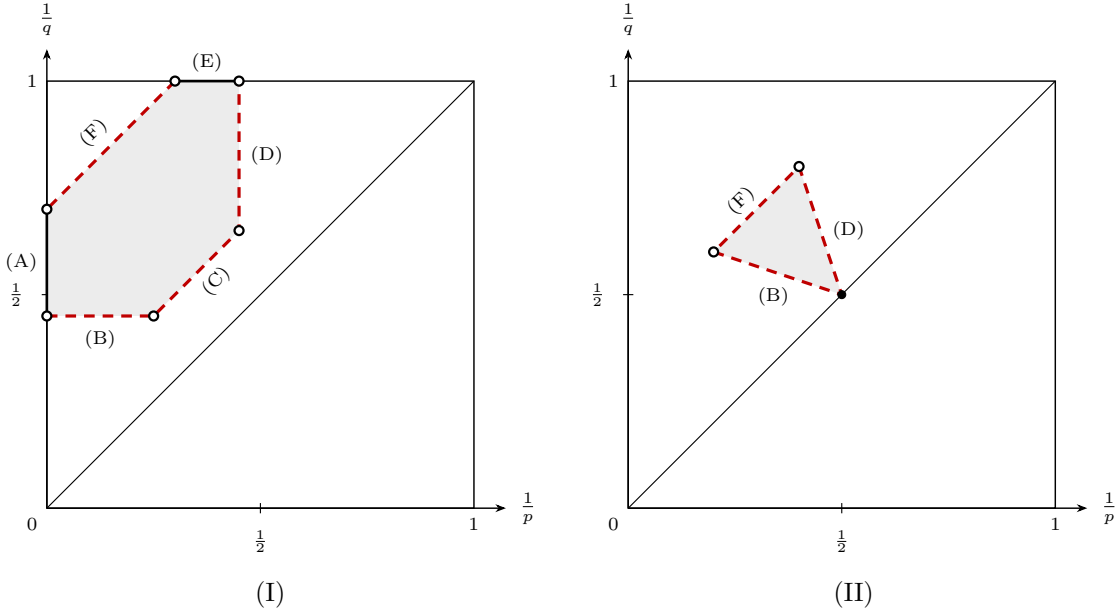
\begin{figure}[H]
\centering

\begin{minipage}[t]{0.465\textwidth}
\centering
\begin{tikzpicture}[x=5.65cm,y=5.65cm]
\PFTaxes{}{}

\fill[necfill]
  (0,.45)--(.25,.45)--(.45,.65)--(.45,1)--(.30,1)--(0,.70)--cycle;

\draw[necbd] (0,.45)--(0,.70);
\draw[strictbd] (0,.45)--(.25,.45)--(.45,.65)--(.45,1);
\draw[necbd] (.45,1)--(.30,1);
\draw[strictbd] (.30,1)--(0,.70);

\PFTopenpt{0}{.45}
\PFTopenpt{.25}{.45}
\PFTopenpt{.45}{.65}
\PFTopenpt{.45}{1}
\PFTopenpt{.30}{1}
\PFTopenpt{0}{.70}

\node[tinylabel,anchor=east] at (0,.575) {(A)};
\node[tinylabel,anchor=north] at (.125,.44) {(B)};
\node[tinylabel,anchor=west,rotate=45] at (.355,.49) {(C)};
\node[tinylabel,anchor=west] at (.455,.825) {(D)};
\node[tinylabel,anchor=south] at (.375,1) {(E)};
\node[tinylabel,anchor=east,rotate=45] at (.15,.92) {(F)};
\end{tikzpicture}

\vspace{1mm}
\textnormal{(I)}
\end{minipage}
\hspace{0.015\textwidth}
\begin{minipage}[t]{0.465\textwidth}
\centering
\begin{tikzpicture}[x=5.65cm,y=5.65cm]
\PFTaxes{}{}

\coordinate (P0) at (.20,.60);
\coordinate (P1) at (.50,.50);
\coordinate (P2) at (.40,.80);

\fill[necfill] (P0)--(P1)--(P2)--cycle;
\draw[strictbd] (P0)--(P1)--(P2)--cycle;

\PFTopenpt{.20}{.60}
\filldraw[black] ({.50},{.50}) circle[radius=1.5pt];
\PFTopenpt{.40}{.80}

\node[tinylabel,anchor=north] at (.34,.54) {(B)};
\node[tinylabel,anchor=west] at (.456,.65) {(D)};
\node[tinylabel,anchor=east,rotate=45] at (.305,.76) {(F)};
\end{tikzpicture}

\vspace{1mm}
\textnormal{(II)}
\end{minipage}

\caption{%
Necessary conditions for Pitt's inequality for the partial Fourier
transform. The first figure \textnormal{(I)} represents the case $n=2$, $r=1$, $\alpha=0.75>0.65=\beta$, in which the conditions $(\textnormal{G}_{+})$ and $(\textnormal{G}_{-})$ are redundant. The second figure \textnormal{(II)} represents the case $n=6$, $r=4$, $\alpha=1.2=\beta$, where $(\textnormal{G}_{+})$, $(\textnormal{G}_{-})$ and (C) are redundant. The boundaries (A) and (E) represent the endpoint cases $p=\infty$ and $q=1$, respectively.
}
\label{fig:PFT-necessary-1}
\end{figure}


\begin{figure}[H]
\centering

\begin{minipage}[t]{0.465\textwidth}
\centering
\begin{tikzpicture}[x=5.65cm,y=5.65cm]
\PFTaxes{}{}

\coordinate (P0) at ({1/6},{1/2});
\coordinate (P1) at ({17/80},{29/80});
\coordinate (P2) at ({13/40},{13/40});
\coordinate (P3) at ({29/120},{23/40});

\fill[necfill] (P0)--(P1)--(P2)--(P3)--cycle;
\draw[strictbd] (P0)--(P1)--(P2)--(P3)--cycle;

\PFTopenpt{1/6}{1/2}
\PFTopenpt{17/80}{29/80}
\filldraw[black] ({13/40},{13/40}) circle[radius=1.5pt];
\PFTopenpt{29/120}{23/40}

\node[tinylabel,anchor=east] at (.2,.415) {$(\textnormal{G}_+)$};
\node[tinylabel,anchor=north] at (.252,.352) {(B)};
\node[tinylabel,anchor=west] at (.288,.455) {(D)};
\node[tinylabel,anchor=east,rotate=45] at (.211,.6) {(F)};
\end{tikzpicture}

\vspace{1mm}
\textnormal{(III)}
\end{minipage}
\hspace{0.015\textwidth}
\begin{minipage}[t]{0.465\textwidth}
\centering
\begin{tikzpicture}[x=5.65cm,y=5.65cm]
\PFTaxes{}{}

\coordinate (P0) at ({17/40},{91/120});
\coordinate (P1) at ({27/40},{27/40});
\coordinate (P2) at ({51/80},{63/80});
\coordinate (P3) at ({1/2},{5/6});

\fill[necfill] (P0)--(P1)--(P2)--(P3)--cycle;
\draw[strictbd] (P0)--(P1)--(P2)--(P3)--cycle;

\PFTopenpt{17/40}{91/120}
\filldraw[black] ({27/40},{27/40}) circle[radius=1.5pt];
\PFTopenpt{51/80}{63/80}
\PFTopenpt{1/2}{5/6}

\node[tinylabel,anchor=north] at (.535,.72) {(B)};
\node[tinylabel,anchor=west] at (.64,.753) {(D)};
\node[tinylabel,anchor=south] at (.575,.815) {$(\textnormal{G}_-)$};
\node[tinylabel,anchor=east,rotate=45] at (.47,.86) {(F)};
\end{tikzpicture}

\vspace{1mm}
\textnormal{(IV)}
\end{minipage}

\caption{%
Necessary conditions for Pitt's inequality for the partial Fourier transform in situations where $(\textnormal{G}_{+})$ and $(\textnormal{G}_{-})$ are not redundant, but (C) is. Figure \textnormal{(III)} represents the case $n=6$, $r=4$, $\alpha=1.7>0.3=\beta$, whereas figure \textnormal{(IV)} represents the case  $n=6$, $r=4$, $\alpha=0.3<1.7=\beta$.
}
\label{fig:PFT-necessary-2}
\end{figure}
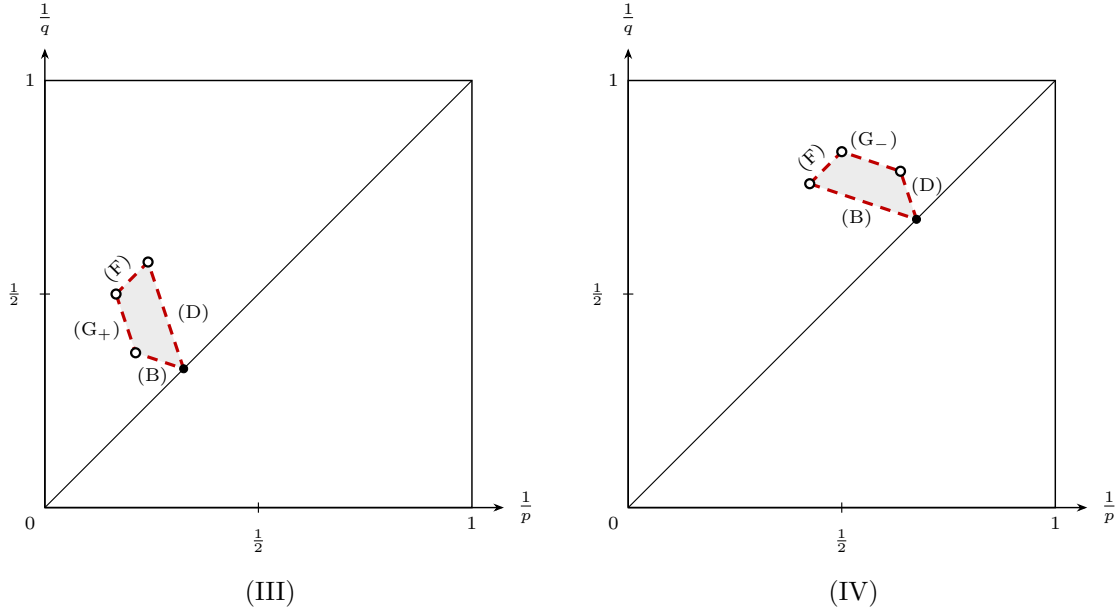

The rest of this section is devoted to the proof of Theorem~\ref{prop:PFT-necessary}. It is organised in two parts: in subsection \ref{necessity-13082026} we show that the claimed conditions are necessary and in subsection \ref{sufficiency-13082026} that the same conditions are sufficient. 
Following an orthogonal change of variables, we reduce to the case $\mathcal{L}^\perp=\bR^r$ and $\mathcal{L}=\bR^{n-r}$. Therefore, we limit ourselves to proving the theorem for the partial Fourier transform $\mathcal F_{\mathcal J}$, for $\mathcal J=\{1,\ldots,r\}$, given by
\begin{equation}\label{PartialFT}
    \mathcal{F}_{\mathcal{J}}f(x_{\mathcal{J}^c}+\xi_{\mathcal{J}})=i^{-r/2}\int_{\bR^r}f(x)e^{-2\pi i\xi_{\mathcal{J}}x}\mathrm{d}x_{\mathcal{J}}, \qquad f\in\mathcal{S}(\rd),
\end{equation}
where $x_{\mathcal{J}}$ denotes the vector $(x_1,\ldots,x_n)$ with $x_j=0$ if $j\notin\mathcal{J}$, while $\mathrm{d}x_{\mathcal{J}}=\mathrm{d}x_{j_1}...\mathrm{d}x_{j_{r}}$, where $\mathcal{J}=\{1\leq j_1<\ldots< j_{r}\leq n\}$.

\subsection{Necessity}\label{necessity-13082026} 
We begin by proving the necessity of $q\leq p$. Fix $h=e^{-\pi|\cdot|^2}\in\mathcal S(\bR^r)$. Let $\psi\in \mathcal C_c^\infty(\bR^{n-r})$ be supported in the unit ball. Since $x\perp u$ and $\xi\perp u$,
\[
|u+\xi|^{-\beta}\geq(1+|\xi|^2)^{-\beta/2}, \qquad |u+x|^\alpha\leq(1+|x|^2)^{\alpha/2}
\]
on the support of $\psi$. Hence \eqref{eq:PFT-Pitt-necessary-assumption} applied to $f_\lambda(x,u)=h(x)\psi(\lambda u)$, $\lambda\geq1$, implies
\begin{equation}\label{PFT-LpLq-fibre-test}
\|\psi\|_{q}\lesssim_{\alpha,\beta,p,q,r,n}\lambda^{(n-r)(1/q-1/p)}\frac{\|(1+|\cdot|^2)^{\alpha/2}h\|_p}{\|(1+|\cdot|^2)^{-\beta/2}h\|_q}\|\psi\|_{p}.
\end{equation}
Letting $\lambda\to\infty$ forces $q\leq p$.

\medskip
\noindent
{\bf Step 1 ($p=q$).} Assume that $p=q$. We prove the necessity of the Pitt-admissibility conditions in this case, that is conditions \eqref{Pitt_admiss_iii}. 
Fix $h\in\mathcal S(\bR^r)$ and $\psi\in \mathcal C_c^\infty(\bR^{n-r})$ with $\|\psi\|_p\neq0$, and define
\[
\psi_\varepsilon(u)=\varepsilon^{-\frac{n-r}{p}}\psi(u/\varepsilon),
\qquad
f_\varepsilon(x,u)=h(x)\psi_\varepsilon(u).
\]
If $p<\infty$, applying \eqref{eq:PFT-Pitt-necessary-assumption} to $f_\varepsilon$, we have
\begin{align}
    \varepsilon^{-(n-r)}\int_{\rd}|u+\xi|^{-\beta p}|\hat h(\xi)|^p|\psi(u/\varepsilon)|^p\rmd u\rmd\xi\lesssim \varepsilon^{-(n-r)}\int_{\rd}|u+x|^{\alpha p}|h(x)|^p|\psi(u/\varepsilon)|^p\rmd x\rmd u.
\end{align}
Changing variables via $u=\varepsilon v$, we obtain
\begin{equation}
    \int_{\rd} |\varepsilon v+\xi|^{-\beta p}|\hat h(\xi)|^p|\psi(v)|^p\rmd v\rmd\xi\lesssim \int_{\rd}|x+\varepsilon v|^{\alpha p}|h(x)|^p|\psi(v)|^p\rmd x\rmd v.
\end{equation}
Therefore,
\begin{equation}
    \liminf_{\varepsilon\to0}\int_{\rd} |\xi+\varepsilon v|^{-\beta p}|\hat h(\xi)|^p|\psi(v)|^p\rmd v\rmd\xi\lesssim \liminf_{\varepsilon\to0}\int_{\rd}|x+\varepsilon v|^{\alpha p}|h(x)|^p|\psi(v)|^p\rmd x\rmd v.
\end{equation}
By Fatou's lemma, the left-hand side becomes
\begin{align}
    \liminf_{\varepsilon\to0}\int_{\rd} |\xi+\varepsilon v|^{-\beta p}|\hat h(\xi)|^p|\psi(v)|^p\rmd v\rmd\xi&\geq\int_{\rd} |\xi|^{-\beta p}|\hat h(\xi)|^p|\psi(v)|^p\rmd v\rmd\xi\\
    &=\|\psi\|_{L^p(\bR^{n-r})}^p\||\cdot|^{-\beta}\hat h\|_{L^p(\bR^r)}^p.
\end{align}
On the other hand, the dominated convergence theorem on the right-hand side gives
\begin{equation}
\liminf_{\varepsilon\to0}\int_{\rd}|x+\varepsilon v|^{\alpha p}|h(x)|^p|\psi(v)|^p\rmd x\rmd v=\int_{\rd}|x|^{\alpha p}|h(x)|^p|\psi(v)|^p\rmd x\rmd v=\||\cdot|^\alpha h\|^p_{L^p(\bR^{r})}\|\psi\|_{L^p(\bR^{n-r})}^p.
\end{equation}
Consequently,
\[
\big\||\cdot|^{-\beta}\widehat h\big\|_{L^p(\bR^r)} \lesssim \big\||\cdot|^\alpha h\big\|_{L^p(\bR^r)}.
\]
For $p=\infty$, assume again that $\psi\in\mathcal C^\infty_c(\bR^{n-r})$ and furthermore assume $\psi(0)\neq0$. Applying Pitt's inequality to $f_\varepsilon(x,u)=h(x)\psi(u/\varepsilon)$ and evaluating the left-hand side on the fibre $u=0$ yields
\begin{equation}
    \||x+u|^{-\beta}\widehat h(x)\psi(u/\varepsilon)\|_\infty\geq |x|^{-\beta}|\widehat h(x)||\psi(0)|.
\end{equation}
As for the right-hand side, observe that
\begin{equation}
    \sup_{u,x}|x+u|^\alpha |h(x)\psi(u/\varepsilon)|=\sup_{u,x}|x+u\varepsilon|^\alpha |h(x)\psi(u)|
    \leq\|\psi\|_\infty\sup_{\substack{x\in\bR^r\\u\in\operatorname{supp}\psi}}|x+\varepsilon u|^\alpha|h(x)|.
\end{equation}
Since $\psi$ has compact support and $h\in\cS(\bR^r)$, $|x+\varepsilon u|^\alpha|h(x)|$ converges to $|x|^\alpha|h(x)|$ uniformly for $x\in\bR^r$ and $u\in\operatorname{supp}\psi$.
Taking the limsup as $\varepsilon\to0$, we obtain
\[
\limsup_{\varepsilon\to0}
\big\||\cdot|^\alpha f_\varepsilon\big\|_\infty
\leq
\|\psi\|_\infty
\big\||\cdot|^\alpha h\big\|_\infty.
\]
Consequently,
\begin{equation}
    \||\cdot|^{-\beta}\hat h\|_\infty\lesssim \||\cdot|^\alpha h\|_\infty.
\end{equation}
Thus, if $p=q$, Theorem \ref{thm:PittpqFT} in dimension $r$ provides the necessity of the conditions in Definition \ref{intro.defPittAdmissible} \eqref{Pitt_admiss_iii}. 

\medskip
\noindent
\textbf{Step 2 ($q<p$).} Assuming now that $q<p$, it remains to consider the necessity of the conditions in Definition \ref{intro.defPittAdmissible} \eqref{Pitt_admiss_iv}.
We will first run a scaling argument in two parameters to obtain \eqref{eq:PFT-necessary-interface-lower}, \eqref{eq:PFT-necessary-upper-sum} ,\eqref{eq:PFT-necessary-interface-upper}, \eqref{eq:PFT-necessary-lower-sum-plus} and \eqref{eq:PFT-necessary-lower-sum-minus} \textit{without} the strict inequality, which will then be modified to rule out the equality case. Condition \eqref{eq:PFT-necessary-packet} is more delicate and will be dealt with subsequently.

\medskip
\noindent
\textit{Step 2.1. Necessity of \eqref{eq:PFT-necessary-interface-lower}, \eqref{eq:PFT-necessary-upper-sum}, \eqref{eq:PFT-necessary-interface-upper}, \eqref{eq:PFT-necessary-lower-sum-plus} and \eqref{eq:PFT-necessary-lower-sum-minus} without strict inequality.} Let $\phi\in \mathcal C^{\infty}_{c}(\mathbb{R}^{r})$ be nonnegative, supported in the annulus $\{x:1< |x|< 2\}$ and strictly positive in a slightly smaller annulus, which guarantees that $\widehat{\phi}(0)>0$ and, by continuity, that there are $0<\rho^{\ast}<1$ and $c>0$ depending only on $\phi$ such that $|\widehat{\phi}(\xi)|>c$ for all $\xi\in B(0,\rho^{\ast})$. Likewise, let $\psi\in \mathcal C^{\infty}_{c}(\mathbb{R}^{n-r})$ be smooth, nonnegative, supported in $\{u: 1<|u|< 2\}$ and strictly positive in a slightly smaller annulus. For $\lambda,\mu>0$ define
    \begin{equation}\label{f-lambda-mu}
        f_{\lambda,\mu}(x,u)=\phi(\lambda x)\psi(\mu u).
    \end{equation}
    By the changes of variables $\mu u=v$, $\lambda x=y$, for $p<\infty$ we can estimate the right-hand side of \eqref{eq:PFT-Pitt-necessary-assumption} for $f_{\lambda,\mu}$ by
    \begin{equation*}
    \eqalign{
       &\displaystyle\left(\int_{\mathbb{R}^{n-r}}\int_{\mathbb{R}^{r}}|u+x|^{\alpha p}|f_{\lambda,\mu}(x,u)|^{p}\mathrm{d}x\mathrm{d}u\right)^{\frac{1}{p}} \cr
       &\displaystyle\qquad\qquad\qquad\qquad\qquad = \left(\int_{\mathbb{R}^{n-r}}|\psi(\mu u)|^{p}\int_{\mathbb{R}^{r}}|u+x|^{\alpha p}|\phi(\lambda x)|^{p}\mathrm{d}x\mathrm{d}u\right)^{\frac{1}{p}} \cr
       &\displaystyle\qquad\qquad\qquad\qquad\qquad =\mu^{-\frac{(n-r)}{p}}\lambda^{-\frac{r}{p}} \left(\int_{\mathbb{R}^{n-r}}|\psi(v)|^{p}\int_{\mathbb{R}^{r}}|\mu^{-1}v+\lambda^{-1}y|^{\alpha p}|\phi(y)|^{p}\mathrm{d}y\mathrm{d}v\right)^{\frac{1}{p}} \cr
       &\displaystyle\qquad\qquad\qquad\qquad\qquad\approx_{\alpha,p}\mu^{-\frac{(n-r)}{p}}\lambda^{-\frac{r}{p}}\max\{\mu^{-1},\lambda^{-1}\}^{\alpha},
       }
    \end{equation*}
    since $|v|,|y|\approx 1$. One can similarly conclude that $\||u+x|^{\alpha}f_{\lambda,\mu}\|_{\infty}\approx \max\{\mu^{-1},\lambda^{-1}\}^{\alpha}$. 
    
    We now seek a lower bound for the left-hand side. This will be obtained by restricting the frequencies $\xi$ to the window $A_{\rho}:=\{\xi\in\mathbb{R}^{r}:\frac{\rho}{2}<|\xi|<\rho\}$, for some $\rho>0$ such that $\lambda^{-1} A_{\rho}\subset B(0,\rho^{\ast})$ with $\rho$ to be later chosen. We remark that the condition imposed on $\rho$ implies that $|\widehat{\phi}(\lambda^{-1}\xi)|>c$. By the changes of variables $\mu u=v$, $\lambda^{-1}\xi=\eta$, we have
    \begin{equation}\label{LHS-Pitt-PSF-10082026}
    \eqalign{
       &\displaystyle\left(\int_{\mathbb{R}^{n-r}}\int_{\mathbb{R}^{r}}|u+\xi|^{-\beta q}|\mathcal{F}_{x}f_{\lambda,\mu}(\xi,u)|^{q}\mathrm{d}\xi\mathrm{d}u\right)^{\frac{1}{q}} \cr
       &\displaystyle \qquad\qquad\qquad\qquad\qquad= \left(\int_{\mathbb{R}^{n-r}}|\psi(\mu u)|^{q}\int_{\mathbb{R}^{r}}|u+\xi|^{-\beta q}\lambda^{-rq}|\widehat{\phi}(\lambda^{-1}\xi)|^{q}\mathrm{d}\xi\mathrm{d}u\right)^{\frac{1}{q}} \cr
       &\displaystyle \qquad\qquad\qquad\qquad\qquad\geq \left(\int_{\mathbb{R}^{n-r}}|\psi(\mu u)|^{q}\int_{A_{\rho}}|u+\xi|^{-\beta q}\lambda^{-rq}|\widehat{\phi}(\lambda^{-1}\xi)|^{q}\mathrm{d}\xi\mathrm{d}u\right)^{\frac{1}{q}} \cr
       &\displaystyle \qquad\qquad\qquad\qquad\qquad=\mu^{-\frac{(n-r)}{q}}\lambda^{-r+\frac{r}{q}} \left(\int_{\mathbb{R}^{n-r}}|\psi(v)|^{q}\int_{\lambda^{-1} A_{\rho}}|\mu^{-1}v+\lambda\eta|^{-\beta q}|\widehat{\phi}(\eta)|^{q}\mathrm{d}\eta\mathrm{d}v\right)^{\frac{1}{q}} \cr
       &\displaystyle \qquad\qquad\qquad\qquad\qquad\gtrsim\mu^{-\frac{(n-r)}{q}}\lambda^{-r+\frac{r}{q}} \left(\int_{\mathbb{R}^{n-r}}|\psi(v)|^{q}\int_{\lambda^{-1} A_{\rho}}(\mu^{-1}+\rho)^{-\beta q}|\widehat{\phi}(\eta)|^{q}\mathrm{d}\eta\mathrm{d}v\right)^{\frac{1}{q}} \cr
     &\displaystyle\qquad\qquad\qquad\qquad\qquad\gtrsim_{\beta,q}\mu^{-\frac{(n-r)}{q}}\lambda^{-r}\rho^{\frac{r}{q}}\max\{\mu^{-1},\rho\}^{-\beta}.
       }
    \end{equation}
    For $p<\infty$, we conclude that for \eqref{eq:PFT-Pitt-necessary-assumption} to hold we must have
    \begin{equation}\label{general-double-scaling-10082026}
        \mu^{-\frac{(n-r)}{q}}\lambda^{-r}\rho^{\frac{r}{q}}\max\{\mu^{-1},\rho\}^{-\beta}\lesssim_{\alpha,\beta,p,q}\mu^{-\frac{(n-r)}{p}}\lambda^{-\frac{r}{p}}\max\{\mu^{-1},\lambda^{-1}\}^{\alpha}.
    \end{equation}
    Let $L>0$ be a suitably large parameter. The following table determines how to choose $\lambda$, $\mu$ and $\rho$ as a function of $L$ in each case studied here. The last column displays what \eqref{general-double-scaling-10082026} becomes after substituting $\lambda$, $\mu$ and $\rho$ by the corresponding choices.
    \[
\begin{array}{c|c|c|c|c}
\text{Condition}&\lambda&\mu&A_{\rho}&\text{Simplified \eqref{general-double-scaling-10082026}}\\ \hline
\eqref{eq:PFT-necessary-interface-lower}&L^{-1}&L& A_{\rho^{\ast}L^{-1}}&L^{\alpha-\beta-r\left(1-\frac{1}{p}-\frac{1}{q}\right)}\gtrsim L^{-(n-r)\left(\frac{1}{q}-\frac{1}{p}\right)}\\
\eqref{eq:PFT-necessary-upper-sum}&L&L& A_{\rho^{\ast}L^{-1}}& L^{\alpha+\beta}\lesssim L^{r+n\left(\frac{1}{q}-\frac{1}{p}\right)}\\
\eqref{eq:PFT-necessary-interface-upper}&L&L& A_{\rho^{\ast}L}&L^{\alpha-\beta-r\left(1-\frac{1}{p}-\frac{1}{q}\right)}\lesssim L^{(n-r)\left(\frac{1}{q}-\frac{1}{p}\right)}\\
\eqref{eq:PFT-necessary-lower-sum-plus}&L^{-1}&L^{-1}& A_{\rho^{\ast}L^{-1}}& L^{\alpha+\beta}\gtrsim L^{(n-r)\left(\frac{1}{q}-\frac{1}{p}\right)+r\left(1-\frac{1}{p}-\frac{1}{q}\right)}\\
\eqref{eq:PFT-necessary-lower-sum-minus}&L&L^{-1}& A_{\rho^{\ast}L}& L^{\alpha+\beta}\gtrsim L^{(n-r)\left(\frac{1}{q}-\frac{1}{p}\right)-r\left(1-\frac{1}{p}-\frac{1}{q}\right)}
\end{array}
\]
Observe that the choices of $\rho$-windows, $A_{\rho}$, are consistent with our construction, that is each pair $\lambda$ and $A_{\rho}$ satisfies $\lambda^{-1} A_{\rho}\subset B(0,\rho^{\ast})$. For example, take $\lambda=L$ and $\rho=\rho^{\ast}L^{-1}$ we have $\xi\in L^{-1} A_{\rho^{\ast}L^{-1}}\Longrightarrow |\xi|<\rho^{\ast}L^{-2}<\rho^{\ast}$. The other two cases of $\lambda=L^{-1}$ with $\rho=\rho^{\ast}L^{-1}$ and $\lambda=L$ with $\rho=\rho^{\ast}L$ follow similarly.

Taking $L\rightarrow\infty$ in each of the expressions of the last column above gives \eqref{eq:PFT-necessary-interface-lower}, \eqref{eq:PFT-necessary-upper-sum} ,\eqref{eq:PFT-necessary-interface-upper}, \eqref{eq:PFT-necessary-lower-sum-plus} and \eqref{eq:PFT-necessary-lower-sum-minus} without strict inequality. The appropriate modifications show that this conclusion still holds for $p=\infty$.

\medskip
\noindent
\textit{Step 2.2. Necessity of \eqref{eq:PFT-necessary-interface-lower}, \eqref{eq:PFT-necessary-upper-sum} ,\eqref{eq:PFT-necessary-interface-upper}, \eqref{eq:PFT-necessary-lower-sum-plus} and \eqref{eq:PFT-necessary-lower-sum-minus} with strict inequality.} 
Assume for a contradiction that Pitt's inequality holds and one of \eqref{eq:PFT-necessary-interface-lower}, \eqref{eq:PFT-necessary-upper-sum} ,\eqref{eq:PFT-necessary-interface-upper}, \eqref{eq:PFT-necessary-lower-sum-plus} and \eqref{eq:PFT-necessary-lower-sum-minus} holds with equality.

We will run a \textit{dyadic superposition} argument. 
For a parameter $L>0$ and some condition $\textnormal{(X)}$ where $\textnormal{X}\in \{\textnormal{B},\textnormal{C},\textnormal{D},\textnormal{G}_{+},\textnormal{G}_{-}\}$, define for $f_{\lambda,\mu}$ as in \eqref{f-lambda-mu},
\begin{equation*}
    g_{L}^{\textnormal{(X)}}=A_{L}^{\textnormal{(X)}}f_{\lambda,\mu},
\end{equation*}
where $\lambda, \mu$ are determined by the entries of the row (X) and the constants $A_{L}^{\textnormal{(X)}}$ are chosen in the following way:
\[
\begin{array}{c|c|c|c|c|c}
\textnormal{(X)}&\eqref{eq:PFT-necessary-interface-lower}&\eqref{eq:PFT-necessary-upper-sum}&\eqref{eq:PFT-necessary-interface-upper}&\eqref{eq:PFT-necessary-lower-sum-plus}&\eqref{eq:PFT-necessary-lower-sum-minus}\\ \hline
A_L^{\textnormal{(X)}}&L^{-\alpha+\frac{(n-2r)}{p}}&L^{\alpha+\frac{n}{p}}&L^{\alpha+\frac{n}{p}}&L^{-\alpha-\frac{n}{p}}&L^{-\alpha-\frac{(n-2r)}{p}}
\end{array}
\]
This choice has a simple reason: it implies, for all $1\leq p \leq\infty$,
\begin{equation*}
    \left(\int_{\mathbb{R}^{n-r}}\int_{\mathbb{R}^{r}}|u+x|^{\alpha p}|g_{L}^{\textnormal{(X)}}(x,u)|^{p}\mathrm{d}x\mathrm{d}u\right)^{\frac{1}{p}}\approx 1,
\end{equation*}
which will be very convenient in the upcoming computations. Likewise, $\||u+x|^{\alpha}g_{L}^{\textnormal{(X)}}\|_{\infty}\approx 1$. As for the left-hand side of \eqref{eq:PFT-Pitt-necessary-assumption}, an immediate consequence of \eqref{LHS-Pitt-PSF-10082026} is

\begin{equation}\label{ineq1-11082026}
    \left(\int_{\mathbb{R}^{n-r}}\int_{\mathbb{R}^{r}}|u+\xi|^{-\beta q}|\mathcal{F}_{x}g_{L}^{\textnormal{(X)}}(\xi,u)|^{q}\mathrm{d}\xi\mathrm{d}u\right)^{\frac{1}{q}}\gtrsim_{\beta,q}A_{L}^{\textnormal{(X)}}\mu^{-\frac{(n-r)}{q}}\lambda^{-r}\rho^{\frac{r}{q}}\max\{\mu^{-1},\rho\}^{-\beta}.
\end{equation}
It can be readily checked that if equality holds for any of the conditions \eqref{eq:PFT-necessary-interface-lower}, \eqref{eq:PFT-necessary-upper-sum} ,\eqref{eq:PFT-necessary-interface-upper}, \eqref{eq:PFT-necessary-lower-sum-plus} or \eqref{eq:PFT-necessary-lower-sum-minus}, the corresponding right-hand side of \eqref{ineq1-11082026} (after substitution by the appropriate choices of $\lambda$, $\mu$, $\rho$ and $A_L^{\textnormal{(X)}}$ in the tables above) is bounded from below  by a positive constant independent of $L$.
Let $L_{j}=4^{j}$ and define
\begin{equation}
    G_{N}^{\textnormal{(X)}}:=\sum_{j=1}^N g_{L_j}^{\textnormal{(X)}}.
\end{equation}
The choice of $L_{j}$ guarantees that the supports of $g_{L_j}^{\textnormal{(X)}}$ are disjoint, therefore, for $p<\infty$
\begin{equation*}
    \left(\int_{\mathbb{R}^{n-r}}\int_{\mathbb{R}^{r}}|u+x|^{\alpha p}|G_{N}^{\textnormal{(X)}}(x,u)|^{p}\mathrm{d}x\mathrm{d}u\right)^{\frac{1}{p}}\lesssim N^{\frac{1}{p}}.
\end{equation*}
Similarly, $\||u+x|^{\alpha}G_{N}^{\textnormal{(X)}}\|_{\infty}\lesssim 1$. On the other hand, $\mathcal{F}_{x}$ preserves the $u$ support since it acts only in $x$, hence pointwise in $(\xi,u)$ at most one $\mathcal{F}_{x}g_{L_{j}}^{\textnormal{(X)}}$ is nonzero. This way, for finite $q$,
\begin{equation*}
    |\mathcal{F}_{x}G_{N}^{\textnormal{(X)}}(\xi,u)|^{q}=\sum_{j=1}^{N}|\mathcal{F}_{x}g_{L_{j}}^{\textnormal{(X)}}(\xi,u)|^{q}.
\end{equation*}
By \eqref{ineq1-11082026} and \eqref{eq:PFT-Pitt-necessary-assumption}, for $p<\infty$,
\begin{equation}\label{ineq2-11082026}
\eqalign{
    \displaystyle N^{\frac{1}{q}}&\displaystyle\lesssim_{\beta,q}\left(\sum_{j=1}^{N}\int_{\mathbb{R}^{n-r}}\int_{\mathbb{R}^{r}}|u+\xi|^{-\beta q}|\mathcal{F}_{x}g_{L_{j}}^{\textnormal{(X)}}(\xi,u)|^{q}\mathrm{d}\xi\mathrm{d}u\right)^{\frac{1}{q}} \cr
    &\displaystyle\leq \left(\int_{\mathbb{R}^{n-r}}\int_{\mathbb{R}^{r}}|u+\xi|^{-\beta q}|\mathcal{F}_{x}G_{N}^{\textnormal{(X)}}(\xi,u)|^{q}\mathrm{d}\xi\mathrm{d}u\right)^{\frac{1}{q}} \cr
    &\displaystyle\lesssim_{\alpha,\beta,p,q,r,n} \left(\int_{\mathbb{R}^{n-r}}\int_{\mathbb{R}^{r}}|u+x|^{\alpha p}|G_{N}^{\textnormal{(X)}}(x,u)|^{p}\mathrm{d}x\mathrm{d}u\right)^{\frac{1}{p}} \cr
    &\displaystyle\approx N^{\frac{1}{p}}.
    }
\end{equation}
This implies $p\leq q$ as $N\rightarrow\infty$, which contradicts our $p>q$ hypothesis. A similar conclusion holds for $p=\infty$. Hence, all of the inequalities in $\{\textnormal{B},\textnormal{C},\textnormal{D},\textnormal{G}_{+},\textnormal{G}_{-}\}$ must be strict.

\medskip
\noindent
\textit{Step 2.3. Necessity of \eqref{eq:PFT-necessary-packet} without strict inequality.} Unlike the previous cases, this obstruction is not produced by a single double-scaling bump function: it comes from simultaneously translating and modulating many packets. We remark that this translation--modulation mechanism is not new: it features in Theorem~1.2 of \cite{DGT}, where the authors obtain necessary conditions for weighted Fourier inequalities.

Let $T\gg1$ be a large integer and let $b>10$ be a large constant to be chosen later.
\begin{itemize}
    \item There exists a constant $c_{r}>0$ depending on $r=\dim{(\mathcal{L}^{\perp})}$ such that for $T$ large enough we can choose exactly $T^r$ points $x_j$ and $T^r$ points $\eta_j$ inside the shell of outer radius $2bc_{r}T$ and inner radius $bc_{r}T$ in $\mathbb{R}^{r}$ centred at the origin, with $|x_{j_1}-x_{j_2}|>b$, $|\eta_{j_1}-\eta_{j_2}|>b$ for all $j_{1},j_{2}\in \{1,\ldots,T^{r}\}$, $j_{1}\neq j_{2}$. 
    \item There exists a constant $\widetilde{c}_{n-r}>0$ depending on $n-r=\dim{(\mathcal{L})}$ such that for $T$ large enough we can choose exactly $T^{n-r}$ points $u_{\ell}$ inside the shell of outer radius $2b\widetilde{c}_{n-r}T$ and inner radius $b\widetilde{c}_{n-r}T$ in $\mathbb{R}^{n-r}$ centred at the origin, with $|u_{\ell_1}-u_{\ell_2}|>b$ for all $\ell_{1},\ell_{2}\in \{1,\ldots,T^{n-r}\}$, $\ell_{1}\neq \ell_{2}$.
\end{itemize}

Defining $\phi\in \mathcal C^{\infty}_{c}(\mathbb{R}^{r})$ and $\psi\in \mathcal C^{\infty}_{c}(\mathbb{R}^{n-r})$ as in step 2.1, we now introduce the following sum of translated and modulated wave-packets
\begin{equation}\label{PT-construction-14082026}
     P_T(x,u)
 :=\sum_{\ell=1}^{T^{n-r}}\sum_{j=1}^{T^{r}}
 e^{2\pi i\eta_j\cdot x}\phi(x-x_j)\psi(u-u_\ell).
\end{equation}
Denote the supports of $x\mapsto \phi(x-x_{j})$ and $u\mapsto\psi(u-u_{\ell})$ by $\Omega_{x_{j}}^{(1)}$ and $\Omega_{u_{\ell}}^{(2)}$, respectively. For $T$ suitably large, we have
\begin{equation*}
\textnormal{supp}\left(P_{T}\right)=\left(\bigcup_{j=1}^{T^{r}}\Omega_{x_{j}}^{(1)}\right)\times\left(\bigcup_{\ell=1}^{T^{n-r}}\Omega_{u_{\ell}}^{(2)}\right),
\end{equation*}
and the product cells $\Omega_{x_{j}}^{(1)}\times\Omega_{u_{\ell}}^{(2)}$ are pairwise disjoint.

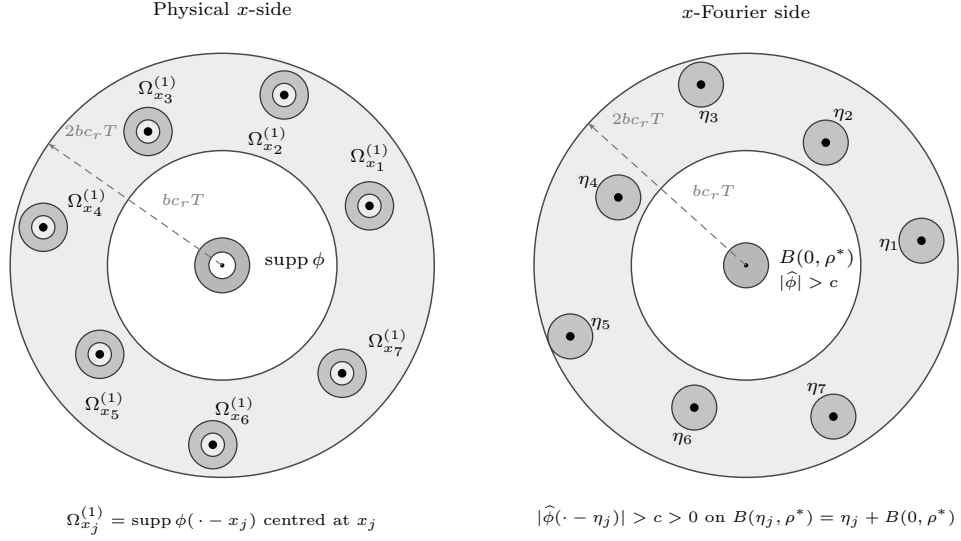
\begin{figure}[ht]
\centering
\begin{tikzpicture}[
  scale=1.1,
  boundary/.style={line width=.55pt,color=black!72},
  guide/.style={line width=.48pt,densely dashed,color=black!52},
  packetbd/.style={line width=.48pt,color=black!78},
  pt/.style={circle,fill=black,inner sep=1.15pt},
  lab/.style={font=\scriptsize},
  idx/.style={font=\scriptsize},
  tiny/.style={font=\tiny}
]
\begin{scope}[xshift=-3.15cm]
  \fill[black!7,even odd rule]
    (0,0) circle[radius=2.55]
    (0,0) circle[radius=1.38];
  \draw[boundary] (0,0) circle[radius=1.38];
  \draw[boundary] (0,0) circle[radius=2.55];

  \fill[black!28,even odd rule]
    (0,0) circle[radius=.33]
    (0,0) circle[radius=.16];
  \draw[packetbd] (0,0) circle[radius=.33];
  \draw[packetbd] (0,0) circle[radius=.16];
  \fill (0,0) circle[radius=.026];

  \node[lab,anchor=west] at (.40,.06)
    {$\operatorname{supp}\phi$};

  \coordinate (x1) at (22:1.91);
  \coordinate (x2) at (70:2.18);
  \coordinate (x3) at (119:1.84);
  \coordinate (x4) at (168:2.20);
  \coordinate (x5) at (216:1.82);
  \coordinate (x6) at (267:2.16);
  \coordinate (x7) at (318:1.94);

  \foreach \j in {1,...,7}{
    \fill[black!23,even odd rule]
      (x\j) circle[radius=.29]
      (x\j) circle[radius=.14];
    \draw[packetbd] (x\j) circle[radius=.29];
    \draw[packetbd] (x\j) circle[radius=.14];
    \node[pt] at (x\j) {};
  }

  \node[idx,anchor=center] at (36:2.2)  {$\Omega_{x_{1}}^{(1)}$};
  \node[idx,anchor=center] at (70:1.6)  {$\Omega_{x_{2}}^{(1)}$};
  \node[idx,anchor=center] at (110:2.25) {$\Omega_{x_{3}}^{(1)}$};
  \node[idx,anchor=center] at (155:1.79) {$\Omega_{x_{4}}^{(1)}$};
  \node[idx,anchor=center] at (230:2.2) {$\Omega_{x_{5}}^{(1)}$};
  \node[idx,anchor=center] at (275:1.75) {$\Omega_{x_{6}}^{(1)}$};
  \node[idx,anchor=center] at (335:2.2) {$\Omega_{x_{7}}^{(1)}$};

  \draw[guide,-{Stealth[length=2.7pt]}]
    (0,0)--(145:2.55)
    node[pos=.40,above right,tiny] {$bc_{r}T$}
    node[pos=.96,above right,tiny] {$2bc_{r}T$};

  \node[lab,anchor=south] at (0,2.86)
    {Physical $x$-side};

  \node[tiny,align=center,anchor=north] at (0,-2.78)
    {$\Omega_{x_{j}}^{(1)}=\operatorname{supp}\phi(\,\cdot-x_j)$ centred at $x_j$};
\end{scope}

\begin{scope}[xshift=3.15cm]
  \fill[black!7,even odd rule]
    (0,0) circle[radius=2.55]
    (0,0) circle[radius=1.38];
  \draw[boundary] (0,0) circle[radius=1.38];
  \draw[boundary] (0,0) circle[radius=2.55];

  \fill[black!28] (0,0) circle[radius=.27];
  \draw[packetbd] (0,0) circle[radius=.27];
  \fill (0,0) circle[radius=.026];

  \node[lab,anchor=west] at (.28,.08)
    {$B(0,\rho^{\ast})$};
  \node[tiny,anchor=west] at (.28,-.19)
    {$|\widehat\phi|> c$};

  \coordinate (e1) at (8:2.13);
  \coordinate (e2) at (57:1.76);
  \coordinate (e3) at (104:2.24);
  \coordinate (e4) at (152:1.74);
  \coordinate (e5) at (202:2.28);
  \coordinate (e6) at (250:1.82);
  \coordinate (e7) at (300:2.10);

  \foreach \j in {1,...,7}{
    \fill[black!23] (e\j) circle[radius=.27];
    \draw[packetbd] (e\j) circle[radius=.27];
    \node[pt] at (e\j) {};
  }

  \node[idx,anchor=center] at (8:1.74)   {$\eta_{1}$};
  \node[idx,anchor=center] at (57:2.16)  {$\eta_{2}$};
  \node[idx,anchor=center] at (104:1.84) {$\eta_{3}$};
  \node[idx,anchor=center] at (152:2.14) {$\eta_{4}$};
  \node[idx,anchor=center] at (202:1.88) {$\eta_{5}$};
  \node[idx,anchor=center] at (250:2.22) {$\eta_{6}$};
  \node[idx,anchor=center] at (300:1.71) {$\eta_{7}$};

  \draw[guide,-{Stealth[length=2.7pt]}]
    (0,0)--(138:2.55)
    node[pos=.40,above right,tiny] {$bc_{r}T$}
    node[pos=.91,above right,tiny] {$2bc_{r}T$};

  \node[lab,anchor=south] at (0,2.86)
    {$x$-Fourier side};

  \node[tiny,align=center,anchor=north] at (0,-2.78)
    {$|\widehat{\phi}(\cdot -\eta_{j})|>c>0$ on $B(\eta_j,\rho^{\ast})=\eta_j+B(0,\rho^{\ast})$};
\end{scope}
\end{tikzpicture}

\caption{On the left: each annulus $\Omega_{x_{j}}^{(1)}$ represents the support of $\phi(\cdot-x_{j})$. On the right, we have $|\widehat{\phi}(\xi-\eta_{j})|>c>0$ for all $\xi\in B(\eta_{j},\rho^{\ast})$. The choice of the points $u_{\ell}\in\mathbb{R}^{n-r}$ is similar to the points $x_{j}$ and is not displayed here.}
\label{fig:packet-geometry}
\end{figure}

We then have, for $p<\infty$,
\begin{equation}\label{ineq1-12082026}
    \eqalign{
    &\displaystyle\left(\int_{\mathbb{R}^{n-r}}\int_{\mathbb{R}^{r}}|u+x|^{\alpha p}|P_{T}(x,u)|^{p}\mathrm{d}x\mathrm{d}u\right)^{\frac{1}{p}} \cr
       &\displaystyle\qquad\qquad\qquad\qquad = \left(\int_{\mathbb{R}^{n-r}}\int_{\mathbb{R}^{r}}|u+x|^{\alpha p}\left|\sum_{\ell=1}^{T^{n-r}}\sum_{j=1}^{T^{r}}
 e^{2\pi i\eta_j\cdot x}\phi(x-x_j)\psi(u-u_\ell)\right|^{p}\mathrm{d}x\mathrm{d}u\right)^{\frac{1}{p}} \cr
 &\displaystyle\qquad\qquad\qquad\qquad\lesssim\left(\int_{\cup_{\ell=1}^{T^{n-r}}\Omega_{u_{\ell}}^{(2)}}\int_{\cup_{j=1}^{T^{r}}\Omega_{x_{j}}^{(1)}}|u+x|^{\alpha p}\mathrm{d}x\mathrm{d}u\right)^{\frac{1}{p}} \cr
 &\displaystyle\qquad\qquad\qquad\qquad\lesssim T^{\alpha}\left|\bigcup_{j=1}^{T^{r}}\Omega_{x_{j}}^{(1)}\right|^{\frac{1}{p}}\left|\bigcup_{\ell=1}^{T^{n-r}}\Omega_{u_{\ell}}^{(2)}\right|^{\frac{1}{p}} \cr
 &\displaystyle\qquad\qquad\qquad\qquad\lesssim T^{\alpha+\frac{n}{p}}.
    }
\end{equation}
Similarly, $\||u+x|^{\alpha }P_{T}\|_{\infty}\lesssim T^{\alpha}$. 

\medskip
\noindent
We need to address two technical points (TP) before computing the left-hand side of \eqref{eq:PFT-Pitt-necessary-assumption} for $P_T$. \\
\noindent
\textbf{(TP 1)} First, we need to choose $b$ appropriately: it will be such that
\begin{equation}\label{tail1-12082026}
    \left|\sum_{j=1}^{T^{r}}
 e^{-2\pi ix_{j}\cdot(\xi-\eta_{j})}\widehat{\phi}(\xi-\eta_j)\right|>\frac{c}{2},\quad\forall \xi\in \bigcup_{j=1}^{T^{r}}B(\eta_{j},\rho^{\ast}).
\end{equation}
Let us see why this is possible. Fix $j\in\{1,\ldots,T^{r}\}$ and let
\[
\mathcal{A}_{k}(j)
=
\left\{
i\neq j:
kb\leq |\eta_{i}-\eta_{j}|<(k+1)b
\right\},
\qquad k\geq 1.
\]
Since the points $\eta_i$ are $b$-separated, we have $\#\mathcal{A}_{k}(j)\lesssim_{r} k^{r-1}$. Since 
$\widehat{\phi}\in\mathcal{S}(\mathbb{R}^{r})$, for every $N>0$ it holds that
\[
|\widehat{\phi}(\xi)|
\lesssim_{N,\phi}
(1+|\xi|)^{-N}.
\]
If $\xi\in B(\eta_j,\rho^{\ast})$ and $i\in\mathcal{A}_{k}(j)$, then
\[
|\xi-\eta_i|
\geq
|\eta_i-\eta_j|-|\xi-\eta_j|
\geq
kb-\rho^{\ast}.
\]
Therefore,
\[
|\widehat{\phi}(\xi-\eta_i)|
\lesssim_{N,\phi}
(1+kb-\rho^{\ast})^{-N}.
\]
Assuming $b\geq 2\rho^{\ast}$ gives $kb-\rho^{\ast}\gtrsim kb$, and hence
\begin{equation*}
\sum_{i\neq j}
|\widehat{\phi}(\xi-\eta_i)|
=
\sum_{k\geq 1}
\sum_{i\in\mathcal{A}_{k}(j)}
|\widehat{\phi}(\xi-\eta_i)|
\lesssim_{N,r,\phi}
\sum_{k\geq 1}
k^{r-1}(kb)^{-N}=
b^{-N}
\sum_{k\geq 1}
k^{r-1-N}.
\end{equation*}
Choosing $N>r$, the series converges, and therefore for $b$ sufficiently large we have
\[
\sup_{1\leq j\leq T^{r}}
\sup_{\xi\in B(\eta_j,\rho^{\ast})}
\sum_{i\neq j}
|\widehat{\phi}(\xi-\eta_i)|
\lesssim_{N,r,\phi}
b^{-N}<\frac{c}{2}.
\]
Consequently, if
\[
S(\xi)
=
\sum_{i=1}^{T^{r}}
e^{-2\pi i(\xi-\eta_i)\cdot x_i}
\widehat{\phi}(\xi-\eta_i),
\]
then for every $\xi\in B(\eta_j,\rho^{\ast})$,
\begin{equation*}
|S(\xi)|
\geq
|\widehat{\phi}(\xi-\eta_j)|
-
\sum_{i\neq j}
|\widehat{\phi}(\xi-\eta_i)|
\geq
c-\frac{c}{2}
=
\frac{c}{2},
\end{equation*}
Thus, since $B(\eta_{j_{1}},\rho^{\ast})\cap B(\eta_{j_{2}},\rho^{\ast})=\emptyset$ if $j_{1}\neq j_{2}$, \eqref{tail1-12082026} holds.

\medskip
\noindent
\textbf{(TP 2)} The second technical point involves the bump function $\psi$. Let $E\subset \{u:\psi(u)\neq 0\}$ be compact with $|E|>0$ and $\textnormal{diam}(E)<1$ such that $\psi\geq c_{\psi}>0$ for some constant $c_{\psi}$. This compact set is guaranteed to exist by our hypothesis on $\psi$, and these properties imply that the $u_{\ell}+E$ are pairwise disjoint, $1\leq\ell\leq T^{n-r}$. We make this choice so that
\begin{equation*}
    \left|\sum_{\ell=1}^{T^{n-r}}
 \psi(u-u_\ell)\right|\geq c_{\psi}\quad\textnormal{for }u\in\bigcup_{\ell=1}^{T^{n-r}}(u_{\ell}+E),
\end{equation*}
which will be convenient in what follows.

\medskip
\noindent
We are now ready to obtain a lower bound for the left-hand side of \eqref{eq:PFT-Pitt-necessary-assumption}. Observe that
\begin{equation}\label{ineq2-12082026}
    \eqalign{
     &\displaystyle\left(\int_{\mathbb{R}^{n-r}}\int_{\mathbb{R}^{r}}|u+\xi|^{-\beta q}|\mathcal{F}_{x}P_{T}(\xi,u)|^{q}\mathrm{d}\xi\mathrm{d}u\right)^{\frac{1}{q}} \cr
     &\displaystyle\quad \geq \left(\int_{\cup_{\ell=1}^{T^{n-r}}(E+u_{\ell})}\int_{\bigcup_{j=1}^{T^{r}}B(\eta_{j},\rho^{\ast})}|u+\xi|^{-\beta q}\left|\sum_{\ell=1}^{T^{n-r}}\sum_{j=1}^{T^{r}}
 e^{-2\pi ix_{j}\cdot(\xi-\eta_{j})}\widehat{\phi}(\xi-\eta_j)\psi(u-u_\ell)\right|^{q}\mathrm{d}\xi\mathrm{d}u\right)^{\frac{1}{q}} \cr
 &\displaystyle\quad\gtrsim_{\phi,\psi}\left(\int_{\cup_{\ell=1}^{T^{n-r}}(E+u_{\ell})}\int_{\bigcup_{j=1}^{T^{r}}B(\eta_{j},\rho^{\ast})}|u+\xi|^{-\beta q}\mathrm{d}\xi\mathrm{d}u\right)^{\frac{1}{q}} \cr
 &\displaystyle\quad\gtrsim T^{-\beta}(T^{r}T^{n-r})^{\frac{1}{q}} \cr
 &\displaystyle\quad= T^{-\beta+\frac{n}{q}},
    }
\end{equation}
which for $p<\infty$, combined with \eqref{ineq1-12082026} and \eqref{eq:PFT-Pitt-necessary-assumption}, implies $T^{-\beta+\frac{n}{q}}\lesssim T^{\alpha+\frac{n}{p}}$. By taking $T\rightarrow\infty$, we obtain \eqref{eq:PFT-necessary-packet} without strict inequality. A similar conclusion holds for $p=\infty$.

\medskip
\noindent
\textit{Step 2.4. Necessity of \eqref{eq:PFT-necessary-packet} with strict inequality.} We will appropriately modify the function $P_{T}$ and mimic the previous dyadic superposition argument. Define
\[
 g_T^{\textnormal{(F)}}:=T^{-\alpha-n/p}P_T.
\]
As before, this choice of amplitude combined with the computation in \eqref{ineq1-12082026} guarantees that, for $p<\infty$,
\begin{equation*}
    \left(\int_{\mathbb{R}^{n-r}}\int_{\mathbb{R}^{r}}|u+x|^{\alpha p}|g_T^{\textnormal{(F)}}(x,u)|^{p}\mathrm{d}x\mathrm{d}u\right)^{\frac{1}{p}}\lesssim 1.
\end{equation*}
Similarly, $\||u+x|^{\alpha }g_T^{\textnormal{(F)}}(x,u)\|_{\infty}\lesssim 1$. On the other hand, assuming that $-\alpha-\frac{n}{p}=\beta-\frac{n}{q}$ holds, \eqref{ineq2-12082026} implies
\begin{equation*}
    \left(\int_{\cup_{\ell=1}^{T^{n-r}}(E+u_{\ell})}\int_{\bigcup_{j=1}^{T^{r}}B(\eta_{j},\rho^{\ast})}|u+\xi|^{-\beta q}|\mathcal{F}_{x}g_T^{\textnormal{(F)}}(\xi,u)|^{q}\mathrm{d}\xi\mathrm{d}u\right)^{\frac{1}{q}}\gtrsim 1
\end{equation*}
To apply the dyadic-superposition argument to $g_T^{\textnormal{(F)}}$ we will choose a sequence of widely separated radii $T_k$. More precisely, for $T_{k}>0$, let $U_{k}$ be the $u$-support of $g_{T_{k}}^{\textnormal{(F)}}$. Since the $u$-centres at scale $T_{k}$ lie in an annulus of radius comparable to $T_{k}$, if we pick $T_{k}=2^{Mk}$ then $U_{k_{1}}\cap U_{k_{2}}=\emptyset$, $k_{1}\neq k_{2}$, if $M$ is large enough. Define
\[
 G_N^{\textnormal{(F)}}:=\sum_{k=1}^N g_{T_k}^{\textnormal{(F)}}.
\] 
Because $\mathcal{F}_{x}$ preserves $u$-support, for $u \in U_{k}$ we have
\begin{equation*}
    \mathcal{F}_{x} G_N^{\textnormal{(F)}}(\xi,u)=\mathcal{F}_{x} g_{T_{k}}^{\textnormal{(F)}}(\xi,u).
\end{equation*}
The same choice of $T_{k}$ guarantees that the supports of $g_{T_{k}}^{\textnormal{(F)}}$ are pairwise disjoint. Recall that we previously picked $T_{k}^{r}$ points $\eta_{j}^{(k)}$ and $T_{k}^{n-r}$ points $u_{\ell}^{(k)}$ with certain properties. Let
\begin{equation*}
    \Omega_{k}:=\left(\bigcup_{j=1}^{T_{k}^{r}}B(\eta_{j}^{(k)},\rho^{\ast})\right)\times \left(\bigcup_{\ell=1}^{T_{k}^{n-r}}(E+u_{\ell}^{(k)})\right).
\end{equation*}
To simplify the notation, we will drop the $k$-dependence on the points $\eta_{j}^{(k)}$ and $u_{\ell}^{(k)}$, denoting them by $\eta_{j}$ and $u_{\ell}$ henceforth, but this will not compromise the argument. The computations above together with \eqref{eq:PFT-Pitt-necessary-assumption} yield, for $p<\infty$,
\begin{equation}
    \eqalign{
   \displaystyle N^{\frac{1}{q}}&\displaystyle\lesssim\left(\sum_{k=1}^{N}\int_{\cup_{\ell=1}^{T_{k}^{n-r}}(E+u_{\ell})}\int_{\cup_{j=1}^{T_{k}^{r}}B(\eta_{j},\rho^{\ast})}|u+\xi|^{-\beta q}|\mathcal{F}_{x}g_{T_k}^{\textnormal{(F)}}(\xi,u)|^{q}\mathrm{d}\xi\mathrm{d}u\right)^{\frac{1}{q}} \cr
   &=\displaystyle\left(\int_{\mathbb{R}^{n-r}}\int_{\mathbb{R}^{r}}|u+\xi|^{-\beta q}\left(\sum_{k=1}^{N}\chi_{\Omega_{k}}(\xi,u)|\mathcal{F}_{x}g_{T_k}^{\textnormal{(F)}}(\xi,u)|^{q}\right)\mathrm{d}\xi\mathrm{d}u\right)^{\frac{1}{q}} \cr
   &\leq\displaystyle\left(\int_{\mathbb{R}^{n-r}}\int_{\mathbb{R}^{r}}|u+\xi|^{-\beta q}|\mathcal{F}_{x}G_{N}^{\textnormal{(F)}}(\xi,u)|^{q}\mathrm{d}\xi\mathrm{d}u\right)^{\frac{1}{q}} \cr
   &\lesssim \displaystyle\left(\int_{\mathbb{R}^{n-r}}\int_{\mathbb{R}^{r}}|u+x|^{\alpha p}|G_N^{\textnormal{(F)}}(x,u)|^{p}\mathrm{d}x\mathrm{d}u\right)^{\frac{1}{p}} \cr
   &\leq \displaystyle\left(\sum_{k=1}^{N}\int_{\mathbb{R}^{n-r}}\int_{\mathbb{R}^{r}}|u+x|^{\alpha p}|g_{T_{k}}^{\textnormal{(F)}}(x,u)|^{p}\mathrm{d}x\mathrm{d}u\right)^{\frac{1}{p}} \cr
   &\lesssim N^{\frac{1}{p}}.
    }
\end{equation}
This implies $p\leq q$ as $N\rightarrow\infty$, which again contradicts our $p>q$ hypothesis. A similar conclusion holds for $p=\infty$.

\subsection{Sufficiency}\label{sufficiency-13082026} We will treat the diagonal and off-diagonal cases separately.\\
\noindent
{\bf Step 1 ($p=q$).} For a fixed $u\in\mathbb{R}^{n-r}$, Theorem \ref{thm:PittpqFT} gives
\begin{equation*}
    \left(\int_{\bR^r}|\xi|^{-\beta p}|\mathcal{F}_{x}f(\xi,u)|^{p}\mathrm{d}\xi\right)^{\frac1p} \lesssim_{\alpha,\beta,p,r}\left(\int_{\bR^r}|x|^{\alpha p}|f(x,u)|^{p}\mathrm{d}x\right)^{\frac1p},
\end{equation*}
provided
\begin{equation*}
    1<p<\infty,
    \qquad
    \alpha-\beta
    =
    r\left(1-\frac{2}{p}\right),
    \qquad
    \alpha+\beta<r.
\end{equation*}
On the other hand, since $x\perp u$, $\xi\perp u$, $\alpha\geq0$ and $\beta\geq0$, we have
\begin{equation*}
    |u+\xi|^{-\beta}\leq |\xi|^{-\beta},
    \qquad
    |x|^{\alpha}\leq |u+x|^{\alpha}.
\end{equation*}
Therefore,
\begin{equation}\label{diag-Pitt-14082026}
    \left(\int_{\bR^r}|u+\xi|^{-\beta p} |\mathcal{F}_{x}f(\xi,u)|^{p}\mathrm{d}\xi\right)^{\frac1p} \lesssim_{\alpha,\beta,p,r} \left(\int_{\bR^r} |u+x|^{\alpha p} |f(x,u)|^{p}\mathrm{d}x\right)^{\frac1p}.
\end{equation}
Raising both sides of \eqref{diag-Pitt-14082026} to the power $p$ and integrating in $u$ over $\mathbb{R}^{n-r}$ concludes the diagonal case.

\medskip
\noindent
{\bf Step 2 ($q<p$).}
We start with a sketch of the argument to help us build the intuition to close it. Let $1\leq q < p\leq\infty$ be a point inside the range determined by \eqref{eq:PFT-necessary-interface-lower}, \eqref{eq:PFT-necessary-upper-sum}, \eqref{eq:PFT-necessary-interface-upper}, \eqref{eq:PFT-necessary-packet}, \eqref{eq:PFT-necessary-lower-sum-plus} and \eqref{eq:PFT-necessary-lower-sum-minus}. Start by fixing $u\in\mathbb{R}^{n-r}\backslash\{0\}$ and define
\begin{equation*}
    F_{u}(x):= |u+x|^{\alpha }f(x,u), \qquad H(u):=\|F_{u}\|_{L^{p}_{x}}.
\end{equation*}
This way, $f(x,u)=|u+x|^{-\alpha}F_{u}(x)$. Let $s,p_{1},p_{2}\in [1,\infty]$ be parameters to be chosen later. By formally applying Hausdorff-Young and H\"older twice and by performing the changes of variables $\eta=|u|\xi$ and $x=|u|y$,
\begin{equation}\label{Pitt-PSF-13082026}
    \eqalign{
    \displaystyle\||u+\eta|^{-\beta }\mathcal{F}_{x}f(\eta,u)\|_{L^{q}_{\eta}}&\leq\displaystyle \||u+\eta|^{-\beta }\|_{L^{p_{2}}_{\eta}}\|\mathcal{F}_{x}f(\eta,u)\|_{L^{s'}_{\eta}} \cr
    &\leq\displaystyle\||u+\eta|^{-\beta }\|_{L^{p_{2}}_{\eta}}\|f(x,u)\|_{L^{s}_{x}} \cr
    &\leq\displaystyle\||u+\eta|^{-\beta }\|_{L^{p_{2}}_{\eta}}\||u+x|^{-\alpha }\|_{L^{p_{1}}_{x}}\|F_{u}\|_{L^{p}_{x}}. \cr
    &=\displaystyle |u|^{\frac{r}{p_{2}}-\beta}\|(1+|\xi|^{2})^{-\frac{\beta}{2} }\|_{L^{p_{2}}_{\xi}}|u|^{\frac{r}{p_{1}}-\alpha}\|(1+|y|^{2})^{-\frac{\alpha}{2} }\|_{L^{p_{1}}_{y}}\|F_{u}\|_{L^{p}_{x}} \cr
    &=\displaystyle |u|^{A(s)+B(s)}\|(1+|\xi|^{2})^{-\frac{\beta}{2} }\|_{L^{p_{2}}_{\xi}}\|(1+|y|^{2})^{-\frac{\alpha}{2} }\|_{L^{p_{1}}_{y}}\|F_{u}\|_{L^{p}_{x}},
    }
\end{equation}
where
\begin{equation*}
    A(s):=r\left(\frac{1}{s}-\frac{1}{p}\right)-\alpha,\qquad B(s):=r\left(\frac{1}{q}-\frac{1}{s'}\right)-\beta.
\end{equation*}
Several conditions must hold true for the computations above to be valid and for the final right-hand side to be finite:
\begin{itemize}
    \item The parameter $s$ must satisfy
    \begin{equation}\label{ineq1-13082026}
        1\leq s\leq\min\{2,p,q'\}.
    \end{equation}
    These constraints come from both Hausdorff-Young and H\"older.
    \item The parameters $p_{1}$ and $p_{2}$ must satisfy
    \begin{equation}\label{ineq2-13082026}
        \frac{1}{s}=\frac{1}{p}+\frac{1}{p_{1}},\qquad \frac{1}{q}=\frac{1}{s'}+\frac{1}{p_{2}}.
    \end{equation}
    These conditions are imposed by H\"older.
    \item We must have
    \begin{equation*}
        \|(1+|\xi|^{2})^{-\frac{\beta}{2} }\|_{L^{p_{2}}_{\xi}}<\infty\quad\textnormal{and}\quad \|(1+|y|^{2})^{-\frac{\alpha}{2} }\|_{L^{p_{1}}_{y}}<\infty.
    \end{equation*}
    For the values of $s$ that will be used below, we will have $s<p$ and $s<q'$, hence $p_{1},p_{2}<\infty$. In this case finiteness of the expressions above is equivalent to 
    \begin{equation}\label{ineq3-13082026}
        B(s)<0\quad\textnormal{and}\quad A(s)<0,
    \end{equation}
    respectively.
\end{itemize}
Proceeding with the formal argument, the next step is to integrate both sides of \eqref{Pitt-PSF-13082026} in $L^{q}(\mathbb{R}^{n-r})$. Again by H\"older,
\begin{equation}\label{est1-13082026}
\eqalign{
\displaystyle\||u+\eta|^{-\beta }\mathcal{F}_{x}f(\eta,u)\|_{L^{q}_{\eta,u}}&\displaystyle\lesssim \||u|^{A(s)+B(s)}\|F_{u}\|_{L^{p}_{x}}\|_{L^{q}_{u}} \cr
&\displaystyle\leq \||u|^{A(s)+B(s)}\|_{L^{p_{3}}_{u}}\||u+x|^{\alpha }f(x,u)\|_{L^{p}_{x,u}}, \cr
}
\end{equation}
where
\begin{equation}\label{p3-def-13082026}
    \frac{1}{p_{3}}=\frac{1}{q}-\frac{1}{p}.
\end{equation}
Observe that the radial power $|u|^{(A(s)+B(s))p_{3}}$ 
\begin{itemize}
    \item Integrates near zero if $(A(s)+B(s))p_{3}>-(n-r)$,
    \item Integrates near $\infty$ if $(A(s)+B(s))p_{3}<-(n-r)$.
\end{itemize}
It is then more convenient to go a step back and split the left-hand side of \eqref{est1-13082026} in the regimes $|u|\leq 1$ and $|u|>1$, and to perform the argument we are designing for exponents arbitrarily close, from above and below, to the \textit{critical power} $s=s^{\ast}$ such that
\begin{equation*}
    (A(s_{\ast})+B(s_{\ast}))p_{3}=-(n-r),
\end{equation*}
provided $s_{\ast}$ satisfies all the conditions we need. Explicitly, 
\begin{equation*}
    (A(s_{\ast})+B(s_{\ast}))p_{3}=-(n-r)\quad\Longleftrightarrow\quad s_{\ast}=\frac{2r}{r+(\alpha+\beta)-n\left(\frac{1}{q}-\frac{1}{p}\right)}.
\end{equation*}
Observe that $\alpha+\beta>n\left(\frac{1}{q}-\frac{1}{p}\right)$ by condition \eqref{eq:PFT-necessary-packet}, therefore $s_{\ast}>0$. Let us assume for a moment that $s_{\ast}$ satisfies $1<s_{\ast}<\min\{2,p,q'\}$, $A(s_{\ast})<0$ and $B(s_{\ast})<0$. By continuity, there is $\varepsilon>0$ such that $s_{\ast}^{-}$ and $s_{\ast}^{+}$ given by
\begin{equation*}
    \frac{1}{s_{\ast}^{-}}:=\frac{1}{s_{\ast}}+\varepsilon\quad\textnormal{and}\quad \frac{1}{s_{\ast}^{+}}:=\frac{1}{s_{\ast}}-\varepsilon
\end{equation*}
also satisfy $1<s_{\ast}^{-},s_{\ast}^{+}<\min\{2,p,q'\}$, $A(s_{\ast}^{-}),A(s_{\ast}^{+}),B(s_{\ast}^{-}),B(s_{\ast}^{+})<0$. Applying \eqref{Pitt-PSF-13082026} for $s_{\ast}^{-}$ and $s_{\ast}^{+}$ gives 
\begin{equation}
\eqalign{
\displaystyle\||u+\eta|^{-\beta }\mathcal{F}_{x}f(\eta,u)\|_{L^{q}_{\eta}}
&\displaystyle\lesssim |u|^{A(s_{\ast}^{-})+B(s_{\ast}^{-})}\||u+x|^{\alpha }f(x,u)\|_{L^{p}_{x}} \cr
&\displaystyle= |u|^{A(s_{\ast})+B(s_{\ast})+2r\varepsilon}\||u+x|^{\alpha }f(x,u)\|_{L^{p}_{x}}
}
\end{equation}
and
\begin{equation}
\eqalign{
\displaystyle\||u+\eta|^{-\beta }\mathcal{F}_{x}f(\eta,u)\|_{L^{q}_{\eta}}
&\displaystyle\lesssim |u|^{A(s_{\ast}^{+})+B(s_{\ast}^{+})}\||u+x|^{\alpha }f(x,u)\|_{L^{p}_{x}} \cr
&\displaystyle= |u|^{A(s_{\ast})+B(s_{\ast})-2r\varepsilon}\||u+x|^{\alpha }f(x,u)\|_{L^{p}_{x}}.
}
\end{equation}
This way, for a fixed $u\in\mathbb{R}^{n-r}\backslash\{0\}$,
\begin{equation*}
    \displaystyle\||u+\eta|^{-\beta }\mathcal{F}_{x}f(\eta,u)\|_{L^{q}_{\eta}}\lesssim\begin{cases}
|u|^{-\frac{(n-r)}{p_{3}}+2r\varepsilon}\||u+x|^{\alpha }f(x,u)\|_{L^{p}_{x}}, & 0<|u|\leq 1,\\[2mm]
|u|^{-\frac{(n-r)}{p_{3}}-2r\varepsilon}\||u+x|^{\alpha }f(x,u)\|_{L^{p}_{x}}, & |u|> 1.
\end{cases}
\end{equation*}
And this yields \eqref{eq:PFT-Pitt-necessary-assumption} after raising both sides to the power $q$, integrating then over $\mathbb{R}^{n-r}$ and repeating the steps of \eqref{est1-13082026}. There is one last step to conclude the proof.

\begin{claim}\label{Claim1-13082026} The critical exponent $s_{\ast}$ satisfies
\begin{equation*}
    1<s_{\ast}<\min\{2,p,q'\},\quad A(s_{\ast})<0\quad\textnormal{and}\quad B(s_{\ast})<0.
\end{equation*}
\end{claim}
\begin{proof}[Proof of Claim \ref{Claim1-13082026}] We will check each condition separately.

Observe that
\begin{equation}\label{cond1-14082026}  
    \displaystyle s_{\ast}<2 \quad\displaystyle\Longleftrightarrow\quad \frac{2r}{r+(\alpha+\beta)-n\left(\frac{1}{q}-\frac{1}{p}\right)}<2 \quad\Longleftrightarrow\quad \alpha+\beta > n\left(\frac{1}{q}-\frac{1}{p}\right),
\end{equation}
which is condition \eqref{eq:PFT-necessary-packet}. If $p=\infty$, then $s_{\ast}<p$ follows trivially. If $p<\infty$,
   \begin{equation}\label{cond2-14082026}  
    \displaystyle s_{\ast}<p \quad\displaystyle\Longleftrightarrow\quad \frac{2r}{r+(\alpha+\beta)-n\left(\frac{1}{q}-\frac{1}{p}\right)}<p \quad\Longleftrightarrow\quad \alpha+\beta > (n-r)\left(\frac{1}{q}-\frac{1}{p}\right)-r\left(1-\frac{1}{q}-\frac{1}{p}\right),
\end{equation}
which is condition \eqref{eq:PFT-necessary-lower-sum-minus}. If $q=1$, then $s_{\ast}<q'$ follows trivially. If $q>1$,
   \begin{equation}\label{cond3-14082026}  
    \displaystyle s_{\ast}<q' \quad\displaystyle\Longleftrightarrow\quad \frac{2r}{r+(\alpha+\beta)-n\left(\frac{1}{q}-\frac{1}{p}\right)}<q' \quad\Longleftrightarrow\quad \alpha+\beta > (n-r)\left(\frac{1}{q}-\frac{1}{p}\right)+r\left(1-\frac{1}{p}-\frac{1}{q}\right),
\end{equation}
which is condition \eqref{eq:PFT-necessary-lower-sum-plus}.
   \begin{equation}\label{cond4-14082026}  
    \displaystyle 1<s_{\ast} \quad\displaystyle\Longleftrightarrow\quad 1<\frac{2r}{r+(\alpha+\beta)-n\left(\frac{1}{q}-\frac{1}{p}\right)} \quad\Longleftrightarrow\quad \alpha+\beta < r+n\left(\frac{1}{q}-\frac{1}{p}\right),
\end{equation}
which is condition \eqref{eq:PFT-necessary-upper-sum}.
   \begin{equation}\label{cond5-14082026}  
    \displaystyle A(s_{\ast})<0 \quad\displaystyle\Longleftrightarrow\quad r\left(\frac{1}{s_{\ast}}-\frac{1}{p}\right)-\alpha<0 \quad\Longleftrightarrow\quad \alpha-\beta
-r\left(1-\frac1p-\frac1q\right)>-(n-r)\left(\frac1q-\frac1p\right),
\end{equation}
which is condition \eqref{eq:PFT-necessary-interface-lower}. Finally,
   \begin{equation}\label{cond6-14082026}  
    \displaystyle B(s_{\ast})<0 \quad\displaystyle\Longleftrightarrow\quad r\left(\frac{1}{q}-\frac{1}{(s_{\ast})'}\right)-\beta<0 \quad\Longleftrightarrow\quad \alpha-\beta
-r\left(1-\frac1p-\frac1q\right)<(n-r)\left(\frac1q-\frac1p\right),
\end{equation}
which is condition \eqref{eq:PFT-necessary-interface-upper}.
\end{proof}

The proof of sufficiency is concluded with Claim \ref{Claim1-13082026}.

\section{Pitt inequalities for metaplectic operators}\label{sec:Pitt}
In this section, we investigate Pitt's inequality for metaplectic operators under two complementary perspectives. First, we use Theorem \ref{prop:PFT-necessary} to derive necessary and sufficient conditions for $(p,q,\alpha,\beta,r,n)$ under which \eqref{Pittgeneral} holds for every $f\in\cS(\rd)$, that is we prove Theorem \ref{intro.thm:mainmetap}. 
We will refer to \eqref{Pittgeneral} as to the {\em radial Pitt's inequality} for it is expressed in terms of radial power weights. 
Second, a \textit{directional version of Pitt's inequality} can be derived directly from the classical Theorem \ref{thm:PittpqFT}, see Proposition \ref{Cor-Pitt-03082026}.  
        
    \subsection{Proof of Theorem \ref{intro.thm:mainmetap}} We consider each case in turn, proving both necessity and sufficiency. \\
   {\bf Case 1 ($B=O$).} By \eqref{repformulaBO},
\begin{equation*}
    |\widehat{S}f(x)|=|\det(A)|^{-\frac{1}{2}}|f(A^{-1}x)|.
\end{equation*}
Since $A$ is invertible, \eqref{Pittgeneral} is equivalent to
\begin{equation}\label{B0-reduction}
    \big\||\cdot|^{-\beta}f\big\|_{L^q(\mathbb R^n)}
    \lesssim_A
    \big\||\cdot|^\alpha f\big\|_{L^p(\mathbb R^n)}.
\end{equation}
For necessity, fix a non-zero
$\phi\in \mathcal C_c^\infty(\{1<|x|<2\})$ and set
$\phi_\lambda(x)=\phi(\lambda x)$, $\lambda>0$. By \eqref{B0-reduction} applied to $\phi_{\lambda}$ and a change of variables,
\[
    \lambda^{\beta-\frac nq}
    \lesssim
    \lambda^{-\alpha-\frac np},
\]
hence by taking $\lambda\rightarrow 0$ and $\lambda\rightarrow\infty$ we obtain
\begin{equation}\label{B0-scaling}
    \alpha+\beta
    =
    n\left(\frac1q-\frac1p\right).
\end{equation}
Since $\alpha,\beta\geq0$, in the case $q=\infty$ this implies $p=\infty$ and $\alpha=\beta=0$. If $q<\infty$, the constraints $\alpha,\beta\geq0$ imply $q\leq p$. Suppose that $q<p$, let $R\gg 1$ and let $g_{R}$ be a smooth truncation of the function $\widetilde{g}_{R}(x)=|x|^{-n/p}\mathbf 1_{\{2<|x|<R\}}$ (a smooth function equal to $\widetilde{g}_{R}$ on $3<|x|<R/2$ so we can use the a priori estimate \eqref{B0-reduction}). Inequality \eqref{B0-reduction} applied to
\begin{equation*}
    f(x)=|x|^{-\alpha}g_{R}(x)
\end{equation*}
gives
\begin{equation*}
    \big\||x|^{-(\alpha+\beta)}g_{R}\big\|_{L^q}
    \lesssim
    \|g_{R}\|_{L^p}.
\end{equation*}
But by \eqref{B0-scaling},
\[
    \|g_R\|_{L^p}\lesssim(\log R)^{1/p},
    \qquad
    \big\||x|^{-(\alpha+\beta)}g_R\big\|_{L^q}
    \gtrsim(\log R)^{1/q},
\]
Letting $R\to\infty$ contradicts
$q<p$, thus $p=q$ and \eqref{B0-scaling} gives
$\alpha+\beta=0$. Since $\alpha,\beta\geq 0$, we conclude that $\alpha=\beta=0$.

On the other hand, observe that for $\alpha\in\bR$ and $0<p<\infty$
   
    \begin{align}
    \int_{\mathbb{R}^{n}}|x|^{\alpha p}|\widehat{S}f(x)|^{p} \rmd x  &= |\det(A)|^{-p/2}\int_{\mathbb{R}^{n}}|x|^{\alpha p}|f(A^{-1}x)|^{p} \rmd x \\
    \label{eqn:B=O_intermediate_step}
&=|\det(A)|^{1-\frac{p}{2}}\int_{\mathbb{R}^{n}}|Ax|^{\alpha p}|f(x)|^{p} \rmd x\\
&\leq |\det(A)|^{1-\frac{p}{2}}\cdot  \begin{cases}
    \sigma_{\max}(A)^{\alpha p} & \text{if $\alpha\geq0$,}\\
    \sigma_{\min}(A)^{\alpha p} & \text{if $\alpha<0$}
    \end{cases}\cdot\int_{\mathbb{R}^{n}}|x|^{\alpha p}|f(x)|^{p} \rmd x\\
    &= |\det(A)|^{ 1-\frac{p}{2}}\max\{\sigma_{\max}(A)^\alpha ,\sigma_{\min}(A)^\alpha \}^{p}\int_{\mathbb{R}^{n}}|x|^{\alpha p}|f(x)|^{p} \rmd x,
    \end{align}
which proves the additional weighted estimate stated when $B=O$, upon replacing $\alpha$ by $\gamma$. For $p=q=\infty$ and $\alpha=\beta=0$,
\begin{equation*}
    \|\widehat{S}f\|_{\infty}=|\det{(A)}|^{-\frac{1}{2}}\|f\|_{\infty}.
\end{equation*}
Thus if $\alpha,\beta\geq 0$ and $1\leq p,q\leq \infty$, then the necessary conditions obtained above are indeed sufficient.

\medskip
\noindent
{\bf Case 2 ($B$ invertible).} In what follows, we will reduce Pitt's inequality for $\widehat{S}$ to Pitt's inequality for the Fourier transform by making use of expression \eqref{compactSffree}. We begin by proving the necessity of conditions \eqref{Pitt_admiss_ii} of Definition \ref{intro.defPittAdmissible}. By \eqref{compactSffree},
    \begin{equation}
        \widehat Sf(\xi)=|\det(B)|^{-1/2}e^{i\pi DB^{-1}\xi\cdot\xi}\widehat g(B^{-1}\xi), 
\end{equation}
where $g(y)=f(y)e^{i\pi B^{-1}Ay\cdot y}$, so that $|\widehat Sf(\xi)|=|\det(B)^{-1/2}\widehat g(B^{-1}\xi)|$. If \eqref{Pittgeneral} holds, then

\begin{equation}\label{ineq1-030826}
\eqalign{
         \displaystyle\left(\int_{\rd}|x|^{\alpha p}|g(x)|^p\rmd x\right)^{1/p}&\displaystyle\gtrsim|\det(B)|^{-1/2}\left(\int_{\rd}|\xi|^{-\beta q}|\widehat g(B^{-1}\xi)|^q\rmd\xi\right)^{1/q} \cr
         &\displaystyle\gtrsim_{B}\left(\int_{\rd}|B\xi|^{-\beta q}|\widehat g(\xi)|^q\rmd\xi\right)^{1/q}\cr
         &\displaystyle\gtrsim_{B}\left(\int_{\rd}|\xi|^{-\beta q}|\widehat g(\xi)|^q\rmd\xi\right)^{1/q},
         }
    \end{equation}
where in the last inequality we used the simple bound $|B\eta|\leq\sigma_{\max}(B)|\eta|$. Up to the implicit multiplicative constant, \eqref{ineq1-030826} is identical to Pitt's inequality for the Fourier transform in Theorem \ref{thm:PittpqFT} and the mapping $f\leftrightarrow g$ is an automorphism of $\cS(\rd)$, hence the  conditions in \eqref{assumptionPitt} are necessary for \eqref{Pittgeneral} to hold.

As for sufficiency, by applying Theorem \ref{thm:PittpqFT} to $g$ we find:
\begin{align}\label{ewf}
    \left(\int_{\rd}|\eta|^{-\beta q}|\widehat g(\eta)|^q\rmd \eta\right)^{1/q}\leq K(p,q,\alpha,\beta,n)
            \left(\int_{\rd}|u|^{\alpha p}|g(u)|^p\rmd u\right)^{1/p}
\end{align}
if and only if the relations in \eqref{assumptionPitt} hold. The integral at the right-hand side of \eqref{ewf} is clearly
\begin{equation}\label{presharp0}
    \left(\int_{\rd}|u|^{\alpha p}|g(u)|^p\rmd u\right)^{1/p}=\left(\int_{\rd}|x|^{\alpha p}|f(x)|^p\rmd x\right)^{1/p}.
\end{equation}
For the left-hand side, we use the definition of $g$ to obtain
\begin{align}
    \left(\int_{\rd}|\eta|^{-\beta q}|\widehat g(\eta)|^q\rmd \eta\right)^{1/q}=|\det(B)|^{1/2}\left(\int_{\rd}|\eta|^{-\beta q}|\widehat Sf(B\eta)|^q\rmd \eta\right)^{1/q}
\end{align}
and
\begin{align}\label{presharp1}
    \left(\int_{\rd}|\eta|^{-\beta q}|\widehat Sf(B\eta)|^q\rmd \eta\right)^{1/q}&=|\det(B)|^{-1/q}\left(\int_{\rd}|B^{-1}\xi|^{-\beta q}|\widehat Sf(\xi)|^q\rmd \xi\right)^{1/q}\\
    &\geq |\det(B)|^{-1/q}\sigma_{\max}(B^{-1})^{-\beta}\left(\int_{\rd}|\xi|^{-\beta q}|\widehat Sf(\xi)|^q\rmd \xi\right)^{1/q}\\
    &=|\det(B)|^{-1/q}\sigma_{\min}(B)^{\beta}\left(\int_{\rd}|\xi|^{-\beta q}|\widehat Sf(\xi)|^q\rmd \xi\right)^{1/q}.
\end{align}
In synthesis,
\begin{equation}\label{ineq1-18082026}
     \eqalign{
     &\displaystyle\left(\int_{\rd}|\xi|^{-\beta q}|\widehat Sf(\xi)|^q\rmd \xi\right)^{1/q} \cr
     &\qquad\qquad\qquad\displaystyle\leq K(p,q,\alpha,\beta,n)|\det(B)|^{\frac 1q-\frac12}\sigma_{\min}(B)^{-\beta} \left(\int_{\rd}|x|^{\alpha p}|f(x)|^p\rmd x\right)^{1/p}.
     }
\end{equation}

\medskip
\noindent
{\bf Case 3 ($B \neq O$ not invertible).} This case covers two items of the Pitt-admissibility definition (\eqref{Pitt_admiss_iii} and \eqref{Pitt_admiss_iv}) but a distinction between them is not made, rather encompassed in the work done for Pitt's inequality for the partial Fourier transform. 
In Step 3.1 we show the necessity claim, whereas in Step 3.2 we show sufficiency. In both cases, we will argue for $p<\infty$. When $p=\infty$, the same proof applies with the corresponding integrals replaced by essential suprema.

\medskip
\noindent 
{\it Step 3.1. Proof of necessity of Pitt-admissibility.} In a few words, the strategy for this step is to reduce matters to the case $\widehat{S}f=\mathcal{F}_{x}f$, where $\mathcal{F}_{x}$ is the partial Fourier transform on a certain vector subspace of $\mathbb{R}^{n}$. For $y\in \ker(B)$ define
\begin{equation*}
    H(Vu,y):=g_{Ay}(u)=f(Vu+D^{\top}Ay)e^{i\pi(V^\top B^+AVu\cdot u-2V^\top C^\top Ay\cdot u)}, \qquad u\in\bR^r.
\end{equation*}
This way, 
\begin{equation}\label{idt1-17082026}
    |H(Vu,y)|=|f(Vu+D^{\top}Ay)|.
\end{equation}
Moreover the partial Fourier transform along the direction of $\ker(B)^\perp$ satisfies
\begin{equation*}    
\displaystyle\mathcal{F}_{x}H(Vz,y)=\int_{\ker(B)^{\perp}}H(x,y)e^{-2\pi iVz\cdot x}\mathrm{d}x=\int_{\mathbb{R}^{r}}H(Vu,y)e^{-2\pi iVz\cdot Vu}\mathrm{d}u=\widehat{g_{Ay}}(z), 
\end{equation*}
since $V$ is orthogonal. This way, by Lemma \ref{lemmaDecomp},
\begin{equation*}
   |\widehat{S}f(BVz+Ay)|=\mu_{S}|\mathcal{F}_{x}H(Vz,y)|.
\end{equation*}
By 
\begin{equation}\label{CV-flambda-output-inverse}
    \int_{\rd}\phi(\xi)\rmd\xi=\mu_S^{-2}\int_{\ker(B)}\int_{\bR^r}\phi(BV\eta+Ay_2)\rmd\eta\rmd y_2,
\end{equation}
so that
\begin{equation}\label{ineq1-17082026}
    \eqalign{   \displaystyle\int_{\mathbb{R}^{n}}|\xi|^{-\beta q}|\widehat{S}f(\xi)|^{q} \rmd \xi&\displaystyle=\mu_{S}^{-2}\int_{\ker(B)}\int_{\mathbb{R}^{r}}|BVz+Ay|^{-\beta q}|\widehat{S}f(BVz+Ay)|^{q} \rmd z\rmd y \cr
    &\displaystyle=\mu_{S}^{q-2}\int_{\ker(B)}\int_{\mathbb{R}^{r}}|BVz+Ay|^{-\beta q}|\mathcal{F}_{x}H(Vz,y)|^{q} \rmd z\mathrm{d}y. \cr    
    }
\end{equation}
Since $Vz\in\ker(B)^\perp$ and $y\in\ker(B)$,
\[
    BVz+Ay=(B+AP)(Vz+y).
\]
Therefore,
\begin{equation}\label{ineq2-17082026}
    |BVz+Ay|\leq\|B+AP\||Vz+y|.
\end{equation}
Using \eqref{ineq2-17082026} in \eqref{ineq1-17082026},
\begin{equation}\label{ineq3-17082026}
    \eqalign{   \displaystyle\int_{\mathbb{R}^{n}}|\xi|^{-\beta q}|\widehat{S}f(\xi)|^{q} \rmd \xi
    &\displaystyle\gtrsim\int_{\ker(B)}\int_{\mathbb{R}^{r}}|Vz+y|^{-\beta q}|\mathcal{F}_{x}H(Vz,y)|^{q} \rmd z\mathrm{d}y \cr
    &\displaystyle=\displaystyle\int_{\ker(B)}\int_{\ker(B)^{\perp}}|\eta+y|^{-\beta q}|\mathcal{F}_{x}H(\eta,y)|^{q} \rmd \eta\mathrm{d}y. \cr
    }
\end{equation}
On the other hand, by \eqref{Edo1} and \eqref{idt1-17082026}
\begin{equation}\label{ineq4-17082026}
    \eqalign{   \displaystyle\int_{\mathbb{R}^{n}}|x|^{\alpha p}|f(x)|^{p} \rmd x&\displaystyle=\int_{\ker(B)}\int_{\ker(B)^{\perp}}|z+D^{\top}Ay|^{\alpha p}|f(z+D^{\top}Ay)|^{p} \rmd z\mathrm{d}y \cr    &\displaystyle=\int_{\ker(B)}\int_{\ker(B)^{\perp}}|z+D^{\top}Ay|^{\alpha p}|H(z,y)|^{p} \rmd z\mathrm{d}y \cr
    }
\end{equation}
Moreover, for $z\in\ker(B)^\perp$ and $y\in\ker(B)$, by \eqref{symplRel3}
\[
    z+D^\top Ay=(I+B^\top CP)(z+y),
\]
and hence
\begin{equation}\label{ineq9-17082026}
    |z+D^\top Ay|\leq\|I+B^\top CP\||z+y|,
\end{equation}
which combined with \eqref{ineq4-17082026} gives
\begin{equation}\label{ineq5-17082026}
    \eqalign{   \displaystyle\int_{\mathbb{R}^{n}}|x|^{\alpha p}|f(x)|^{p} \rmd x  &\displaystyle\lesssim_{S}\int_{\ker(B)}\int_{\ker(B)^{\perp}}|z+y|^{\alpha p}|H(z,y)|^{p} \rmd z\mathrm{d}y. \cr
    }
\end{equation}
We conclude that, if \eqref{Pittgeneral} holds for some $p,q,\alpha$ and $\beta$, then \eqref{ineq3-17082026} and \eqref{ineq5-17082026} imply
\begin{equation}\label{ineq6-17082026}
    \left(\int_{\ker(B)}\int_{\ker(B)^{\perp}}|\eta+y|^{-\beta q}|\mathcal{F}_{x}H(\eta,y)|^{q} \rmd \eta\mathrm{d}y\right)^{\frac{1}{q}}\lesssim_{S}\left(\int_{\ker(B)}\int_{\ker(B)^{\perp}}|z+y|^{\alpha p}|H(z,y)|^{p} \rmd z\mathrm{d}y\right)^{\frac{1}{p}},
\end{equation}
which has the exact form of Pitt's inequality for the partial Fourier transform $\mathcal{F}_{x}$ applied to $H$. To complete this step, we claim that the necessary conditions for \eqref{eq:PFT-Pitt-necessary-assumption} to hold are also necessary for \eqref{Pittgeneral}. This will follow by showing that, if \eqref{Pittgeneral} holds, then \eqref{ineq6-17082026} holds for arbitrary $H\in\mathcal{S}(\ker(B)^{\perp}\times \ker(B))$. To that end, given $H\in\mathcal{S}(\ker(B)^{\perp}\times \ker(B))$ define $f$ by
\begin{equation}\label{deffproof}
    f(Vu+D^{\top}Ay):=H(Vu,y)e^{-i\pi(V^\top B^+AVu\cdot u-2V^\top C^\top Ay\cdot u)}.
\end{equation}
Since
\begin{equation}\label{intermediateStep}
    Vu+D^\top Ay=(I+B^\top CP)(Vu+y), 
\end{equation}
and, by Proposition~\ref{propProjections} $(iii)$, the matrix $I+B^\top CP$ is invertible, equation \eqref{deffproof} defines a unique function $f\in\mathcal{S}(\mathbb{R}^{n})$.

Running the argument above for $f$ yields \eqref{ineq6-17082026} for the initially given $H$. We conclude that the Pitt-admissibility of $(p, q, \alpha, \beta, r, n)$ is necessary for \eqref{Pittgeneral} to hold.

\medskip
\noindent
 {\it Step 3.2. Sufficiency of Pitt-admissibility.}
 {
 We now turn to the proof of sufficiency. We shall use Proposition~\ref{propProjections} $(iii)$-$(iv)$ for the inverse changes of coordinates introduced in Step 3.1.
 For $f\in\mathcal S(\rd)$, let $H$ be the auxiliary function introduced in {\it Step 3.1}. As established there, $H\in\mathcal S(\ker(B)^\perp\times\ker(B))$.
 } 
{By \eqref{ineq1-17082026} and Proposition~\ref{propProjections} $(iv)$,}
\begin{equation}\label{ineq7-17082026}
    \eqalign{   &\displaystyle\int_{\ker(B)}\int_{\ker(B)^{\perp}}|\eta+y|^{-\beta q}|\mathcal{F}_{x}H(\eta,y)|^{q} \rmd \eta\mathrm{d}y \cr
    &\displaystyle\qquad\qquad\qquad=\int_{\ker(B)}\int_{\ker(B)^{\perp}}\left|{(B+AP)^{-1}}(B\eta+Ay)\right|^{-\beta q}|\mathcal{F}_{x}H(\eta,y)|^{q} \rmd \eta\mathrm{d}y \cr
    &\qquad\qquad\qquad\displaystyle\geq\left\|{(B+AP)^{-1}}\right\|^{-\beta q}\int_{\ker(B)}\int_{\mathbb{R}^{r}}|BVz+Ay|^{-\beta q}|\mathcal{F}_{x}H(Vz,y)|^{q} \rmd z\mathrm{d}y \cr
    &\qquad\qquad\qquad=\displaystyle\frac{1}{\mu_{S}^{q-2}}\left\|{(B+AP)^{-1}}\right\|^{-\beta q}\int_{\mathbb{R}^{n}}|\xi|^{-\beta q}|\widehat{S}f(\xi)|^{q} \rmd \xi.
 }
\end{equation}
Similarly, by \eqref{ineq4-17082026} and {Proposition~\ref{propProjections} $(iii)$},
\begin{equation}\label{ineq8-17082026}
\begin{split}
    \displaystyle\int_{\ker(B)}\int_{\ker(B)^{\perp}}|z+y|^{\alpha p}|H(z,y)|^{p} \rmd z\mathrm{d}y&=\int_{\ker(B)}\int_{\ker(B)^{\perp}}\left|\left(
    {I-B^\top CP}
    \right)(z+D^{\top}Ay)\right|^{\alpha p}|H(z,y)|^{p} \rmd z\mathrm{d}y \cr    
    &\leq\|{I-B^\top CP}\|^{\alpha p} \int_{\mathbb{R}^{n}}|x|^{\alpha p}|f(x)|^{p}\mathrm{d}x. 
\end{split}
\end{equation}
Let $K^{\ast}(p,q,\alpha,\beta,r,n)$ be the constant $K^{\ast}$ in the right-hand side of \eqref{eq:PFT-Pitt-necessary-assumption}. Since we are assuming Pitt-admissibility, combining \eqref{ineq7-17082026}, \eqref{ineq8-17082026} and \eqref{eq:PFT-Pitt-necessary-assumption} yields
\begin{equation}
\displaystyle\left(\int_{\mathbb{R}^{n}}|\xi|^{-\beta q}|\widehat{S}f(\xi)|^{q} \rmd \xi\right)^{\frac{1}{q}} \leq K_{S}(p,q,\alpha,\beta,r,n)\left(\int_{\mathbb{R}^{n}}|x|^{\alpha p}|f(x)|^{p}\mathrm{d}x\right)^{\frac{1}{p}},\end{equation}
where
\begin{equation}\label{defKS2}
    K_{S}(p,q,\alpha,\beta,r,n):=\mu_{S}^{1-\frac{2}{q}}\|{I-B^\top CP}\|^{\alpha}\left\|{(B+AP)^{-1}}\right\|^{\beta} K^{\ast}(p,q,\alpha,\beta,r,n).
\end{equation}
Thus these conditions are sufficient for \eqref{Pittgeneral} to hold.

\begin{remark} \label{rem:constants}
Equations \eqref{ineq1-18082026} and \eqref{defKS2} give {admissible constants} for Theorem \ref{intro.thm:mainmetap} for $r=n$ and $1 \leq r < n$, {respectively and, in general, we do not discuss the sharpness of these constants. Nevertheless, for the case $p=q$, the directional Pitt's inequality discussed in the next section gives an additional radial bound, see Remark \ref{rem:Step344} below. 
Separately, Remark~\ref{remarkmatrices} proves the optimality of the directional squared-integral constant $K_{r,\alpha}$ when $p=q=2$, and no optimality claim is made for the resulting general radial bound.}
\end{remark}

\subsection{Directional Pitt's inequality}
The above proof of the case $1\leq \mathrm{rank}(B)<n$ did not distinguish between items \eqref{Pitt_admiss_iii} and \eqref{Pitt_admiss_iv} in the Pitt-admissibility definition. In this subsection we state and prove a new Pitt-type inequality that emphasises, at the level of the weight present in the integrand, the distinction between the effective and singular directions. In the introduction, we discussed that this refined directional version acquires its own importance, as it entails the logarithmic uncertainty principle for metaplectic operators, see Proposition \ref{intro.logUP}.
Moreover, as anticipated in Remark \ref{rem:constants} above, {for $p=q$, this procedure also gives the bound in Remark~\ref{rem:Step344} below. Finally, when $p=q=2$, the optimal norm constant in \eqref{PittForMetaps3} is $K_{r,\alpha}^{1/2}$, as follows from Remark~\ref{remarkmatrices}.}

    \begin{proposition}\label{Cor-Pitt-03082026}
        Let $\widehat S\in\Mp(n,\bR)$ with $ 1\leq r=\mathrm{rank}(B)<n$. Then,
        \begin{equation}\label{PittForMetaps3}
            \left(\int_{\rd}|{B^+{B(B+AP)^{-1}}}\xi|^{-\beta q}|\widehat Sf(\xi)|^q\rmd\xi\right)^{1/q}\leq \widetilde{K_S}(p,q,\alpha,\beta,r)\left(\int_{\rd}|{(I-D^\top AP)x|^{\alpha p}|f(x)|^p\rmd x}\right)^{1/p},
    \end{equation}
    holds for $f\in\mathcal{S}(\rd)$ if and only if the conditions \eqref{Pitt_admiss_iii} of Definition \ref{intro.defPittAdmissible} are met.
    In which case,
    \begin{equation}\label{wfeds}
        \widetilde{K_S}(p,q,\alpha,\beta,r,n)=K(p,q,\alpha,\beta,r){|\det(B+AP)|}^{\frac 1q-\frac 12 }.
    \end{equation}
    \end{proposition}
        Observe that if $r=n$ then the effective directions coincide with $\rd$, moreover $B^+ = B^{-1}$. The analogue of \eqref{PittForMetaps3} when $B\in\mathrm{GL}(n,\bR)$ follows directly from \eqref{presharp0} and \eqref{presharp1}, that is
        \begin{equation}\label{PittForMetaps4}
        \left(\int_{\rd}|{B^{-1}}\xi|^{-\beta q}|\widehat Sf(\xi)|^q\rmd\xi\right)^{1/q}\leq |\det(B)|^{\frac 1q-\frac 12}K(p,q,\alpha,\beta,n)\left(\int_{\rd}|x|^{\alpha p}|f(x)|^p\rmd x\right)^{1/p},
    \end{equation}
    which holds if and only if conditions \eqref{Pitt_admiss_ii} of Definition \ref{intro.defPittAdmissible} are met. This is consistent with \eqref{PittForMetaps3}, as in this case {$P=O$}. 
\begin{proof}
    {\bf Step 1 (necessity).} To prove the necessity of conditions \eqref{Pitt_admiss_iii} from Definition \ref{intro.defPittAdmissible}, let $f\in\cS(\rd)$ be in the form 
    \begin{equation}\label{defdVuDTAy}
        f(Vu+D^\top Ay)=h_1(Vu)h_2(D^\top Ay)e^{-i\pi (V^\top B^+AVu\cdot u-2V^\top C^\top Ay\cdot u)}, \qquad u\in\bR^r, y\in\ker(B),
    \end{equation}
    where $h_1\in\cS(\ker(B)^\perp)$ and $h_2\in\cS(D^\top A(\ker(B)))$. The corresponding auxiliary function $g_{Ay}$ of Lemma \ref{lemmaDecomp} tensorises as $g_{Ay}(u)=h_1(Vu)h_2(D^\top Ay)$ and consequently \eqref{Sfghat} reads as
    \begin{equation}
        |\widehat Sf(\xi_1+Ay)|=\mu_S|\widehat g_{Ay}(V^\top B^+\xi_1)|.
    \end{equation}
    {
    By substituting \eqref{defdVuDTAy} into \eqref{PittForMetaps3}, and using \eqref{Edo1} and \eqref{eq:output-coordinates}, \eqref{PittForMetaps3} becomes
    \[
        \begin{split}
        &\mu_S^{1-\frac2q}\bigl\||\cdot|^{-\beta}
       \widehat{h_1\circ V}\bigr\|_{L^q(\bR^r)}\|h_2\circ D^\top A\|_{L^q(\ker(B))}\leq
       \widetilde K_S\bigl\||\cdot|^\alpha(h_1\circ V)\bigr\|_{L^p(\bR^r)}\|h_2\circ D^\top A\|_{L^p(\ker(B))}.
        \end{split}
    \]
    }
    Choose $h_2\circ D^\top A=e^{-\pi|\cdot|^2}$, so that both its $L^p(\ker(B))$ and $L^q(\ker(B))$ norms are finite and positive. We thereby obtain for $\widetilde h_1=h_1\circ V$,
    \begin{equation}
        \||\cdot|^{-\beta}\widehat{\widetilde h_1}\|_{L^q(\mathbb{R}^r)}\lesssim\frac{\|h_2\circ D^\top A\|_{L^p(\ker(B))}}{\|h_2\circ D^\top A\|_{L^q(\ker(B))}}\||\cdot|^\alpha \widetilde h_1\|_{L^p(\mathbb{R}^r)}.
    \end{equation}
    The mapping $V:\bR^r\to\ker(B)^\perp$ is an isomorphism, so $h_1\in\cS(\ker(B)^\perp)\mapsto \widetilde h_1=h_1\circ V\in\cS(\bR^r)$ is also an automorphism. Consequently, by Theorem \ref{thm:PittpqFT}, the conditions \eqref{assumptionPitt} in dimension $r$ are necessary for \eqref{PittForMetaps3} to hold. On the other hand, the same computation above shows that
    \begin{equation}\label{refsd}
         \|h_2\circ D^\top A\|_{L^q(\ker(B))}\lesssim \frac{\||\cdot|^\alpha (h_1\circ V)\|_{L^p(\mathbb{R}^r)}}{\||\cdot|^{-\beta}\widehat{h_1\circ V}\|_{L^q(\mathbb{R}^r)}}\|h_2\circ D^\top A\|_{L^p(\ker(B))}.
    \end{equation}
    Let $W:\bR^{n-r}\to D^\top A(\ker(B))$ be a linear parametrisation of $D^\top A(\ker(B))$. Then, equation \eqref{refsd} reads for $\widetilde h_2(v)=h_2\circ W(v)$,
    \begin{equation}\label{3rewsdzd}
        \|\widetilde h_2\|_{q}\lesssim \frac{\||\cdot|^\alpha (h_1\circ V)\|_{L^p(\mathbb{R}^r)}}{\||\cdot|^{-\beta}\widehat{h_1\circ V}\|_{L^q(\mathbb{R}^r)}}\|\widetilde h_2\|_{p}.
    \end{equation}
    Choose $h_1$ so that the ratio in \eqref{3rewsdzd} is positive and finite. Since the mapping $h_2\in\cS(D^\top A(\ker(B)))\mapsto \widetilde h_2\in\cS(\bR^{n-r})$ is an isomorphism, applying \eqref{3rewsdzd} to the dilations $(\widetilde h_2)_\lambda(y)=\widetilde h_2(\lambda y)$ and the assumption $r<n$ yields $p=q$. This proves the necessity of conditions \eqref{Pitt_admiss_iii} of Definition \ref{intro.defPittAdmissible}.
    
    \medskip
    \noindent
    {\bf Step 2 (sufficiency).}
    We begin by applying \eqref{eqn:pitt2} to $g_{\xi_2}\in\mathcal{S}(\bR^r)$ under the extra assumption $p=q$:
\begin{equation}\label{ThesisForg}
    \int_{\bR^r}|\eta|^{-\beta p}|\widehat g_{\xi_2}(\eta)|^p\rmd \eta\leq K(p,p,\alpha,\beta,r)^p\int_{\bR^r}|u|^{\alpha p}|g_{\xi_{2}}(u)|^p\rmd u,
\end{equation}
{We use the same changes of variables as above. The substitution $\xi_2=Av$, followed by \eqref{Edo1}, gives
\begin{equation}\label{geds}
    \int_{A(\ker(B))}\int_{\bR^r}|u|^{\alpha p}|g_{\xi_2}(u)|^p\rmd u\rmd\xi_2=q_{\ker(B)}(A)\int_{\rd}\left|\mathcal P_{\ker(B)^\perp}^{D^\top A(\ker(B))}x\right|^{\alpha p}|f(x)|^p\rmd x.
\end{equation}
Similarly, Lemma~\ref{lemmaDecomp} and
\eqref{eq:output-coordinates} yield
\begin{equation}\label{presharp2}
    \int_{A(\ker(B))}\int_{\bR^r}|\eta|^{-\beta p}|\widehat{g_{\xi_2}}(\eta)|^p\rmd\eta\rmd\xi_2=q_{\ker(B)}(A)\mu_S^{2-p}\int_{\rd}\left|B^+\mathcal P_{R(B)}^{A(\ker(B))}\xi\right|^{-\beta p}|\widehat Sf(\xi)|^p\rmd\xi.
\end{equation}
By Proposition~\ref{propProjections}, 
        \[
        \begin{split}
            &\mathcal P_{\ker(B)^\perp}^{D^\top A(\ker(B))}=I-P-B^\top CP=I-D^\top AP,\\
            &B^+\mathcal P_{R(B)}^{A(\ker(B))}=(I-P)(B+AP)^{-1},
        \end{split}
        \]
        where the second identity follows by left-multiplication by $I-P$ in Proposition~\ref{propProjections} $(iv)$.
The inequality \eqref{PittForMetaps3} with
\[
    \widetilde K_S=\mu_S^{1-\frac2p}K(p,p,\alpha,\beta,r)=|\det(B+AP)|^{-\frac 12+\frac 1p}K(p,p,\alpha,\beta,r).
\]
follows then by integrating \eqref{ThesisForg} in $\xi_2$, using \eqref{geds} and \eqref{presharp2}, and taking the $p$-th root.}

\end{proof}
\begin{remark}\label{rem:Step344}
A consequence of \eqref{PittForMetaps3} is an additional radial Pitt's inequality of the form \eqref{Pittgeneral}, with the constant appearing in Remark \ref{rem:constants}.
    This merely makes use of the observation that the term $|B^+{\mathcal{P}_{R(B)}^{A(\ker(B))}\xi}|$ at the left-hand side of \eqref{PittForMetaps3} can be estimated from above as follows:
    \begin{align}\label{estgg1}
        |B^+{B(B+AP)^{-1}}\xi|\leq\Vert B^{+}{B(B+AP)^{-1}}\Vert|\xi|,
    \end{align}
    where the spectral norm is implied, and similarly
    \begin{equation}\label{estgg2}
        |(I-D^\top AP)x|\leq \Vert I-D^\top AP\Vert |x|.
    \end{equation}
    Specifically, we obtain \eqref{Pittgeneral} with
    {\begin{equation}\label{eqn53}
        K_S(p,q,\alpha,\beta,r,n)=|\det(B+AP)|^{-\frac 12+\frac1q}\|I-D^\top AP\|^\alpha\|(I-P)(B+AP)^{-1}\|^\beta K(p,q,\alpha,\beta,r).
    \end{equation}}
    where $K(p,q,\alpha,\beta,r)$ is defined as in Theorem \ref{thm:PittpqFT}.
    For $r=n$, {we have $P=O$ and \eqref{eqn53} coincides with \eqref{ineq1-18082026}.}
    {
    When $p=q=2$ and $\alpha=\beta$, the optimal constant is $K(2,2,\alpha,\alpha,r)=K_{r,\alpha}^{1/2}$, with $r=n$ in the full-rank case.
    }
\end{remark}

\begin{remark} 
    In the directional setting in the case of $1\leq r<n$, the singular directions force $p=q$. Indeed, after separating the effective and singular variables, there remains an unweighted $L^p$-$L^q$ estimate in the singular directions, which by scaling can hold only if $p=q$. 
    Conversely, in the radial inequality the weights also act on the singular directions, allowing the additional regime $q<p$ described in \eqref{Pitt_admiss_iv}. This highlights the different role played by the singular directions in the directional and radial settings from the perspective of Pitt's inequality. 
\end{remark}
    
\begin{remark}The directional inequalities \eqref{PittForMetaps3} and \eqref{PittForMetaps4} may be viewed as versions of \eqref{Pittgeneral} in which the effective directions in the domain of $f$ are correctly paired with those in the domain of $\widehat{S}f$. This pairing is evident from the projections appearing in the radial weights of both the inequalities.
\end{remark}

{
\begin{remark}\label{remarkmatrices}
Recall that the constant $K(2,2,\alpha,\alpha,r)=K_{r,\alpha}^{1/2}$ is sharp for the Fourier transform.
On the other hand, inequality \eqref{PittForMetaps3} for $p=q=2$ gives \eqref{PittwithmatricesBinv}. This is precisely Proposition~\ref{intro.thm:Pittpq2}.
We now show that the constant is sharp.
Let $(h_1^{(k)})_k\subset\mathcal{S}(\bR^r)$ be a sequence of functions so that
\begin{equation}
    \lim_{k\to\infty}\frac{\int_{\bR^r}|\eta|^{-2\alpha}|\widehat h_1^{(k)}(\eta)|^2\rmd\eta}{\int_{\bR^r}|u|^{2\alpha}|h_1^{(k)}(u)|^2\rmd u}=K_{r,\alpha}.
\end{equation}
Choose $h_2$ such that $h_2(D^\top\cdot)\in \mathcal{S}(A(\ker(B)))\setminus\{0\}$.
By \eqref{symplDecom1} and Proposition~\ref{PropIsomorphisms} $(iii)$, every $x\in\rd$ is uniquely written as $x=Vu+D^\top\xi_2$, with $u\in\bR^r$ and $\xi_2\in A(\ker(B))$.
Thus, setting
\[
    f^{(k)}(Vu+D^\top\xi_2):=h_1^{(k)}(u)h_2(D^\top\xi_2)e^{-i\pi(V^\top B^+AVu\cdot u-2V^\top C^\top\xi_2\cdot u)}
\]
defines $f^{(k)}\in\mathcal{S}(\rd)$. By \eqref{defgxi2}, we have
\[
    g_{\xi_2}^{(k)}(u)=h_1^{(k)}(u)h_2(D^\top\xi_2).
\]
Applying \eqref{geds} and \eqref{presharp2} with $p=2$ and $\beta=\alpha$, both the geometric factor $q_{\ker(B)}(A)$ and the common factor depending on $h_2$ cancel. Consequently,
\begin{equation}
    \frac{\displaystyle\int_{\rd}|B^+B(B+AP)^{-1}\xi|^{-2\alpha}|\widehat Sf^{(k)}(\xi)|^2\rmd\xi}{\displaystyle\int_{\rd}|(I-D^\top AP)x|^{2\alpha}|f^{(k)}(x)|^2\rmd x}=
    \frac{\displaystyle\int_{\bR^r}|\eta|^{-2\alpha}|\widehat{h_1^{(k)}}(\eta)|^2\rmd\eta}{\displaystyle\int_{\bR^r}|u|^{2\alpha}|h_1^{(k)}(u)|^2\rmd u}\xrightarrow[k\to\infty]{}K_{r,\alpha}.
\end{equation}
Therefore no constant strictly smaller than $K_{r,\alpha}$ can replace this constant in \eqref{PittwithmatricesBinv}, which proves the sharpness.
It is worth stressing that the extremal ratio depends on the metaplectic operator only through $r=\mathrm{rank}(B)$: the singular directions contribute the same factor to both sides and disappear from the sharp constant.
\end{remark}
}

\subsection{Metaplectic estimates on homogeneous Sobolev spaces}\label{sec:MEHSS}
    The behaviour of metaplectic operators within Lebesgue spaces is rather rigid. By the classification in \cite{Giacchi}, non-trivial $L^p$-$L^q$ boundedness away from the unitary $L^2$ case occurs precisely at the two extreme possibilities for the upper-right block in \eqref{blockS}. If $B=O$, then $\widehat S$ is an isomorphism of $L^p(\rd)$ for every $0<p\leq\infty$. If $B\in\mathrm{GL}(n,\bR)$, then $\widehat S$ satisfies the
metaplectic Hausdorff-Young estimate
\begin{equation}\label{sob.HYmetaplectic}
    \|\widehat Sf\|_{L^{p'}}
    \leq
    |\det B|^{{\frac12-\frac1p}}
    \left(\frac{p^{1/p}}{(p')^{1/p'}}\right)^{n/2}
    \|f\|_{L^p},
    \qquad 1\leq p\leq2.
\end{equation}
When $0<\mathrm{rank}(B)<n$, no such boundedness between Lebesgue spaces is available apart from the unitary $L^2$-$L^2$ action. Thus the Lebesgue scale only sees metaplectic operators whose projection is either lower block triangular or free. 

\begin{remark}
    It is easy to see that if $1\leq \mathrm{rank}(B)<n$, then $\widehat S:L^p(\rd)\to L^p_{\mathrm{loc}}(\rd)$ for every $1\leq p\leq2$. This follows from \eqref{Sfghat}, after suitable changes of variables and observing that $L^{p'}_{\mathrm{loc}},L^\infty\hookrightarrow L^p_{\mathrm{loc}}$ when $1\leq p\leq2$.
\end{remark}

As discussed in the introduction, modulation spaces provide a more flexible framework to settle this theory, and boundedness results for metaplectic operators are usually stated in their context. 
{Our Pitt inequalities yield explicit bounds for the operator norms of metaplectic operators within $\dot H^\alpha$ spaces.}

The estimates below are first proved for functions whose Fourier transforms are smooth and compactly supported away from the origin, and then are extended by a density argument. If $C\neq O$, the exponents in the Pitt estimate satisfy $0\leq\alpha<\operatorname{rank}(C)/2\leq n/2$, so both \(\dot H^\alpha\) and \(\dot H^{-\alpha}\) lie within this canonical distributional realisation.

The block governing homogeneous Sobolev estimates is $C$, rather than $B$.
Indeed, conjugation by the Fourier transform replaces $S$ with
\begin{equation}\label{sob.Sflat}
    S^\flat:=JSJ^{-1}
    =\begin{pmatrix}
        D&-C\\
        -B&A
      \end{pmatrix}.
\end{equation}
Up to a phase, $\mathcal{F}\widehat S\mathcal{F}^{-1}=\widehat{S^\flat}$, and hence
\begin{equation}\label{sob.Fconjugation}
    \big|\mathcal{F}(\widehat Sf)(\xi)\big|
    =\big|\widehat{S^\flat}\widehat f(\xi)\big|.
\end{equation}

\begin{theorem}[Norm estimates on homogeneous Sobolev spaces]\label{thm:metaplecticSobolev}
Let $\widehat S\in\Mp(n,\bR)$.
\begin{enumerate}[(i)]
    \item If $C=O$, then $\widehat S$ extends to an isomorphism of $\dot H^\alpha(\rd)$ for every $-\infty<\alpha<n/2$. More precisely,
    \begin{equation}\label{sob.CzeroIdentity}
        \|\widehat Sf\|_{\dot H^\alpha}^2
        =\int_{\rd}|D\xi|^{2\alpha}|\widehat f(\xi)|^2\rmd\xi,
    \end{equation}
    and
    \begin{equation}\label{sob.CzeroNorm}
        \|\widehat S\|_{\dot H^\alpha\to\dot H^\alpha}
        {=}\max\big\{\sigma_{\min}(D)^\alpha,\sigma_{\max}(D)^\alpha\big\}.
    \end{equation}
    {In particular, if furthermore $D=I$, then $ \|\widehat Sf\|_{\dot H^\alpha}
        =\|f\|_{\dot H^\alpha}$. }
    
    \item If $C\neq O$, set $r=\mathrm{rank}(C)$. For every
    $0\leq\alpha<r/2$, the operator $\widehat S$ extends boundedly from
    $\dot H^\alpha(\rd)$ to $\dot H^{-\alpha}(\rd)$, and
    \begin{equation}\label{sob.PittSobolev}
        \|\widehat Sf\|_{\dot H^{-\alpha}}
        \leq
        K_{r,\alpha}^{1/2}
        \Vert{I-A^\top DP_C}\Vert^\alpha
        \Vert {C^+C(C-DP_C)^{-1}}\Vert^\alpha
        \|f\|_{\dot H^\alpha},
    \end{equation}
    {where $P_C=I-C^+C$ is the orthogonal projection onto $\ker(C)$.}
\end{enumerate}
\end{theorem}

\begin{proof}
{Proposition~\ref{propProjections} applied to
\[
S^\flat=\begin{pmatrix}
    D&-C\\
    -B&A
    \end{pmatrix},
\]
implies that $C-DP_C$ is invertible and that
\[
\begin{aligned}
    &\mathcal P_{\ker(C)^\perp}^{A^\top D(\ker(C))}=I-A^\top DP_C=C^+C-C^\top BP_C,\\
    &C^+\mathcal P_{R(C)}^{D(\ker(C))}=C^+C(C-DP_C)^{-1}.
\end{aligned}
\]
Then, the assertion follows by} Theorem~\ref{intro.thm:mainmetap} and Remark \ref{rem:Step344} applied to $\widehat S^\flat\widehat f=\mathcal F\widehat Sf$ with $p=q=2$. 
{
Moreover, by \eqref{sob.CzeroIdentity} and the standard formula for the norm of multiplier operators on $L^2$, see for example, \cite[Theorem~II.1.5]{Conway}, we obtain
\[
    \|\widehat S\|_{\dot H^\alpha\to\dot H^\alpha}=\mathrm{esssup}_{\xi\neq0}\left(\frac{|D\xi|}{|\xi|}\right)^\alpha=\max\{\sigma_{\min}(D)^\alpha,\sigma_{\max}(D)^\alpha\}.
\]
}
We also comment that \eqref{sob.CzeroIdentity} arises from  \eqref{eqn:B=O_intermediate_step}, this is an intermediate step in the proof and is specified for clarity, and if $D=I$ it immediately implies that $\widehat S$ is an isometry of $\dot H^\alpha(\rd)$.
\end{proof}

\begin{remark}\label{remInhomogeneous}
    {Let us comment on inhomogeneous Sobolev spaces, defined for $s\in\bR$ by the finiteness of the norm
    \[
        \|f\|_{H^s}=\left(\int_{\rd}(1+|\xi|^2)^s|\widehat f(\xi)|^2\rmd\xi\right)^{1/2}.
    \]
    For every $\widehat S\in \Mp(n,\bR)$ and every $s\geq0$, unitarity gives
    \[
        \|\widehat Sf\|_{H^{-s}}\leq\|\widehat Sf\|_2=\|f\|_2\leq\|f\|_{H^s}.
    \]
    Therefore, $\widehat S$ is bounded from $H^s$ to $H^{-s}$ for every $s>0$, without further assumptions on the rank of the block $C$, and it does not require Pitt's inequality.
    Moreover, if $C=O$, the same change of variables as in \eqref{sob.CzeroIdentity} gives
    \[
    \|\widehat Sf\|_{H^s}^2=\int_{\rd}(1+|D\xi|^2)^s|\widehat f(\xi)|^2\rmd\xi,
    \]
    so $\widehat S$ is an isomorphism of $H^s(\rd)$ for every $s\in\bR$.} 
    
    {In contrast, for $0<s<n/2$, neither $\dot H^s\hookrightarrow L^2$ nor $L^2\hookrightarrow\dot H^{-s}$ holds in general.
    The homogeneous estimate in Theorem~\ref{thm:metaplecticSobolev} therefore supplies information at the frequency origin that is not contained in the preceding consequence of unitarity.}
\end{remark}

\section{Applications to Schr\"odinger equations}\label{sec:SchrodingerApplications}
We collect here the applications of the entropic uncertainty principle (Proposition \ref{thm:main1} or Theorem \ref{thm:main2}) and of the boundedness on homogeneous Sobolev spaces (Theorem \ref{thm:metaplecticSobolev}) to the quadratic Schr\"odinger evolutions considered in this paper. 
For each example we present two propositions: the first providing the corresponding entropic uncertainty relation and the second recording the boundedness of the propagator within homogeneous Sobolev spaces.
Explicitly, the entropic uncertainty principles are contained in Propositions \ref{prop:freeParticleEntropy}, \ref{prop:harmonicEntropy}, \ref{prop:magneticEntropy} and \ref{prop:anisotropicEntropy}, and they follow by applying Proposition \ref{thm:main1} (for the first three listed propositions) or Theorem \ref{thm:main2} (in the final instance) to the corresponding equations. Moreover, the boundedness results for the propagators are contained in Propositions \ref{prop:freeParticleSobolev}, \ref{prop:harmonicSobolev}, \ref{prop:magneticSobolev} and \ref{prop:anisotropicSobolev}, and they are all direct consequences of Theorem \ref{thm:metaplecticSobolev}. {We limit ourselves here to including the details for the anisotropic oscillator, to identify explicitly its marginal entropy and the projection factors in the Sobolev estimate.}

\subsection{The free particle}\label{subsec:freeParticle}
The first prototypical example of Schr\"odinger evolution is the free particle equation, where the Hamiltonian of \eqref{Schro} is ${a^\mathrm{w}(x,\mathrm{D})}=-\frac{1}{8\pi^2}\Delta$. The Hamiltonian flow can be computed explicitly, see for example \cite[Chapter 15]{DeGosson} with $\hbar=1/2\pi$,
\begin{equation}\label{defStfp}
    S_t=\begin{pmatrix}
        I & tI\\
        O & I
    \end{pmatrix}.
\end{equation}
The corresponding propagator is 
\begin{equation}\label{fpp}
    e^{-2\pi it{a^\mathrm{w}(x,\mathrm{D})}}u_0(x)=\mathcal{F}^{-1}(\Phi_{-tI}\widehat{u_0})(x)=\int_{\rd}e^{-i\pi t|\xi|^2}\widehat{u_0}(\xi)e^{2\pi i\xi\cdot x}\mathrm{d}\xi,
\end{equation}
where we recall that $\Phi_{-tI}(y)=e^{-i\pi t|y|^2}$.

\begin{proposition}\label{prop:freeParticleEntropy}
Let $u_0\in\mathcal{S}(\rd)$ satisfy $\|u_0\|_2=1$. For every $t\neq0$,
\begin{equation}
H[|u_0|^2]+H[|u(t,\cdot)|^2]\geq \frac{n}{2}\left(1-\ln\Big(\frac{2}{|t|}\Big)\right).
\end{equation}
\end{proposition}

\begin{proposition}\label{prop:freeParticleSobolev}
For every $t\in\bR$ and every $0\leq\alpha<n/2$, 
\begin{equation}
\big\|e^{-2\pi it{a^\mathrm{w}(x,\mathrm{D})}}\big\|_{\dot H^\alpha\to\dot H^\alpha}= 1.
\end{equation}
\end{proposition}

\subsection{The quantum harmonic oscillator}\label{subsec:harmonicOscillator}
The quantum harmonic oscillator corresponds to \eqref{Schro} with the choice ${a^\mathrm{w}(x,\mathrm{D})}=-\frac{1}{8\pi^2}\Delta+\frac{1}{2}|x|^2$. The Hamiltonian flow is the rotation matrix
\begin{equation}\label{HamflowH}
    S_t=\begin{pmatrix}
        \cos(t)I & \sin(t)I\\
        -\sin(t)I & \cos(t)I
    \end{pmatrix},
\end{equation}
and the corresponding evolution is expressed, up to a phase, in terms of the {\it fractional Fourier transform}
\begin{equation}\label{propHarmonOsc2}
    e^{-2\pi it{a^\mathrm{w}(x,\mathrm{D})}}u_0(x)=\mathcal{F}_tu_0(x)=|\sin(t)|^{-n/2}e^{i\pi \cot(t)|x|^2}\int_{\rd}u_0(y)e^{i\pi \cot(t)|y|^2}e^{-2\pi ix\cdot y/\sin(t)}\mathrm{d}y,
\end{equation}
for $t\neq k\pi$, $k\in\mathbb{Z}$. For $t=k\pi$, the propagator is, up to the metaplectic phase, either the identity or the flip operator $u_0\mapsto u_0(-\cdot)$.

\begin{proposition}\label{prop:harmonicEntropy}
Let $u_0\in\mathcal{S}(\rd)$ satisfy $\|u_0\|_2=1$. If $\sin(t)\neq0$, then
\begin{equation}
H[|u_0|^2]+H[|u(t,\cdot)|^2] \geq \frac{n}{2}\left(1-\ln\Big(\frac 2{|\sin(t)|}\Big)\right).
\end{equation}
\end{proposition}

\begin{proposition}\label{prop:harmonicSobolev}
If $\sin(t)=0$, then for every $\alpha<n/2$,
\begin{equation}
\big\|e^{-2\pi it{a^\mathrm{w}(x,\mathrm{D})}}\big\|_{\dot H^\alpha\to\dot H^\alpha}= 1.
\end{equation}
Otherwise, for every $0\leq\alpha<n/2$,
\begin{equation}
\big\|e^{-2\pi it{a^\mathrm{w}(x,\mathrm{D})}}\big\|_{\dot H^\alpha\to\dot H^{-\alpha}}\leq K_{n,\alpha}^{1/2}|\sin(t)|^{-\alpha}.
\end{equation}
\end{proposition}

\subsection{The uniform magnetic potential}\label{subsec:UMP}
We next consider the Schr\"odinger equation with uniform magnetic potential, corresponding to
\begin{equation}\label{HamiltonUMP}
    {a^\mathrm{w}(x,\mathrm{D})}=-\frac{1}{2m}\left(\frac{1}{2\pi}\nabla-im\omega Mx\right)^2, \quad m>0,\; \omega\neq0, \; M\in\bR^{n\times n}\setminus\{0\}, \; M^\top=-M,\; M^\top M=I.
\end{equation}
Such a matrix $M$ exists if and only if $n$ is even. An example in this case is $M=J$. 
Associated with \eqref{HamiltonUMP} is the quadratic form $q(z)=\frac1{2m}|\xi-m\omega Mx|^2$, with Hamilton matrix
\begin{equation}
    X=\begin{pmatrix}
        -\omega M & I/m\\
        -m\omega^2 I & -\omega M
    \end{pmatrix}, \quad X^{2k+1}=(-1)^k(2\omega)^{2k}X, \quad X^{2k}=(-1)^{k-1}(2\omega)^{2(k-1)}X^2, \quad k\geq1,
\end{equation}
see for example \cite{Knutsen}. The corresponding Hamiltonian flow is
\begin{equation}\label{StUMP}\begin{split}
    S_t&=e^{tX}\\
    &=\frac 1 2 \begin{pmatrix}
            \big(1+\cos(2\omega t)\big)I-\sin(2\omega t)M & \frac{1}{m\omega}\big((\cos(2\omega t)-1)M+\sin(2\omega t)I\big)\\
            -m\omega\big((\cos(2\omega t)-1)M+\sin(2\omega t)I\big) & \big(1+\cos(2\omega t)\big)I-\sin(2\omega t)M
        \end{pmatrix}\\
    &=\begin{pmatrix}
           \cos(\omega t)\big(\cos(\omega t)I-\sin(\omega t)M\big) & \frac{\sin(\omega t)}{m\omega}\big(-\sin(\omega t)M+\cos(\omega t)I\big)\\
           m\omega\sin(\omega t)\big(\sin(\omega t)M-\cos(\omega t)I\big) & \cos(\omega t)\big(\cos(\omega t)I-\sin(\omega t)M\big)
        \end{pmatrix}\\
    &=:\begin{pmatrix}
        \mathcal{A}(t)& \mathcal{B}(t) \\
        \mathcal{C}(t)&\mathcal{D}(t)
        \end{pmatrix}.
\end{split}\end{equation}
A straightforward computation using the defining properties of $M$ gives
\begin{equation}\label{BTBminus1}
    \mathcal{B}(t)^\top \mathcal{B}(t)=\frac{\sin^2(\omega t)}{m^2\omega^2}I,
\end{equation}
so $\mathcal B(t)$ is invertible precisely when $t\neq k\pi/\omega$ for $k \in \mathbb{Z}$. For these times, \eqref{integralSffree} gives
\begin{equation}
    u(t,x)=\Big|\frac{m\omega}{\sin(\omega t)}\Big|^{n/2}e^{ i\pi m\omega\cot(\omega t)|x|^2 }\int_{\bR^n}u_0(y)e^{i \pi m\omega\cot(\omega t)|y|^2 }e^{-2\pi i m\omega\big(M+\cot(\omega t)I\big) x\cdot y}\mathrm{d}y.
\end{equation}

\begin{proposition}\label{prop:magneticEntropy}
Let $u_0\in\mathcal{S}(\rd)$ satisfy $\|u_0\|_2=1$. If $\sin(\omega t)\neq0$, then
\begin{equation}
    H[|u_0|^2]+H[|u(t,\cdot)|^2]\geq \frac{n}{2}\Big(1-\ln\Big(2\Big|\frac{m\omega}{\sin(\omega t)}\Big|\Big)\Big).
\end{equation}
\end{proposition}

\begin{proposition}\label{prop:magneticSobolev}
If $\sin(\omega t)=0$, then for every $\alpha<n/2$,
\begin{equation}
\big\|e^{-2\pi it{a^\mathrm{w}(x,\mathrm{D})}}\big\|_{\dot H^\alpha\to\dot H^\alpha}= 1.
\end{equation}
Otherwise, for every $0\leq\alpha<n/2$,
\begin{equation}
    \big\|e^{-2\pi it{a^\mathrm{w}(x,\mathrm{D})}}\big\|_{\dot H^\alpha\to\dot H^{-\alpha}}\leq K_{n,\alpha}^{1/2}|m\omega\sin(\omega t)|^{-\alpha}.
\end{equation}
\end{proposition}

\subsection{The anisotropic harmonic oscillator}\label{subsubsec:AHO}
In the examples above, the top-right block of the Hamiltonian flow is either $O$ or invertible. We now consider a case in which it can be singular and non-zero. For simplicity, we work in dimension $n=2$ and consider \eqref{Schro} with Hamiltonian
\begin{equation}
    {a^\mathrm{w}(x,\mathrm{D})}=-\frac{1}{8\pi^2}\partial_{x_2}^2+\frac{1}{2}x_2^2.
\end{equation}
Thus the harmonic oscillator acts only on the second variable of $u_0=u_0(x_1,x_2)$. The Hamiltonian flow is the $4\times4$ one-parameter subgroup of $\Sp(2,\bR)$
\begin{equation}\label{StAHO}
    S_t=\left(\begin{array}{cc|cc}
        1 & 0 & 0 & 0\\
        0 & \cos(t) & 0 & \sin(t)\\
        \hline
        0 & 0 & 1 & 0\\
        0 & -\sin(t) & 0 & \cos(t)
    \end{array}\right)=:\begin{pmatrix}
    \mathcal{A}(t)& \mathcal{B}(t) \\
    \mathcal{C}(t)&\mathcal{D}(t)
    \end{pmatrix},
\end{equation}
see for example \cite{CGM2025}. For $t\neq k\pi$, $k\in\mathbb Z$,
\begin{equation}
    \ker(\mathcal{B}(t))^\perp=R(\mathcal{B}(t))=\{(0,x_2):x_2\in\bR\},
\end{equation}
and the propagator is the one-parameter subgroup of fractional partial Fourier transforms
\begin{equation}\label{propHarmonosc}
    e^{-2\pi it{a^\mathrm{w}(x,\mathrm{D})}}u_0(x)=\widehat{S}_tu_0(x)=|\sin(t)|^{-1/2}e^{i\pi \cot(t)x_2^2}\int_{\bR}u_0(x_1,y_2)e^{i\pi \cot(t)y_2^2}e^{-2\pi ix_2\cdot y_2/\sin(t)}\mathrm{d}y_2.
\end{equation}

\begin{proposition}\label{prop:anisotropicEntropy}
Let $u_0\in\mathcal{S}(\bR^2)$ satisfy $\|u_0\|_2=1$. If $\sin(t)\neq0$, then
\begin{equation}\label{AnisotHO}\begin{split}
H[|u_0|^2]+H[|u(t,\cdot)|^2]-H_{S_t}[u_0]\geq\frac 12 \left(1-\ln\Big({\frac{2}{|\sin(t)|}}\Big)\right).
\end{split}\end{equation}
\end{proposition}
{
\begin{proof}
For $\sin t\neq0$,
\[
    \ker(\mathcal{B}(t))=\bR\times\{0\},\qquad\mathrm{rank}(\mathcal B(t))=1.
\]
Thus Definition~\ref{intro.def.corrective.term},
in the standard coordinate on $\ker(\mathcal{B}(t))$, gives
\[
    H_{S_t}[u_0]=h(m_1),\qquad m_1(x_1)=\int_{\bR}|u_0(x_1,x_2)|^2\rmd x_2.
\]
In particular, this corrective term is independent of $t$ whenever $\sin t\neq0$.
Since $\mu_{S_t}=|\sin t|^{-1/2}$, the conclusion follows from Theorem~\ref{thm:main2}.
\end{proof}
}

\begin{proposition}\label{prop:anisotropicSobolev}
If $\sin(t)=0$, then for every $\alpha<1$:
\begin{equation}
\big\|e^{-2\pi it{a^\mathrm{w}(x,\mathrm{D})}}\big\|_{\dot H^\alpha\to\dot H^\alpha}= 1.
\end{equation}
Otherwise, for every $0\leq\alpha<1/2$,
\begin{equation}
\big\|e^{-2\pi it{a^\mathrm{w}(x,\mathrm{D})}}\big\|_{\dot H^\alpha\to\dot H^{-\alpha}}\leq K_{1,\alpha}^{1/2}|\sin(t)|^{-\alpha}.
\end{equation}
\end{proposition}
\begin{proof}
The block $\mathcal C(t)$ in \eqref{StAHO} vanishes precisely when $\sin(t)=0$. In that case $\mathcal D(t)=\operatorname{diag}(1,(-1)^k)$. Even when $\mathcal D(t)\neq I$, the identity in \eqref{sob.CzeroIdentity} of Theorem~\ref{thm:metaplecticSobolev} $(i)$ still gives the isometry property. 
If $\sin(t)\neq0$, then $\operatorname{rank}(\mathcal C(t))=1$ and its unique non-zero singular value is $|\sin(t)|$. 
Moreover, in this case 
{\begin{equation}
    \|I-\mathcal A(t)^\top\mathcal D(t)P_{\mathcal C(t)}\|=1, \qquad \|\mathcal C(t)^+\mathcal C(t)(\mathcal C(t)-\mathcal D(t)P_{\mathcal C(t)})^{-1}\|=|\sin(t)|^{-1}.
\end{equation}}
Plugging this information into \eqref{sob.PittSobolev} concludes the proof.
\end{proof}

\section*{Acknowledgments}

We thank Jonathan Bennett for many stimulating conversations and for
pointing out various references in the literature at several stages of
this work. The third author further thanks her PhD supervisor, Jonathan
Bennett, for his support and guidance, and for introducing her to the
topics that sparked this collaboration.
We also thank Hans Feichtinger for his suggestions on an earlier version 
of the manuscript.

\section*{Statements and Declarations}
\noindent
\textbf{Funding.}
The first author is supported by the Swiss National Science Foundation
(SNSF) Starting Grant ``Multiresolution methods for unstructured data''
(TMSGI2 211684). The second author is supported by his EPSRC
Postdoctoral Fellowship UKRI3285, \textit{New perspectives in
phase-space Analysis and Fourier restriction}. The third author is
supported by the EPSRC Doctoral Training Partnership and the University
of Birmingham.

\medskip
\noindent
\textbf{Competing interests.}
The authors have no relevant financial or non-financial interests to
disclose.

\medskip
\noindent
\textbf{Data availability.}
No datasets were generated or analysed in this study.

\medskip
\noindent
\textbf{Use of generative AI.}
Claude and ChatGPT (GPT-5.6 Sol) were used to check earlier versions of
the manuscript for typographical and logical issues, to assist with
figures and literature searches, and to suggest mathematical approaches.
In particular, ChatGPT (GPT-5.6 Sol):
\begin{itemize}
    \item supported clarification of the probabilistic interpretation of Theorem~\ref{thm:main2} in Remark~\ref{rem:conditional-entropy};
    
    \item drew our attention to the translation--modulation mechanism in
    \cite{DGT} and suggested the wave-packet construction $P_T$ in
    \eqref{PT-construction-14082026}, used in the proof of
    Theorem~\ref{prop:PFT-necessary};

    \item suggested the dyadic superposition mechanism used to rule out
    endpoint estimates in Theorem~\ref{prop:PFT-necessary};

    \item suggested a unification of the arguments in
    Subsection~\ref{sufficiency-13082026}, which was subsequently
    rewritten and refined by the authors.
\end{itemize}
All AI-assisted mathematical suggestions and references were
independently checked and verified by the authors, who take full
responsibility for the content of the manuscript. Apart from the
specific contributions described above, the results and mathematical
ideas are due to the authors.

\begin{appendix}\label{appendix}

   \section{Integration over subspaces and direct sums}\label{appendixA}
    In this section, we provide the main tools for integration of functions defined on $\rd$ over subspaces of $\rd$, and over direct sums. We will use the notation in \cite{TMO}. 

\begin{lemma}
    Under the notation above, if $\varphi\in \mathcal{S}(\mathbb{R}^n)$,
    \begin{equation}\label{CV}
\int_\mathcal{L} \varphi(Mx) \mathrm{d}x = \frac{1}{q_\mathcal{L}(M)} \int_{M(\mathcal{L})} \varphi(x) \mathrm{d}x, \quad (\dim M(\mathcal{L}) = \dim \mathcal{L}).
\end{equation}
\end{lemma}

We next discuss changes of variables adapted to direct sum decompositions.
\begin{lemma}\label{Lemma-COVar}
    Let $\mathcal{L}\subseteq\rd$ be a subspace of $\rd$ of dimension $r>0$. Let $E=(E_1|E_2)$, be such that the columns of $E_1$ form an orthonormal basis of $\mathcal{L}^\perp$ and those of $E_2$ form an orthonormal basis of $\mathcal{L}$. Assume that $\rd=\mathcal{L}^\perp\oplus M(\mathcal{L})$. Then, the modulus of the determinant of the Jacobian of 
    \[\Psi:(u,v)\in\mathbb{R}^{n-r}\times \mathbb{R}^{r}\longmapsto E_1u+ME_2v\in\mathcal{L}^\perp\oplus M(\mathcal{L})
    \]
    is $|\det\nabla_{(u,v)}\Psi|=|\det(E_2^\top ME_2)|$.
\end{lemma}
\begin{proof}
    Since
    \begin{equation}
        \Psi(u,v)=E_1u+ME_2v=(E_1|ME_2)(u,v)^\top,
    \end{equation}
    the matrix of $\Psi$ is $N=(E_1|ME_2)$.
    Therefore, using that $E=(E_1|E_2)$ is an orthogonal linear parametrisation of $\rd$,
    \begin{align}
        |\det\nabla_{(u,v)}\Psi|= |\det(N)|=|\det(E^\top N)|=\left|\det{
        \begin{pmatrix}
            I_{n-r} & E_1^\top ME_2\\
            O_{r\times (n-r)} & E_2^\top ME_2
        \end{pmatrix}
        }\right|=|\det(E_2^\top ME_2)|.
    \end{align}
    This concludes the proof.
    
\end{proof}

We now apply Lemma \ref{Lemma-COVar} to the two concrete cases treated in this paper. 
\begin{enumerate}[(1)]
\item We first consider integrals in the form
\begin{align}
    \int_{\ker(B)}\int_{\ker(B)^\perp}\phi(x_1+D^\top Ax_2)\rmd x_1\rmd x_2,
\end{align}
with suitable integrability assumptions on $\phi$.
First, we parametrise the integral by means of the linear parametrisations $E_1:\mathbb{R}^r\to \ker(B)^\perp$ and $E_2:\mathbb{R}^{n-r}\to \ker(B)$,
\begin{equation}
    \int_{\ker(B)}\int_{\ker(B)^\perp}\phi(x_1+D^\top Ax_2)\rmd x_1\rmd x_2=\int_{\rd}\phi(E_1u+D^\top A{E_2}v)\rmd u\rmd v,
\end{equation}
where we have used \eqref{CV} together with the orthogonality of $E_1$ and $E_2$.
The decomposition given by \eqref{symplDecom1} and the mapping properties of Proposition \ref{PropIsomorphisms} $(iii)$ guarantee that the assumptions of Lemma \ref{Lemma-COVar}, with $\mathcal{L}=\ker(B)$ and $M=D^\top A$ are met.
With the notation therein, the columns of $E_2$ parametrize $\ker(B)$, so that $BE_2=O$ and therefore
\begin{equation}
     E_2^\top D^\top AE_2=E_2^\top(I+B^\top C)E_2=I_{n-r}+(BE_2)^\top CE_2=I_{n-r},
\end{equation}
thereby implying that the determinant of the change of variables is $1$, so that
\begin{equation}
   \int_{\ker(B)}\int_{\ker(B)^\perp}\phi(x_1+D^\top Ax_2)\rmd x_1\rmd x_2=\int_{\bR^{n-r}}\int_{\bR^{r}}\phi(E_1u+D^\top AE_2v)\rmd u\rmd v.
\end{equation}
In conclusion,
\begin{equation}\label{Edo1}
    \int_{\ker(B)}\int_{\ker(B)^\perp}\phi(x_1+D^\top Ax_2)\rmd x_1\rmd x_2=\int_{\bR^r}\int_{\bR^{n-r}}\phi(u,v)\rmd u\rmd v=\int_{\rd}\phi(x)\rmd x.
\end{equation}
\item Then, we shall consider integrals in the form
\begin{align}
    \int_{R(B)}\int_{R(B)^\perp}\phi(x_1+AA^\top \xi_2)\rmd x_1\rmd \xi_2.
\end{align}
Again, we choose two orthogonal parametrisations $E_1:\mathbb{R}^r\to R(B)$ and $E_2:\mathbb{R}^{n-r}\to R(B)^\perp$ and use \eqref{CV} to get:
\begin{align}
    \int_{R(B)^\perp}\int_{R(B)}\phi(x_1+AA^\top \xi_2)\rmd x_1\rmd \xi_2=\int_{\rd}\phi(E_1u+AA^\top E_2v)\rmd u\rmd v.
\end{align}
We then apply Lemma \ref{Lemma-COVar} with $\mathcal{L}=R(B)^\perp$ and $M=AA^\top$, whose assumptions are verified in view of Proposition \ref{PropIsomorphisms} $(iv)$ and the consequent decomposition formula \eqref{symplDecom3}, to change variables:
\begin{align}
    \int_{R(B)^\perp}\int_{R(B)}\phi(x_1+AA^\top \xi_2)\rmd x_1\rmd \xi_2&=\int_{\rd}\phi(E_1u+AA^\top E_2v)\rmd u\rmd v\\
    &=|\det(E_2^\top AA^\top E_2)|^{-1}\int_{\rd}\phi(x)\rmd x.
\end{align}
Since $E_2$ is an orthogonal parametrisation of $R(B)^\perp$, by \eqref{qL} we conclude:
\begin{equation}\label{Edo2}
    \int_{R(B)^\perp}\int_{R(B)}\phi(x_1+AA^\top \xi_2)\rmd x_1\rmd \xi_2=q_{R(B)^\perp}(A^\top)^{-2}\int_{\rd}\phi(x)\rmd x.
\end{equation}
\end{enumerate}

{
Let $S\in\Sp(n,\bR)$ have blocks \eqref{blockS}. Let $V:\bR^r\to\ker(B)^\perp$ be an orthogonal parametrisation given in Lemma~\ref{lemmaDecomp}. For every non-negative measurable function $\varphi$, and for every integrable function $\varphi$, we have
\begin{equation}\label{eq:output-fibres}
    \int_{\ker(B)}\int_{R(B)}\varphi(\xi_1+Av)\rmd\xi_1\rmd v=q_{R(B)^\perp}(A^\top)^{-1}\int_{\rd}\varphi(\xi)\rmd\xi,
\end{equation}
where we used the substitution $v=A^\top w$, with
$w\in R(B)^\perp$, and \eqref{Edo2}, in the order.
By changing variables again by $\xi_1=BV\eta$,
whose Jacobian modulus is $\sigma(B)$, we obtain
\begin{equation}\label{eq:output-coordinates}
\int_{\ker(B)}\int_{\bR^r}\varphi(BV\eta+Av)\rmd\eta\rmd v=\mu_S^2\int_{\rd}\varphi(\xi)\rmd\xi.
\end{equation}
Thus the coordinate map $(\eta,v)\mapsto BV\eta+Av$ has Jacobian modulus $\mu_S^{-2}$.
}

\section{{Proof of Proposition~\ref{propProjections}}}\label{AppendixProj}
    Recall that $P=I-B^+B$. The symplectic relation $D^\top A-B^\top C=I$, together with $PB^\top=O$, implies
    \begin{equation}\label{PDAP}
        PD^\top AP=P,
    \end{equation}
    and similarly, the relation $D^\top B=B^\top D$ implies that 
    \begin{equation}\label{PBD0}
        PD^\top B=O.
    \end{equation}

    \noindent
    {\bf Invertibility of $B+AP$ and determinant formula.}
    Let $\mathcal E_1=\{e_1,\ldots,e_r\}$ and $\mathcal E_2=\{e_{r+1},\ldots,e_n\}$ be orthonormal bases of $\ker(B)^\perp$ and $\ker(B)$, respectively, that form an orthonormal basis of $\rd$. Similarly, let $\widetilde{\mathcal E_1}=\{\widetilde e_1,\ldots,\widetilde e_r\}$ and $\widetilde{\mathcal E_2}=\{\widetilde e_{r+1},\ldots,\widetilde e_n\}$ be orthonormal bases of $R(B)$ and $R(B)^\perp$, respectively, defining an orthonormal basis of $\rd$. Let $v=\sum_{j=1}^n\alpha_je_j\in\rd$. The coordinates of $(B+AP)v$ with respect to the basis $\widetilde{\mathcal{E}_1}\cup \widetilde{\mathcal{E}_2}$ of $\rd$ are given by
    \begin{equation}
    \begin{split}
        (B+AP)v\cdot \widetilde e_l&=\sum_{j=1}^r\alpha_jBe_j\cdot\widetilde e_l+\sum_{j=r+1}^n\alpha_jAe_j\cdot\widetilde e_l\\
        &=\begin{cases}
            \sum_{j=1}^r\alpha_jBe_j\cdot\widetilde e_l+\sum_{j=r+1}^n\alpha_jAe_j\cdot\widetilde e_l, & \text{if $l\leq r$},\\
            \sum_{j=r+1}^n\alpha_jAe_j\cdot\widetilde e_l, & \text{otherwise}.
        \end{cases}
        \end{split}
    \end{equation}
    So, $B+AP$ is orthogonally equivalent to 
    \begin{equation}
        M=\begin{pmatrix}
            M_{11} & M_{12}\\
            O & M_{22}
        \end{pmatrix},
    \end{equation}
    where 
    \begin{equation}
        M_{11}=(Be_j\cdot\widetilde e_l)_{1\leq j,l\leq r}, \qquad M_{12}=(Ae_j\cdot\widetilde e_l)_{r<j\leq n,\, 1\leq l\leq r}, \qquad M_{22}=(Ae_j\cdot\widetilde e_l)_{r+1\leq j,l\leq n}
    \end{equation}
    (here $l$ is the row index and $j$ is the column index).
    Consequently,
    \begin{equation}
        |\det(B+AP)|=|\det(M_{11})\det(M_{22})|.
    \end{equation}
    Since $M_{11}$ represents $B:\ker(B)^\perp\to R(B)$ with respect to orthonormal bases of the two spaces, $|\det(M_{11})|=\sigma(B)$. 
    To compute the second factor, let
    \[
        E_2=(e_{r+1}\mid\cdots\mid e_n),\qquad \widetilde E_2=(\widetilde e_{r+1}\mid\cdots\mid\widetilde e_n).
    \]
    By Proposition~\ref{PropIsomorphisms}(ii),
    $A^\top(R(B)^\perp)=\ker(B)$, and hence
    \[
        A^\top\widetilde E_2=E_2M_{22}^\top.
    \]
    Using \eqref{qL} and
    $E_2^\top E_2=I_{n-r}$, we obtain
    \[
    \begin{aligned}
    q_{R(B)^\perp}(A^\top)=\sqrt{\det(\widetilde E_2^\top AA^\top\widetilde E_2)}=\sqrt{\det(M_{22}E_2^\top E_2M_{22}^\top)}=\sqrt{\det(M_{22}M_{22}^\top)}=|\det M_{22}|.
    \end{aligned}
    \]

    This implies both that $B+AP$ is invertible and that
    \begin{equation}
        |\det(B+AP)|=\sigma(B)q_{R(B)^\perp}(A^\top).
    \end{equation}
    
    \noindent
    {\bf Proof of $(i)$.}
    If $x\in\rd$ and $x=x_1+x_2$ with $x_1=Px\in\ker(B)$ and $x_2\in\ker(B)^\perp$,
    \begin{equation}
        x=x_1+x_2=\underbrace{D^\top Ax_1}_{\in D^\top A(\ker(B))}+\underbrace{(x_2-B^\top Cx_1)}_{\in\ker(B)^\perp}.
    \end{equation}
    By \eqref{symplDecom1}, the sum $\rd=\ker(B)^\perp\oplus D^\top A(\ker(B))$ is direct, so 
    \begin{equation}
        \mathcal P^{\ker(B)^\perp}_{D^\top A(\ker(B))}x=D^\top Ax_1=D^\top APx,
    \end{equation}
    by definition of the oblique projection $\mathcal P^{\ker(B)^\perp}_{D^\top A(\ker(B))}$.\\

    \noindent
    {\bf Proof of $(ii)$.} By \eqref{PDAP} and \eqref{PBD0}, it follows $PD^\top (B+AP)=P$, so that
    \begin{equation}
        P(B+AP)^{-1}=PD^\top.
    \end{equation}
    Therefore,
    \begin{equation}
        B(B+AP)^{-1}=(B+AP-AP)(B+AP)^{-1}=I-AP(B+AP)^{-1}=I-APD^\top.
    \end{equation}
    Let $y=y_1+y_2\in\rd$, with $y_1=B\eta\in R(B)$ and $y_2=Az\in A(\ker(B))$, $\eta\in\ker(B)^\perp$ and $z\in\ker(B)$. By \eqref{PBD0}
    \begin{equation}
        (I-APD^\top)y_1=y_1-APD^\top B\eta=y_1,
    \end{equation}
    while by \eqref{symplRel3}
    \begin{equation}
        (I-APD^\top)y_2=y_2-APD^\top Az=y_2-AP(I+B^\top C)z=y_2-APz=y_2-y_2=0.
    \end{equation}
    Therefore,
    \begin{equation}
        \mathcal P^{A(\ker(B))}_{R(B)}y=y_1=(I-APD^\top)y_1=(I-APD^\top)y=B(B+AP)^{-1}y.
    \end{equation}

    \noindent
    {\bf Proof of $(iii)$.} 
    By \eqref{symplDecom1} and $(i)$,
    \begin{equation}
        \mathcal P^{D^\top A(\ker(B))}_{\ker(B)^\perp}=I-\mathcal{P}^{\ker(B)^\perp}_{D^\top A(\ker(B))}=I-D^\top AP.
    \end{equation}
    Consequently,
    \begin{equation}
        \mathcal P^{D^\top A(\ker(B))}_{\ker(B)^\perp}+P=P+I-D^\top AP=I-B^\top CP.
    \end{equation}
    Moreover, $PB^\top=0$ implies
    \[
        (B^\top CP)^2=B^\top C(PB^\top)CP=0.
    \]
    Hence
    \[
        (I-B^\top CP)(I+B^\top CP)=I,
    \]
    which proves the inverse formula, and concludes the proof of $(iii)$.

    \noindent
    {\bf Proof of $(iv)$.} We have:
    \begin{align}
        B^+\mathcal P_{R(B)}^{A(\ker(B))}+PD^\top = B^+B(B+AP)^{-1}+PD^\top.
    \end{align}
    Therefore, right multiplication by $B+AP$, \eqref{PBD0}, \eqref{PDAP}and the definition of $P$, yield
    \begin{align}
       (B^+\mathcal P_{R(B)}^{A(\ker(B))}+PD^\top)(B+AP)&= (B^+B(B+AP)^{-1}+PD^\top)(B+AP)\\
       &=B^+B+PD^\top B+PD^\top AP=I.
    \end{align}
    This concludes the proof.

\end{appendix}

\bibliographystyle{abbrv}
\bibliography{bibliography}

\end{document}